\documentclass[oneside,a4paper,english, 12pt]{amsart}

\usepackage[margin=2cm]{geometry}
\usepackage{setspace}

\usepackage{stmaryrd}
\usepackage{amsmath}
\usepackage{amsfonts}
\usepackage{enumerate}
\usepackage{amssymb}
\usepackage{amscd}
\usepackage{amsthm}
\usepackage{tikz-cd}

\usepackage{xcolor}
\definecolor{deeppurple}{RGB}{75,0,130}
\definecolor{darkred}{RGB}{120,0,0}
\usepackage[colorlinks=true, linkcolor=blue, citecolor=darkred]{hyperref}
\usepackage{mathabx}
\usepackage[utf8]{inputenc}

 \theoremstyle{plain}

\newtheorem{thm}{Theorem}[subsection]
\theoremstyle{plain}
  \newtheorem{prop}[thm]{Proposition}
\theoremstyle{plain}
 \newtheorem{lemma}[thm]{Lemma}
\theoremstyle{plain}

\theoremstyle{plain}

\theoremstyle{definition}
  \newtheorem{defn}[thm]{Definition}
  \theoremstyle{definition}
  
\theoremstyle{definition}
  
 \theoremstyle{definition}

\theoremstyle{remark}
\newtheorem{rmk}[thm]{Remark}
\theoremstyle{definition}

\usepackage{theoremref}

\newcommand{\isoto}{\buildrel\sim\over\rightarrow}
\newcommand{\MM}{\mathfrak M}
\newcommand{\ML}{\mathfrak M}
\newcommand{\wM}{\widetilde {\mathfrak M}}
\newcommand{\wD}{\widetilde D}
\newcommand{\MMM}{\widehat{\mathfrak M}}
\newcommand{\X}{X_n^{[0,h],\tau}}
\newcommand{\ST}{\mathrm{ST}^{[0,h],\tau}}
\newcommand{\STT}{\mathcal{ST}^{[0,h],\tau}}
\newcommand{\XXX}{\mathcal X_n^{[0,h],\tau}}

\newcommand{\Y}{Y_n^{[0,h],\tau}}
\newcommand{\F}{\Phi{\mathrm {-Mod}}^{\text{\'et}}_n}
\newcommand{\aaa}{\mathbf a}
\newcommand{\uuu}{\llbracket u' \rrbracket}
\newcommand{\vvv}{\llbracket v \rrbracket}
\newcommand{\sss}{s_{\mathrm{or}}}
\newcommand{\ssj}{s_{\mathrm{or,} j'}}
\newcommand{\Ma}{\mathrm{Mat}_n}
\newcommand{\GL}{\mathrm{GL}}
\newcommand{\SSSS}{\mathfrak S_{L',R}}
\newcommand{\SSS}{\mathfrak S_R}
\newcommand{\A}{\mathbf{A}_{\mathrm{cris},R}}
\newcommand{\AC}{\mathbf{A}_{\mathrm{cris}}}
\newcommand{\Aw}{\mathbf{A}_{\mathrm{cris},\widetilde R}}

\newcommand{\AAA}{\mathbf A_{\mathrm{inf},R}}
\newcommand{\AI}{\mathbf{A}_{\mathrm {inf}}}
\newcommand{\WD}{\mathcal{WD}(L'/K)}
\newcommand{\W}{\mathrm{WD}(L'/K)}
\newcommand{\Fr}{\mathrm{Frob}_K}
\newcommand{\D}{\Lambda}
\newcommand{\R}{\mathrm{Rep}_R(G_K)}
\newcommand{\s}{\mathcal S_R[1/p]}

\newcommand{\B}{\mathbf{B}_{\mathrm{cris},R}^+}
\newcommand{\e}{\mathrm{exp}}
\newcommand{\Xe}{(\mathcal X^{[0,h],\tau}_{n})^{\mathrm{rig}}_\eta}
\newcommand{\XXXc}{\mathcal X_{n,\mathrm{cris}}^{[0,h],\tau}}
\newcommand{\Xec}{(\mathcal X^{[0,h],\tau}_{n,\mathrm{cris}})^{\mathrm{rig}}_\eta}

\title{On the Emerton-Gee stack of potentially semistable representations}
\author{Wang Shenrong}
\date{\today}

\begin{document}

\begin{abstract}    
      We first classify the potentially semistable Galois representations of a given tame inertial type using some variants of Breuil-Kisin modules. Then we construct some ``model" of such Breuil-Kisin modules. From here, we build a stack of our version of Breuil-Kisin modules, and show that it is isomorphic to the Emerton-Gee stack of potentially semistable representations. This also makes it possible to construct a map from the rigid generic fiber of the potentially semistable EG stack to the analytification of the stack of Weil-Deligne representations, which ``interpolates" Fontaine's recipe of associating Weil-Deligne representations to potentially semistable Galois representations. We prove the smoothness of this map, and obtain the generic formal smoothness property of the potentially semistable EG stack. This can be regarded as the global version of Kisin's regularity results on the corresponding Galois deformation rings.
\end{abstract}

\maketitle

\tableofcontents

\section{Introduction}

\subsection{Motivation and overview}

In the ground-breaking work \cite{EG}, Emerton and Gee constructed a moduli stack $\mathcal X_n$, parametrizing rank $n$ \'etale $(\varphi,\Gamma_K)$-modules. It may be thought of as the stack of Galois representations. More precisely, let $K$ be a $p$-adic local field, let $\mathcal O$ be the ring of integers in a finite extension $E$ of $\mathbb Q_p$. Fix an algebraic closure $\overline K$ of $K$ and let $G_K$ be the absolute Galois group $\mathrm{Gal}(\overline K/K)$. For any finite flat $\mathcal O$-algebra $R$, $\mathcal X_n(R)$ is equivalent to the groupoid of continuous representations $\rho:G_K \to \text{GL}_n(R)$. We will refer to the stack as the Emerton-Gee stack.

The Emerton-Gee stack is a globalization of Mazur's deformation rings that allows Galois representations to vary continuously. In particular, the versal rings to the Emerton-Gee stack are the corresponding deformation rings. So we can use the stack to study Galois representations in families and to formulate Langlands correspondence in the $p$-adic setting. \cite{EGH22} gives a lot of results and conjectures in this direction.

Within the Emerton-Gee stack, they further constructed a $p$-adic formal algebraic stack $\mathcal X_n^{[0,h],\tau}$ parametrizing $n$-dimensional potentially semistable Galois representations, i.e. $\mathcal X_n^{[0,h],\tau}(\mathcal O)$ is equivalent to the groupoid of continuous representations $\rho:G_K \to \text{GL}_n(\mathcal O)$ such that $\rho$ is potentially semistable of Hodge-Tate weights bounded by $0$ and $h$, and of inertial type $\tau$. However, the detailed construction of $\XXX$ is complicated. In this paper, we will construct the moduli of our potentially semistable Breuil-Kisin modules. It will be used to describe the stack $\mathcal X_n^{[0,h],\tau}$ in a more explicit way. In particular, we can describe any $R$-point of the stack, where $R$ is a $p$-adically complete, topologically of finite type, flat $\mathcal O$-algebra. In general, these points on the stack don't correspond to Galois representations.

We fix our field $K$ to be a finite extension of $\mathbb Q_p$ and let $\mathcal O_K$ be the ring of integers of $K$. We first modify the classification in \cite{Gao} to give a $p$-adic Hodge theoretic description of potentially semistable $G_K$-representations with given tame inertial types using our potentially semistable Breuil-Kisin modules (see \thref{BK} and \thref{BK1}). This leads to our first main result:

\begin{thm} \thlabel{1st}
    Let $R$ be a finite flat $\mathcal O$-algebra. The category of rank $n$ potentially semistable Breuil-Kisin modules of height at most $h$ and inertial type $\tau$ with coefficients in $R$ is equivalent to the category of potentially semistable representations $\rho: G_K \to \GL_n(R)$ of Hodge-Tate weights not exceeding $h$ and inertial type $\tau$.
\end{thm}

In the work \cite{LLHLM1} (and the subsequent \cite{LLHLM3} with possibly ramified base field $K$), a local model essentially made of linear algebraic objects is built to model the potentially crystalline locus of the Emerton-Gee stack. Inspired by this, we aim to ``extend" their model a little bit to not only include the locus of potentially crystalline representations, but also potentially semistable representations. In particular, using \thref{1st} and some more $p$-adic Hodge theory, we will show in \thref{stackyy} that for any finite flat $\mathcal O$-algebra $R$, some parahoric-torsor of $\XXX(R)$ can be described by $\STT(R)$ defined as follows, with every notation defined in the body of the paper.

\begin{defn} \thlabel{def}
Let $\STT$ be the functor from $\mathcal O$-algebras to groupoids, sending any $\mathcal O$-algebra $R$ to the groupoid of triples $((A^{(j)})_{j \in \mathcal J},N,\Xi)$ where
\begin{enumerate} 
    \item $A^{(j)}, (v+p)^h (A^{(j)})^{-1} \in \mathcal P_j(R)$ for all $j \in \mathcal J$ and $N \in \Ma(W(k') \otimes_{\mathbb Z_p} R)$;
    \item $\Xi \in \GL_n(\mathcal S_R[1/p])$ and $\Xi-1 \in u'\Ma(\mathcal S_R[1/p])$, where $\mathcal S_R$ is defined in \thref{Breuil}, satisfying $$\Xi^{-1} \e \left(t\otimes N \right) \psi(\Xi)\in \GL_n(\AAA),$$ and  $$\Xi \varphi(C)\varphi(\Xi)^{{-1}}= \overline C\in \Ma(W(k') \otimes_{\mathbb Z_p} R),$$ where $C:=\mathfrak C((A^{(j)} \omega^{(j)})_{j \in \mathcal J})$ and $\overline C:=C \mathrm{ \ mod \ } u'$; 
    \item $N \overline C =p \overline C \varphi(N)$;
    \item $\Psi^{}- \mathrm{id} \in u'\Ma(\AAA)$, where $\Psi:=\varphi^{-1}(\Xi^{-1} \e \left(t\otimes N \right) \psi(\Xi)) \in \GL_n(\AAA)$.
\end{enumerate}
\end{defn}

We comment here that the matrix $\Xi$ is in fact uniquely determined by the matrices $C$ and $N$ (see the end of \autoref{s4}). So in this way, our result is related to \cite[Theorem 1.2]{Yao}, which gives a fully faithful functor classifying semistable Galois representations: $$\mathrm{Rep}^{\mathrm{st}, \ge0}_{\mathbb Z_p}(G_K) \hookrightarrow \mathrm{Mod}_{\mathfrak S}^{\phi,N_{\mathrm {int}}}.$$ The matrix $N$ here is just their $N_{\mathrm{int}}$. It represents the monodromy operator on some lattice of the $(\phi, N)$-module. Thus, we can describe the essential image of this functor as those $(\phi, N_{\mathrm{int}})$-modules whose associated Galois action (constructed from the operator $N_{\text{int}}$) on the $\mathbf A_{\mathrm{cris}}$-module (see \thref{vip}) preserves the associated $\mathbf A_{\mathrm {inf}}$-module $\MMM$ .

As an application, the data of \thref{def} enables us to build a stack of potentially semistable Breuil-Kisin modules $\X$. Let $R$ be any $p$-adically complete, topologically of finite type, $\mathcal O$-flat algebra.\footnote{Finite flat $\mathcal O$-algebras belong to this category.} The stack $\X$ has the same $R$-points as some parahoric quotient of $\STT(R)$ by our construction. It serves as an analogue for $Y_n^{[0,h],\tau,\nabla_\infty}$ in \cite{LLHLM1}. And we obtain the following result:

\begin{thm}
Let $\tau$ be a tame inertial type. The stack $\XXX$ is isomorphic to the stack of potentially semistable Breuil-Kisin modules $\X$. 
\end{thm}

We will be using this result to describe the $p$-adically complete, topologically of finite type, $\mathcal O$-flat points of $\XXX$ in what follows.

Since the stack is related to deformation rings and it is useful to know whether the latter are smooth for modularity lifting reasons, we study the formal smoothness properties of $\XXX$. We note that in \cite{BG}, a stack parametrizing Weil-Deligne representations is constructed. The Weil-Deligne representations are especially useful when studying Galois representations since they make some information discrete, and thus simpler to describe. We recall the definition of the Weil-Deligne stack and its algebraic construction. They can be related to the potentially semistable locus of the Emerton-Gee stack using the functor introduced by Fontaine in \cite{Fon}. We generalize that functor into a map between the rigid generic fiber of the potentially semistable Emerton-Gee stack and the analytification of the Weil-Deligne stack in the following theorem:

\begin{thm}
    There is a morphism $WD: \Xe \to \WD^{\mathrm{an}}$ between the two stacks such that for any given finite type point $\mathcal M \in \Xe(\mathrm{Sp\ }A)$, the image $WD(\mathcal M)\in \WD^{\mathrm{an}}(\mathrm{Sp\ }A)$ coincides with Fontaine's recipe of associating Weil-Deligne representations to potentially semistable representations. 
\end{thm}

In the end, we will study the formal smoothness properties of the map $WD$ between the two stacks. We obtain the following theorem.

\begin{thm} 
The morphism $WD: \Xe \to \WD^{\mathrm{an}}$ is smooth.
\end{thm} 

We prove this essentially using \thref{good} by constructing a lift of a potentially semistable Breuil-Kisin module across a square-zero extension. This forms the technical heart of the whole paper.

Finally, we obtain the following property about $\XXX$ and the potentially crystalline stack $\Xec$: 

\begin{thm}
     The rigid generic fiber $\Xe$ of the stack $\XXX$ has a dense formally smooth locus. The stack $\Xec$ is formally smooth.
\end{thm}

Using this, we recover properties of the corresponding deformation rings.
\begin{thm}
    The potentially semistable deformation ring $R_{\bar \rho}^{[0,h], \tau}$ has a dense open subscheme which is regular. The generic fiber of the potentially crystalline deformation ring $R_{\bar \rho}^{[0,h],\tau, \mathrm{cr}}$ \footnote{Since we mainly deal with potentially semistable representations in this paper, the corresponding stacks and deformation rings are not decorated with a subscript ``st". When we mention potentially crystalline stacks and deformation rings, there will be crystalline subscripts or superscripts to distinguish.} is regular.
\end{thm}

\subsection{Outline of the paper}
In \autoref{s2}, we recall the definition of the $p$-adic Hodge theoretic rings that will be used in this paper, as well as the basics of the Emerton-Gee stack.

In \autoref{s3}, we give our version of potentially semistable Breuil-Kisin modules and relate them to potentially semistable Galois representations. We also describe the category of our Breuil-Kisin modules in a na\"ive way, as in \thref{big2}. 

In \autoref{s4}, using $p$-adic Hodge theory, we turn our na\"ive description of potentially semistable Breuil-Kisin modules into a moduli stack. We show in \thref{big1} that this stack is isomorphic to the Emerton-Gee stack of potentially semistable representations. 

In \autoref{s5}, we use this isomorphism of stacks to construct the functor relating the potentially semistable Emerton-Gee stack to the stack of Weil-Deligne representations. We study the smoothness properties of the map in our main result \thref{smooth}. We also deduce corresponding results for corresponding deformation rings.

\subsection{Acknowledgments}
This paper forms an essential part of the author's PhD thesis at Johns Hopkins University. The author is very grateful to his advisor David Savitt for his suggestions and support. The author thanks Bao V. Le Hung for his insight in relating the Emerton-Gee stack to the Weil-Deligne stack. He also thanks Brandon Levin for discussions. He is also very grateful to Hui Gao for the comments and improvements after the initial draft.

\subsection{Notations} \label{Notations}

Fix an odd prime $p$ and a positive integer $n$. Let $K$ be a finite extension of $\mathbb Q_p$ of degree $ef$, with $e$ the ramification index and $f$ the residue class degree. Let $\mathcal O_K$ be the ring of integers of $K$, and let $\pi$ be a uniformizer of $\mathcal O_K$. We fix an algebraic closure $\overline{K}$ of $K$. For any extension $K^*/K$ inside $\overline{K}$, let $G_{K^*}= \text{ Gal}(\overline{K}/K^*)$ be the absolute Galois group of $K^*$. Denote the maximal unramified extension of $K$ inside $\overline K$ as $K^{\text{ur}}$. Let $I_K=G_{K^{\text{ur}}}$ be the inertia group of $K$, and let $W_K$ be the Weil group of $K$. For any tower of unramified extensions $K_2/K_1/\mathbb Q_p$, let $\varphi_{K_2/K_1} \in \textrm{Gal}(K_2/K_1)$ denote the arithmetic Frobenius of $K_2$ over $K_1$. Let $\varphi=\varphi_{K^{\text{ur}}/\mathbb Q_p}$. Let $\Fr=\varphi_{K^{\text{ur}}/K}$ be the absolute Frobenius of $K$, and we use the same notation to denote its restriction to any unramified extension of $K$, and its lift to $W_K$, so we have $W_K=I_K\rtimes \langle \Fr \rangle$. Let $\nu:K^{\times} \to \mathbb Z$ be the valuation on $K$ such that $\nu(\pi)=1$. We also denote the pullback of $\nu$ to $W_K$ through the local reciprocity map $W_K \to K^{\times}$ by $\nu$, hence $\nu(\Fr)=1$. Let $k$ be the residue field of $K$, and we fix its algebraic closure $\overline k$. We write $W(k)$ for its ring of Witt vectors and let $K_0=W(k)[1/p]$ be the maximal unramified subextension of $K/\mathbb Q_p$.

Fix a sufficiently large finite extension $E$ of $\mathbb Q_p$. For now, $E$ should contain all embeddings of $K$ into its algebraic closure.  Let $\mathcal O$ be its ring of integers, and let $\varpi \in \mathcal O$ be a uniformizer. Denote the residue field of $E$ as $\mathbb F$. Fix an embedding $\sigma_0: K_0 \hookrightarrow E$ and let $\sigma_j=\sigma_0 \circ \varphi^{-j}$ for $j=0,1,\cdots f-1$.  Let $\mathcal J= \text{Hom}_{\mathbb Q_p}(K_0,E)$, then $\mathcal J=\{\sigma_0, \sigma_1, \cdots, \sigma_{f-1}\} = \mathbb Z/f\mathbb Z$, i.e. we can identify the elements of $\mathcal J$ as integers modulo $f$.

Let $r$ be a positive integer. Let $K'$ be the unique unramified extension of $K$ inside $\overline K$ of degree $r$. Let $k'$ be the residue field of $K'$ inside $\overline k$. Let $f'=fr$, then $k'\cong\mathbb F_{p^{f'}}$. We write $W(k')$ for its ring of Witt vectors and let $K_0'=W(k')[1/p]$ be the maximal unramified subextension of $K'/\mathbb Q_p$. Extend the embedding $\sigma_0$ to $\sigma_0: K_0' \hookrightarrow E$ and let $\sigma_{j'}=\sigma_0 \circ \varphi^{-j'}$ for $j'=0,1,\cdots f'-1$.  Let $\mathcal J'= \text{Hom}_{\mathbb Q_p}(K_0',E)$, then $\mathcal J'=\{\sigma_0, \sigma_1, \cdots, \sigma_{f'-1}\} = \mathbb Z/f'\mathbb Z$. 
 
Let $e'=p^{f'}-1$. Let $\pi'$ be an $e'$-th root of $\pi$, i.e. $(\pi')^{e'}=\pi$, and let $L'=K'(\pi')$. Then $L'$ is a totally tamely ramified extension of $K'$ (inside $\overline{K}$) of ramification degree $e'$. We adjust, if necessary, our $\Fr$ defined just now to the lift of $\Fr \in \text{Gal}(K'/K)$ to $\text{Gal}(L'/K)$ that fixes $\pi'$. 

Let $\Delta=\text{Gal}(L'/K)$ and let $\Delta'=\text{Gal}(L'/K')$. Then $\Delta=\Delta' \rtimes \langle \Fr \rangle \cong \mathbb Z/e'\mathbb Z \rtimes \mathbb Z/r\mathbb Z$. 

Define a character $\omega_{K'}: \Delta' \to W(k') ^\times$ of $\Delta'$ as follows: for any $g\in \Delta'$, let $$\omega_{K'}(g)=\frac{g(\pi')}{\pi'}.$$ Let $$\omega_{K', \sigma_{j'}}=\sigma_{j'} \circ \omega_{K'}: \Delta' \to \mathcal O ^\times$$ for $j'=0,1,\cdots, f'-1$. Set $\omega_{f'}=\omega_{K', \sigma_0}$. Since for any $g\in \Delta'$,$\frac{g(\pi')}{\pi'}$ is an $e'$-th root of unity in $W(k')$, it can be shown that $$\omega_{K', \sigma_{j'}}=\big(\omega_{f'}\big) ^{p^{f'-j'}}$$ for all $j'=0, 1, \cdots, f'-1$. By abuse of notation, we will denote the pullback of the character $\omega_{f'}$ to $I_K$ still by $\omega_{f'}$.

We call a continuous homomorphism $\tau: I_K \to \text{GL}_n(\overline{\mathbb Q_p})$ an inertial type if it can extend to a homomorphism of the Weil group $W_K\to \text{GL}_n(\overline{\mathbb Q_p})$ with $\tau(I_K)$ being finite. We will enlarge $E$, if necessary, so that the image of $\tau$ will land in $\text{GL}_n(E)$. We will be studying tame inertial types that factor through $L'$, so we will abuse the notation and write $\tau:\Delta' \to \text{GL}_n(E)$.

 Fix one such $\tau:\Delta' \to \text{GL}_n(E)$ throughout the paper. Write $\tau=\bigoplus_{1 \le i \le n} \chi_i$ with $\chi_i=\omega_{f'} ^{\textbf a_i}$ with $0 \le \aaa_1 \le \aaa_2 \le \cdots \le \aaa_n \le e'-1 $. For any $j'\in \mathcal J'$, let $\aaa_i^{j'}$ be the unique integer in $[0,e'-1]$ such that $\chi_i=\omega_{K', \sigma_{j'}} ^{\textbf a_i^{j'}}$. Let the orientation of $\tau$ be $\sss=(\ssj)_{j'\in \mathcal J'} \in W^{\mathcal J'}$, where $\ssj \in S_n$ satisfies $\aaa^{j'}_{\ssj(1)} \le \aaa^{j'}_{\ssj(2)} \le \cdots \le \aaa^{j'}_{\ssj(n)}$ \footnote{The precise definition will be given in \autoref{Breuil-Kisin modules with descent data}.}. Note that the orientation may not be unique in our paper since we have no ``generic condition" as in \cite{LLHLM1}.

Let $T$ be the diagonal torus of $\text{GL}_n$, and let $X^*(T)$ (resp. $X_*(T)$) be the group of characters (resp. cocharacters) of $T$.

For any ring $R$, let $\Ma(R)$ be the set of $n \times n$ matrices with entries in $R$. We will regard $s \in W$ as the standard permutation matrix in $\Ma(R)$. Also, for any $\lambda=(\lambda_1, \cdots, \lambda_n) \in X_*(T)$, we will regard $v^{\lambda}$ as the diagonal matrix $\mathrm{Diag}(v^{\lambda_1}, \cdots, v^{\lambda_n})$ in $\Ma(R \vvv)$. Let $\omega^{(j)}=s_j^{-1}v^{\mu_j}\in \Ma(R\vvv)$ for each $j \in \mathcal J$.

For any topological $\mathcal O$-algebra $R$, we use $\R$ to denote the category of continuous representations $G_K \to \GL_n(R)$. We use the notation $\mathrm{Rep}_R^{\mathrm P}(G_K)$ to denote the category of representations which satisfy the property P.

\section{Galois representations and the Emerton-Gee stack} \label{s2}

\subsection{Some \texorpdfstring{$p$}{Lg}-adic Hodge theoretic rings} \label{s21}
Before we introduce the Emerton-Gee stack in the next subsection, we list some $p$-adic Hodge theoretic rings that will be used throughout this paper.

Let $\mathbb C_p$ be the $p$-adic completion of $\overline K$. Let $\mathcal O_{\mathbb C_p}$ be its ring of integers. Denote $\mathcal O_{\mathbb C_p}^\flat$ as its tilting, which means $\lim\limits_{\longleftarrow \varphi} \mathcal O_{\mathbb C_p}/p$, where the Frobenius map $\varphi$ is the map of taking the $p$-th power. Let $$\underline{-p}={(\pi_0, \pi_1, \pi_2, \cdots) \in \mathcal O_{\mathbb C_p} ^\flat},$$ where $\pi_{i+1}^p=\pi_i, \pi_0=-p$ for all $i \ge 0$. Let $$\underline{1}=(1, \zeta_p, \zeta_{p^2}, \cdots) \in \mathcal O_{\mathbb C_p} ^\flat,$$ where  $(\zeta_{p^{i+1}})^p=\zeta_{p^i}$ and $(\zeta_p)^p=1$ for all $i\ge 1$. 

Let $K_{\mathrm{cyc}}=\bigcup_{i \ge 0}K(\zeta_{p^i})$ and let $K_{\infty}=\bigcup_{i \ge 0}K(\pi_i)$. Let $\Gamma_K=\mathrm{Gal}(K_{\mathrm{cyc}}/K)$. 

Let $\mathbf A_{\mathrm{inf}}=W(\mathcal O_{\mathbb C_p}^\flat)$ be the ring of Witt vectors of $\mathcal O_{\mathbb C_p}^\flat$. Let $\varphi_{\mathbf{A}_{\mathrm{inf}}}$ be the Frobenius map of the Witt ring. We will also just denote it by $\varphi$. Let $\mathfrak m$ be the maximal ideal of $\mathcal O_{\mathbb C_p}^\flat$. Let $[\underline{1}]$ (resp. $[\underline{-p}]$) denote the multiplicative lift of $\underline{1}  \in \mathcal O_{\mathbb C_p}^\flat$ (resp. $\underline{-p} \in \mathcal O_{\mathbb C_p}^\flat$) into its Witt ring $\mathbf A_{\mathrm{inf}}$.

Let $\mathbb C_p ^\flat$ be the ring of fractions of $\mathcal O_{\mathbb C_p}^\flat$. We also form its ring of Witt vectors $W(\mathbb C_p ^\flat)$. Since $\mathcal O_{\mathbb C_p ^\flat}$ is a subring of $\mathbb C_p ^\flat$, $\mathbf A_{\mathrm{inf}}$ can be embedded into $W(\mathbb C_p ^\flat)$. We also regard $[\underline{1}],[\underline{-p}]$ as elements of $W(\mathbb C_p ^\flat)$.

We now recall the definition of  crystalline and semistable period rings. We have a map $\theta: \mathbf{A}_{\mathrm {inf}} \to \mathcal O_{\mathbb C_p}$ defined by $$\theta(\sum[c_n]p^n)=\sum c_n^{(0)}p^n,$$ where $c_n \in \mathcal O_{\mathbb C_p}^\flat$, and $[\ \cdot\ ]$ again denote the multiplicative lift into the Witt ring. It is well known that $\theta: \mathbf{A}_{\mathrm {inf}} \to \mathcal O_{\mathbb C_p}$ is a ring homomorphism, with its kernel being a principal ideal generated by $\xi:=[\underline{-p}]+p$.

The map $\theta$ extends to a surjective homomorphism $\theta_{\mathbb Q}: \AI\big[\frac{1}{p}\big] \to \mathbb C_p$. Let $$\mathbf B_{\mathrm{dR}}^+:=\lim\limits_{\longleftarrow i} \AI\big[\frac{1}{p}\big]/ (\mathrm{ker\ }\theta_{\mathbb Q})^i,$$ be the $\mathrm{ker\ }\theta_{\mathbb Q}$-adic completion of $\AI\big[\frac{1}{p}\big] $. The period ring $\mathbf B_{\mathrm{dR}}^+$ is a complete discrete valuation ring. We denote its field of fractions $\mathbf B_{\mathrm{dR}}$. It has a filtration via the $\mathbb Z$-powers of the maximal ideal of $\mathbf B_{\mathrm{dR}}^+$.

However, the Frobenius endomorphism $\varphi$ of $\AI\big[\frac{1}{p}\big]$ does not preserve $\mathrm{ker\ }\theta_{\mathbb Q}$. In order to have a suitable Frobenius endomorphism on the period ring, we introduce the following

\begin{defn}
    The rings $\AC,\mathbf B_{\mathrm{cris}}^+,\mathbf B_{\mathrm{cris}}$ are defined as $$\AC:=\big\{\sum_{i=0}^\infty a_i \frac{\xi^i}{i!} \in \mathbf B_{\mathrm{dR}}^+| a_i \in \AI,a_i \to 0 \ p\mathrm{-adically}  \big\},$$ $$\mathbf B_{\mathrm{cris}}^+:=\AC\big[\frac{1}{p}\big]\subseteq \mathbf B_{\mathrm{dR}}^+.$$ 
\end{defn}

Now, the ring $\AC$ has a Frobenius endomorphism. We still use $\varphi$ to denote it. We need to have a logarithm on the ring $\mathbf B_{\mathrm{cris}}^+$, as usual defined by $$\mathrm{log } ([x])=\sum_{n=1}^\infty (-1)^{n+1} \frac{([x]-1)^n}{n},$$ which is proven to be convergent for all $x \in 1+\mathfrak m$ (see \cite[Lemma 9.2.2]{BC}). We define $t\in \AC$ by $t:=\mathrm{log \ } [\underline 1]$. We have $\varphi(t)=pt$. And $\frac{t^m}{m!}\in \AC$ for all $m \ge 0 $.

\begin{defn}
    The crystalline period ring $\mathbf B_{\mathrm{cris}}$ is defined as $\mathbf B_{\mathrm{cris}}:=\mathbf B_{\mathrm{cris}}^+\big[\frac{1}{t}\big].$ And the semistable period ring $\mathbf B_{\mathrm{st}}$ is defined as $$\mathbf B_{\mathrm{st}}:=\mathbf B_{\mathrm{cris}}[\mathrm{log}([\underline{-p}])],$$ the sub-$\mathbf B_{\mathrm{cris}}$-algebra of $\mathbf B_{\mathrm{dR}}$ generated by $\mathrm{log}([\underline{-p}])$. 
\end{defn}

In fact, we have an isomorphism of $\mathbf B_{\mathrm{cris}}$-algebras: $\mathbf B_{\mathrm{st}} \cong \mathbf B_{\mathrm{cris}}[X]$.

We won't be using the filtration and other structures of these period rings in the main proof of this paper, so we omit them. We arrive at our definition of crystalline and semistable representations.

\begin{defn}
Let our property P be either crystalline or semistable. We say that a $p$-adic Galois representation $\rho:G_K \to \mathrm{GL}_n(E)$ is ``P'' if $$\mathrm{dim}_{K_0 \otimes_{\mathbb Q_p} E} (\mathbf B_{\mathrm P} \otimes_{\mathbb Q_p} \rho )^{G_K}=n.$$ We denote by $\mathrm{Rep}_R^{\mathrm P}(G_K)$ the subcategory of $\R$ which contains ``P'' representations. We say that a representation $\rho: G_K \to \GL_n(E)$ is potentially ``P" if there exists some finite Galois extension $L/K$ such that $\rho|_{G_L}$ is ``P".
\end{defn}

In this paper, we will be focusing on the situation where $L=L'$, a tamely ramified extension of $K$.

Now, we turn to the $p$-adic Hodge theoretic rings in the integral setting following the definitions of Breuil and Kisin. 

\begin{defn} \thlabel{Kis}
    Let $R$ be any $\mathcal O$-algebra. Kisin's ring $\SSSS$ is defined to be $\mathfrak S_{L',R}=(W(k')\otimes _{\mathbb Z_p}R)\llbracket u'\rrbracket$. We extend the Frobenius $\varphi: W(k') \to W(k')$ to a continuous homomorphism $\varphi:\mathfrak S_{L',R} \to \mathfrak S_{L',R}$ by letting it act as $\varphi$ on $W(k')$, trivially on $R$, and sending $u'$ to $(u')^p$. \footnote{We use $\varphi$ for the Frobenius of all the  \textit{rings} throughout this paper. The Frobenius on \textit{modules} will be written as $\phi$, sometimes with a subscript.}
    
    The action of $\Delta$ on $\mathfrak S_{L',R}$ is defined as follows. For any $g \in \Delta' \subseteq \Delta$, let $g(u')=\omega_{K'}(g)u'$, and let $g$ act trivially on $W(k')$ and $R$. The action of $\Fr \in \Delta$ is defined to be trivial on $u'$ and $R$, and equal $\varphi^f$ on the Witt ring $W(k')$. 
    
    Let $v=(u')^{e'}$ and let $\mathfrak S_R=\mathfrak S_{L',R}^{\Delta=1}=(W(k)\otimes _{\mathbb Z_p}R)\llbracket v \rrbracket$. Set $E(u')=(u')^{e'}+p$. We sometimes also write $E$ for $E(u')$ when it is clear that there won't be any mention of the field $E$ in the close vicinity.
\end{defn}

We have a standard $R\llbracket u'\rrbracket$-linear decomposition $\mathfrak S_{L',R} \cong \bigoplus_{j' \in \mathcal J'} R\llbracket u'\rrbracket$ induced by the map $W(k') \otimes _ {\mathbb Z_p} R \to \bigoplus_{j' \in \mathcal J'} R$ defined by $$x \otimes r \mapsto (\sigma_{j'}(x)r)_{j' \in \mathcal J'}.$$ So the action of $g \in \Delta' \subseteq \Delta$ on $u'$ at the $j'$-th embedding is given by $u' \mapsto \omega_{K', \sigma_{j'}}(g) u'$. And the action of $\Fr \in \Delta$ is to map the $j'$-th embedding of $R\llbracket u'\rrbracket$ isomorphically to the $(j'+f)$-th embedding for any $j' \in \mathcal J'$. By abuse of notation, we also use $\sigma_0$ to denote the projection $\mathfrak S_{L',R} \cong \bigoplus_{j' \in \mathcal J'} R\llbracket u'\rrbracket \to R \llbracket u'\rrbracket$ onto the $j'=0$ summand.

In addition to the notations and rings above, we need the following notions.

\begin{defn} \thlabel{Breuil}
    Let $R$ be any $\mathcal O$-algebra. We define Breuil's ring $\mathcal S_R$ by the following formula $$\widehat{\mathcal S_R:=\left(W(k')[u'] \otimes_{\mathbb Z_p} R\right)\left[\left(\frac{E(u')^i}{i!}\right)_{i=1,2,\cdots}\right]},$$
    i.e., the $p$-adic completion of the divided power envelope of $W(k')[u'] \otimes_{\mathbb Z_p} R$ with respect to the ideal generated by $E(u')$.

     The ring $\mathcal S_R$ is given a PD filtration whose $i$-th filtration $\mathrm{Fil}^i(\mathcal S)$ is defined by the ideal of $\mathcal S_R$ generated by $E^j/j!$ for all $j \ge i$.

    The Frobenius map $\varphi$ on $\SSSS$ extends to $\varphi: \s \to \s$. We write $N_{\s}$ for the unique $K'_0$-linear derivation on $\s$ such that $N_{\s}(u')=-u'$. Let $$I^+_{\s}=\big\{x \in \s | x = \sum_{i=1}^{\infty} a_i (u')^i, a_i \in K'_0\big\}.$$ 

\end{defn}

\begin{rmk}
    In both \thref{Kis} and \thref{Breuil}, we will omit the subscript $R$ when $R=\mathbb Z_p$. This coincides with Breuil and Kisin's original definitions.
\end{rmk}

We will be using the following lemma.

\begin{lemma} \thlabel{S}
    Every element in $\mathcal S$ can be written as $\sum_{i=0}^\infty a_i \frac{u'^i}{i!}$ with $a_i \in W(k')$ for all $i$ and $a_i \to 0$ $p$-adically. 
\end{lemma}
\begin{proof}
    Given any $x \in \mathcal S$, we can write $x=\sum_{i=0}^\infty b_i \frac{E^i}{i!}$ with $b_i \in W(k')[u']$ and $b_i \to 0$ $p$-adically. Since $E=u'^{e'}+p$, $$x=\sum_{i=0}^\infty b_i \sum_{j=0}^i \frac{u'^{e'j}}{j!}\frac{p^{i-j}}{(i-j)!}=\sum_{j=0}^\infty \big(\sum_{i=j}^\infty b_i \frac{p^{i-j}}{(i-j)!} \big) \frac{u'^{e'j}}{j!}.$$ Since $\frac{p^{i-j}}{(i-j)!} \in W(k')$, the coefficients $\sum_{i=j}^\infty b_i \frac{p^{i-j}}{(i-j)!}  \in W(k') \llbracket v \rrbracket$. Since $b_i \to 0$ $p$-adically, $\nu(\sum_{i=j}^\infty b_i \frac{p^{i-j}}{(i-j)!}) \ge \mathrm{min}\{\nu(b_i)|i\ge j\} \to \infty$ as $j \to \infty$, i.e. $\sum_{i=j}^\infty b_i \frac{p^{i-j}}{(i-j)!} \to 0$ $p$-adically as well when $j \to \infty$. Thus we can write $x=\sum_{j=0} ^ \infty c_j \frac{u'^{j}}{j!}$ with $c_j \in W(k') \llbracket u' \rrbracket$ satisfying $c_j \to 0$ $p$-adically as $j \to \infty$. Write $c_j = \sum_{i=0}^\infty d_{ji}u'^i$ with $d_{ji} \to 0$ $p$-adically as $i \to \infty$. Since $c_j \to 0$ $p$-adically, we know that $d_{ji} \to 0$ uniformly for all $i$ as $j \to \infty$. Then $$x=\sum_{j=0} ^ \infty c_j \frac{u'^j}{j!}=\sum_{j=0}^\infty \sum_{i=0}^\infty d_{ji}u'^i \frac{u'^j}{j!}=\sum_{l=0}^\infty \big(\sum_{i=0}^l \frac{l! d_{l-i,i}}{(l-i)!}\big) \frac{u'^l}{l!}.$$

    Now, we estimate the $p$-adic order of the coefficient $\frac{l! d_{l-i,i}}{(l-i)!}$: $\nu(\frac{l! d_{l-i,i}}{(l-i)!}) \ge \nu(i!)+\nu(d_{l-i,i}) \ge \frac{i}{p}-1 + \nu(d_{l-i,i})$. Now it suffices to show that $\sum_{i=0}^l \frac{l! d_{l-i,i}}{(l-i)!} \to 0$ $p$-adically as $l \to \infty$. Given any $S>0$, we first take $i_0$ such that $\frac{i_0}{p}-1 \ge S$. Then take $j_0$ such that $\nu(d_{ji}) \ge S+1$ for all $j>j_0$ and $i$. Let $l$ be any positive integer satisfying $l>i_0+j_0$. If $i>i_0$, then $\nu(\frac{l! d_{l-i,i}}{(l-i)!}) \ge \frac{i}{p}-1 \ge \frac{i_0}{p}-1 \ge S$. If $i \le i_0$, then $l-i>j_0$. Thus, $\nu(d_{l-i,i}) \ge S+1$. So $\nu(\frac{l! d_{l-i,i}}{(l-i)!}) \ge -1+\nu(d_{l-i,i}) \ge S$. This proves the convergence claim.
\end{proof}

We will also be using the ring $\mathcal O_{\mathcal E}$ (resp. $\mathcal O_{\mathcal E,L'}$), which is defined to be the $p$-adic completion of $W(k)\llbracket v \rrbracket [1/v]$ (resp. $W(k')\llbracket u' \rrbracket [1/u']$). Since we won't need the entire theory of \'etale $\varphi$-modules, we refer readers to \cite[Section 5.4]{LLHLM1}.

\subsection{The Emerton-Gee stack}
Here we review some basic results of the Emerton-Gee stack. We assume that readers are familiar with basic notions of the \'etale $(\varphi, \Gamma)$-modules. For details, we refer readers to \cite[Section 2.7]{EG}.

\begin{defn}
    We say that $M$ is a projective \'etale $(\varphi,\Gamma_K)$-module of rank $n$ with $R$-coefficients if $M$ is a projective module over $\mathbf A_{K,R}$ of rank $n$, equipped with a continuous semilinear $\Gamma_K$-action, and a $\Gamma_K$-equivariant isomorphism $\Phi:M \otimes_{\mathbf A_{K,R},\varphi}\mathbf A_{K,R} \to M$.
\end{defn}
The canonical topology on $M$ is described in \cite[Appendix D2]{EG}.

We have the following well-known equivalence theorem by Fontaine. Its original form is \cite[Theorem 13.6.5]{BC} and its variant with coefficients is \cite[Theorem 2.1.27]{Dee}:

\begin{thm} \thlabel{phi,gamma}
   Let $R$ be a Noetherian complete local ring with finite residue field of characteristic $p$, then the category of projective \'etale $(\varphi,\Gamma_K)$-modules of rank $n$ with $R$-coefficients is equivalent to the category of rank $n$ Galois representations with coefficients in $R$.
\end{thm}

\begin{rmk}
    Since finite flat $\mathcal O$-algebras always decompose into a finite product of Noetherian complete local rings with finite residue field of characteristic $p$, the equivalence theorem above also applies when the coefficient ring $R$ is a finite flat $\mathcal O$-algebra.  
\end{rmk}

\cite[Definition 3.2.1]{EG} gives the following definition.

\begin{defn}
    Let the stack $\mathcal X_n$ denote the moduli stack over $\mathrm{Spf\ }\mathcal O$ of projective \'etale $(\varphi,\Gamma_K)$-modules of rank $n$. More precisely, if $R$ is a $p$-adically complete $\mathcal O$-algebra, then we define $\mathcal X_n(R)$ to be the groupoid of projective \'etale $(\varphi,\Gamma_K)$-modules of rank $n$ with $R$-coefficients.
\end{defn}

Hence we know that $\mathcal X_n(R)$ is equivalent to the groupoid $\mathrm{Rep}_R(G_K)$ when $R$ is a finite flat $\mathcal O$-algebra. Or in other words, an $R$-point of $\mathcal X_n$ can be seen as the equivalence class of a Galois representation $\rho:G_K \to \mathrm{GL}_n(R)$. Thus, the Emerton-Gee stack can be viewed as the moduli of Galois representations.

We are also very interested in substacks with $p$-adic Hodge theoretic conditions. Now we recall the potentially semistable locus of the Emerton-Gee stack. First, we give the definition of inertial types. We will also introduce coefficients when discussing Galois representations.

Let $\Gamma=\mathrm{Gal}(L/K)$. Let $K'$ be the maximal unramified extension of $K$ inside $L$.

We first recall the definition of $(\phi, N, \Gamma)$-modules.

\begin{defn} \thlabel{phi,n,gammma}
Let $A$ be an $E$-algebra, we say that $(D,\phi_D, N_D)$ is a $(\phi, N, \Gamma)$-module of rank $n$ with coefficients in $A$ if
 \begin{enumerate}
     \item $D$ is a free $K'_0 \otimes_{\mathbb Q_p} A$-module of rank $n$;
     \item $\phi_D:\varphi^*(D) \to D$ is a $K'_0 \otimes_{\mathbb Q_p} A$-linear isomorphism, where $\varphi^*(D)$ is defined to be $D \otimes_{K_0',\varphi}K_0'$;
    \item $N_D \in \mathrm{End}_{K_0' \otimes_{\mathbb Q_p} A}(D)$ satisfies $ N\phi_D=p\phi_D N_D$;
    \item $D$ has a $\Gamma$-semilinear action, which commutes with $\phi_D, N_D$.  
 \end{enumerate}
 We denote this category by $L(\phi, N, \Gamma)_A$.
\end{defn}

\begin{defn}
    Let $D_{\mathrm{pst}}$ be the following functor: $$\rho \mapsto (\mathbf B_{\mathrm {st}}\otimes_{\mathbb Q_p}\rho)^{G_{L}},$$ from the category of $G_K$-representations $\mathrm{Rep}_A(G_K)$ to the category of $(\phi, N, \Delta)$-modules $L(\phi, N, \Delta)_A$ with coefficients in an $E$-algebra $A$. 
\end{defn}

\begin{rmk}
According to the construction of the semistable period ring, $\mathbf B_{\mathrm{st}}\otimes_{\mathbb Q_p}\rho$ naturally has a $G_K$-action. Taking the $G_L$-invariants will let the $G_K$-action factor through the quotient $\Gamma=\mathrm{Gal}(L/K)$.
\end{rmk}
\begin{rmk}
Since we will not use the filtration structure, we have taken the word ``filtered" out from the more common ``filtered $(\phi, N, \Gamma)$-modules". And we will always forget the filtration structure when applying the functor $D_{\text{pst}}$. As usual, we will say $D$ is a $(\phi, N, \Gamma)$-module for short.
\end{rmk}

Recall that a continuous homomorphism $\tau: I_K \to \text{GL}_n(E)$ is called an inertial type if it can extend to a homomorphism of the Weil group $W_K\to \text{GL}_n(E)$ with $\tau(I_K)$ being finite. 

\begin{defn}
    We say that a Galois representation $\rho:G_K \to \mathrm{GL}_n(E)$ is potentially semistable of inertial type $\tau$ if $\rho|_{G_L}$ is semistable and the action of $I_K$ through $\Gamma$ on $D_{\mathrm{pst}}(\rho)$ is $\tau \otimes_E(K'_0\otimes_{\mathbb Q_p}E)$.
\end{defn}

We can always conjugate a semistable Galois representation $\rho:G_K \to \mathrm{GL}_n(E)$ into a lattice, which we still denote $\rho:G_K \to \mathrm{GL}_n(\mathcal O)$. It is important to study families of these representations, or even more generally, potentially semistable representations $\rho:G_K \to \mathrm{GL}_n(R)$, where $R$ is a finite flat $\mathcal O$-algebra. These representations cut out a substack of the Emerton-Gee stack, which we often refer to as the potentially semistable Emerton-Gee stack in this paper. We have the following theorem (see \cite[Theorem 4.8.12]{EG} or \cite[Theorem 7.2.2]{LLHLM1}).

\begin{thm} \thlabel{EG}
    Let $\tau$ be an inertial type and $h>0$ an integer. The closed substack $\XXX$ of $\mathcal X_n$ is a $p$-adic formal algebraic stack which is finite type and flat over $\mathrm{Spf \ } \mathcal O$, and is uniquely determined as the $\mathcal O$-flat closed substack of $\mathcal X_n$ by the following property: if $R$ is a finite flat $\mathcal O$-algebra, then $\XXX(R)$ is precisely the subgroupoid of $\mathcal X_n(R)$ consisting of $G_K$-representations which are potentially semistable of Hodge-Tate weights between $0$ and $h$ and inertial type $\tau$.
\end{thm}

A relation between the Emerton-Gee stack, Mazur's framed deformation rings (see \cite{Ma}) and Kisin's potentially semistable deformation rings (see \cite{Kis}) is as follows.

\begin{thm} \thlabel{versal}
    There is a morphism $\mathrm{Spf\ } R_{\bar \rho}^{\Box} \to \mathcal X_n$ (resp. $\mathrm{Spf\ } R_{\bar \rho}^{\Box} \to \XXX$) which is versal at the point $x$ corresponding to $\bar \rho \in \mathcal X_n(\mathbb F)$ (resp. $\bar \rho \in \XXX(\mathbb F)$).
\end{thm}

This shows that the Emerton-Gee stack encodes all the information of framed deformation rings. It is a globalization of the framed deformation rings. This has great applications. For example, it allows the reformulation of the Breuil-M\'ezard conjecture in a more global way independent of a fixed mod $p$ representation, now called the geometric Breuil-M\'ezard conjecture (see \cite[Section 8.3]{EG} and \cite[Section 8]{LLHLM1}). In this paper, we will be using this result to recover formal smoothness properties of the potentially semistable deformation rings from those of the corresponding stack.

At the end of this section, we record a definition about the rigid generic fiber---\cite[(5.1.11)]{EGH22}, which will be used later on.

\begin{defn} \thlabel{rig}
    The rigid generic fiber $\Xe$ of $\XXX$ is defined as follows. Given any affinoid $E$-algebra $A$, the $A$-points of $\Xe$ are $\lim\limits_{\longrightarrow Y} \XXX(Y)$, where the limit is taken over all formal models (in the sense of Raynaud) $Y$ of $\mathrm{Sp}(A)$.
\end{defn}

By \cite[Proposition 5.1.12]{EGH22}, $\Xe$ is a rigid analytic Artin stack.

We give a result relating $\Xe$ and actual Galois representations. We use $\mathrm{Rep}^{[0,h],\tau}_E(G_K)$ to denote the category of potentially semistable representations of inertial type $\tau$ and height bounded by $h$, with coefficients in the field $E$.

\begin{prop} \thlabel{k}
    There is a functor $\kappa:\mathrm{Rep}^{[0,h],\tau}_E(G_K) \to \Xe(E)$.
\end{prop}

\begin{proof}
    Since any open ideal of $\mathcal O$ is principal and invertible, admissible blowups of $\mathrm{Spf\ } \mathcal O$ are trivial. According to Raynaud's theorem \cite[Section 8.4, Theorem 3]{Bo}, $\mathrm{Spf\ } \mathcal O$ is the only formal model of $\mathrm{Sp\ }E$. For any $\rho \in \mathrm{Rep}^{[0,h],\tau}_E(G_K) $, there are $G_K$-stable lattices $G_K \to \mathrm{GL}_n(\mathcal O)$ in $\rho$. They correspond to points in $\XXX(\mathcal O)$. According to \thref{rig}, these give a point in $\Xe(E)$ after passing to a colimit. The functoriality of $\kappa$ is clear.
\end{proof}

\section{Breuil-Kisin modules and potentially semistable Galois representations} \label{s3}
In this section, we combine the integral $p$-adic Hodge theoretic descriptions used in \cite[Section 2]{LLHLM0} and \cite[Section 5]{LLHLM1} (with ramified modifications in \cite[Section 3]{LLHLM3}) with \cite[Section 7]{Gao} to define and describe our so-called potentially semistable Breuil-Kisin modules. We modify some of the results to fit those $\tau$ that are not generic.

\subsection{Breuil-Kisin modules with descent data}\label{Breuil-Kisin modules with descent data}

Fix an inertial type $\tau: I_K \to \text{GL}_n(E)$. We will only be dealing with tame inertial types in this paper, i.e. those homomorphisms $\tau:I_K \to \mathrm{GL}_n(E)$ that factor through the tame inertia. These tame inertial types can be classified by the Weyl group $W$ and the character group $X^*(T)$ as follows (see \cite[Example 2.4.1]{LLHLM1}). 

Let $$s=(s_0,\cdots, s_{f-1})\in W^\mathcal J,$$ $$\mu=(\mu_0, \cdots, \mu_{f-1})  \in X^*(T)^\mathcal J.$$ 

\begin{defn}
  Let the tame inertial type associated to $(s,\mu)\in W^{\mathcal J}\times X^*(T)^{\mathcal J}$ be defined by $$\tau(s,\mu)=\sum_{i=0}^{d-1}\big((F^*\circ s^{-1})^i(\mu)\big)(\omega_d),$$   where we view $\mu$ here as an element of $X_*(T^\vee)^{\mathcal J}$, $F^*$ is defined to be the endomorphism $p \pi^{-1}$ on $X_*(T^\vee)^{\mathcal J}$, and $d>1$ is an integer such that $(F^*\circ s^{-1})^d=p^d$.
\end{defn}

Suppose we have $\tau=\tau(s,\mu)$ for some $(s,\mu)\in W^{\mathcal J}\times X^*(T)^{\mathcal J}$. Set $$s_\tau=s_0s_1\cdots s_{f-1}\in W.$$ Assume that $s_\tau$ has order $r$ in $S_n$. 

Let the orientation $\sss \in (S_n)^{\mathcal J'}$ be given by $s_{\mathrm{or}, j+kf}=s_\tau^{k+1} (s_{f-1}^{-1}s_{f-2}^{-1} \cdots s_{j+1}^{-1})$ for $0 \le j \le f-1, 0 \le k \le r-1$, where the empty product is interpreted as the identity. Then we have $\aaa^{j'}_{\ssj(1)} \le \aaa^{j'}_{\ssj(2)} \le \cdots \le \aaa^{j'}_{\ssj(n)}$ for all $j' \in \mathcal J'$.

 For each $j' \in \mathcal J'$, let $P_{j'} \subseteq \GL_n$ be the parabolic subgroup according to the data $\aaa^{j'}_{\ssj(1)} \le \aaa^{j'}_{\ssj(2)} \le \cdots \le \aaa^{j'}_{\ssj(n)}$. To be precise, suppose the sequence $\aaa^{j'}_{\ssj(1)} \le \aaa^{j'}_{\ssj(2)} \le \cdots \le \aaa^{j'}_{\ssj(n)}$ consists of distinct numbers $b_1<b_2< \cdots < b_m$ with multiplicities $n_i \ge 1$ for $b_i$. Then $P_{j'}$ is defined to be the group of upper-triangular block matrices with sizes of blocks being $n_1, n_2, \cdots, n_m$. Let  the corresponding parahoric group scheme $\mathcal P_{j'}$ over $\mathcal O$ be the group scheme whose fiber over any $\mathcal O$-algebra $R$ is the set $$\{A \in \text{GL}_n(R \vvv) | A \text{ mod }v \text{\ lies in\ }P_{j'}(R)\}.$$ 

 \begin{lemma}
     The groups $P_{j'}$ and $\mathcal P_{j'}$ only depend on $j' \mathrm{\ mod\ } f$. 
 \end{lemma}
 \begin{proof}
     Since $\omega_{K', \sigma_{j'}}=\omega_{K', \sigma_{j'+f}}^{p^f}$ for all $j' \in \mathcal J'$, $a_i^{j'+f} \equiv p^f a_i^{j'} \mathrm{\ mod\ } e'$. On the other hand, since $\tau$ extends to a representation of $\Delta$, the multiset of $\Delta'$-characters is invariant under Frobenius conjugation. This means, as multisets, we have $$\{p^fa_i^{j'}  \mathrm{\ mod\ } e'|i=1, \cdots,n\}=\{a_i^{j'}  \mathrm{\ mod\ } e'|i=1, \cdots,n\}.$$ Therefore,  $$\{a_i^{j'+f} |i=1, \cdots,n\}=\{a_i^{j'} |i=1, \cdots,n\}.$$ Note that $\sss$ keeps the non-decreasing orders, we have $\aaa^{j'+f}_{s_{\mathrm{or},j'+f}(i)}=\aaa^{j'}_{\ssj(i)}$ for all $j' \in \mathcal J', 1\le i \le n$. This gives the result.
 \end{proof}

 \begin{rmk}
     We will be writing $\mathcal P_j$ with $j \in \mathcal J$ from now on.
 \end{rmk}

\begin{defn}  \thlabel{BK0}
A rank $n$ Breuil-Kisin module of height at most $h$ over $\mathfrak S_R$ is defined to be a pair $(\mathfrak M, \phi_\mathfrak M)$, where $\mathfrak M$ is a finitely generated projective $\mathfrak S_R$-module, which is locally free of rank $n$, and the Frobenius $\phi_\mathfrak M : \varphi^*(\mathfrak M) \to  \mathfrak M$ is an injective $\mathfrak S_R$-linear map whose cokernel is annihilated by $(v+p)^h$.

Similarly, a rank $n$ Breuil-Kisin module of height at most $h$ over $\mathfrak S_{L',R}$ is defined to be a pair $(\mathfrak M, \phi_\mathfrak M)$, where $\mathfrak M$ is a finitely generated projective $\mathfrak S_{L',R}$-module, which is locally free of rank $n$, and the Frobenius $\phi_\mathfrak M : \varphi^*(\mathfrak M) \to  \mathfrak M$ is an injective $\mathfrak S_{L',R}$-linear map whose cokernel is annihilated by $E(u')^h$.
\end{defn}

\begin{rmk}
    The condition of the cokernel being annihilated by $(v+p)^h$ or $E(u')^h$ is the condition of a (potentially) semistable Galois representation having Hodge-Tate weight bounded by $h$ when such a Breuil-Kisin module corresponds to one.
\end{rmk}

Let $(\mathfrak M, \phi_\mathfrak M)$ be a Breuil-Kisin module over $\mathfrak S_{L',R}$, we have a standard $R\llbracket u'\rrbracket$-linear decomposition: $\mathfrak M \cong \bigoplus_{j' \in \mathcal J'} \mathfrak M^{(j')}$ induced by $\mathfrak S_{L',R} \cong \bigoplus_{j' \in \mathcal J'} R\llbracket u'\rrbracket$, where each $\mathfrak M^{(j')}$ is an $R\llbracket u'\rrbracket$-module. The Frobenius $\phi_\mathfrak M$ induces $R\llbracket u' \rrbracket$-linear maps $\phi_\MM^{(j')}: \varphi^*(\mathfrak M^{(j'-1)}) \to \mathfrak M^{(j')}$ for each $j' \in \mathcal J'$. We can obtain an $R$-module $\MM^{(j')}/u'$ from $\MM^{(j')}$.

\begin{defn} \thlabel{BKK1}
A rank $n$ Breuil-Kisin module of height at most $h$ over $\mathfrak S_{L',R}$ with descent data $\tau$ is defined to be a pair $(\mathfrak M, \phi_\mathfrak M)$, where $(\mathfrak M, \phi_\mathfrak M)$ is a rank $n$ Breuil-Kisin module of height at most $h$ over $\mathfrak S_{L',R}$ with a semilinear action of $\Delta$ on $\mathfrak M$ that commutes with $\phi_\mathfrak M$, such that for each $j'\in \mathcal J'$, $\mathfrak M^{(j')} / u' \cong \tau  \otimes_\mathcal O R $ as a $\Delta'$-representation with coefficients in $R$.
\end{defn}

By abuse of language, we will only use $\mathfrak M$ to refer to $(\mathfrak M, \phi_\mathfrak M)$ from now on. Let $Y^{[0,h], \tau}(R)$ be the groupoid of rank $n$ Breuil-Kisin modules of height at most $h$ over $\mathfrak S_{L',R}$ with descent data $\tau$. Note that the semilinear action of $\Fr \in \Delta$ induces a bijection $\iota_\mathfrak M: \Fr^*(\mathfrak M)\cong \mathfrak M$.

\begin{defn} \thlabel{eigenbasis}
Let $\mathfrak M \in Y^{[0,h], \tau}(R)$. 

For each $j' \in \mathcal J'$ and each $i=1,2, \cdots, n$, define $\mathfrak M^{(j')} _i$ to be the $R \llbracket v\rrbracket$-submodule of $\mathfrak M^{(j')}$ such that the $\Delta'$-action is via the character $\chi_i$. \footnote{We have taken a dual to the definition in \cite{LLHLM1} in order to fit the conventions of inertial types in \cite{EG}, which will be used later on.} An eigenbasis of $\mathfrak M$ is a collection of ordered bases $\beta^{(j')}=(f_1^{(j')}, \cdots, f_n^{(j')})$ for each $R \llbracket u' \rrbracket$-module $\mathfrak M^{(j')}$ where $f_i^{(j')} \in \mathfrak M^{(j')} _i$ and $\iota_\mathfrak M(\Fr^*(f^{(j')} _i))=f^{(j'+f)} _i$ for all $i=1,2,\cdots,n$.

For $j'\in \mathcal J'$, we call 
$$\beta_{\ssj(n)}^{(j')}=((u')^{\aaa^{j'}_{\ssj(n)}-\aaa^{j'}_{\ssj(1)}} f^{j'}_{\ssj(1)}, (u')^{\aaa^{j'}_{\ssj(n)}-\aaa^{j'}_{\ssj(2)}} f^{j'}_{\ssj(2)}, \cdots, f^{j'}_{\ssj(n)})$$ 
the basis of $\MM^{(j')}_{\ssj(n)}$ as an $R \llbracket v \rrbracket$-module associated to the eigenbasis $\beta$.

For $j'\in \mathcal J'$, define $^\varphi \mathfrak M^{(j')} _i$ to be the $R \llbracket v\rrbracket$-submodule of $\varphi^*(\MM^{(j')})$ such that the $\Delta'$-action is via the character $\chi_i$ for any $i=1,2,\cdots, n$. We call 
$$^\varphi \beta_{\ssj(n)}^{(j'-1)}=((u')^{\aaa^{j'}_{\ssj(n)}-\aaa^{j'}_{\ssj(1)}}\otimes f^{j'-1}_{\ssj(1)}, (u')^{\aaa^{j'}_{\ssj(n)}-\aaa^{j'}_{\ssj(2)}}\otimes f^{j'-1}_{\ssj(2)}, \cdots, 1\otimes f^{j'-1}_{\ssj(n)})$$
the basis of $^\varphi \MM^{(j'-1)}_{\ssj(n)}$ as an $R \llbracket v \rrbracket$-module associated to the eigenbasis $\beta$.
\end{defn}

\begin{rmk}
    An eigenbasis always exists Zariski locally on the coefficient ring $R$ according to \cite[Proof of Proposition 5.2.1]{LLHLM1}.
\end{rmk}

\begin{defn} \thlabel{matrix}
Let $\mathfrak M \in Y^{[0,h], \tau}(R)$. Given an eigenbasis $\beta$ for $\MM$, for each $j' \in \mathcal J'$, the matrix of Frobenius $C^{(j')}=C^{(j')}_{\MM, \beta}$ is defined to be the matrix of $\phi_\MM^{(j')}$ with respect to the bases $\varphi^*(\beta^{(j'-1)})$ and $\beta^{(j')}$, i.e. $\phi_\MM^{(j')}(\varphi^*(\beta^{(j'-1)}))= \beta^{(j')}C^{(j')} $. The matrix of partial Frobenius $A^{(j')}=A^{(j')}_{\MM, \beta}$ is defined to be the matrix of $\phi_\MM^{(j')} |_{^\varphi \MM^{j'-1}_{\ssj(n)}} $ with respect to the bases $^\varphi \beta_{\ssj(n)}^{(j'-1)}$ and $\beta_{\ssj(n)}^{(j')}$ associated to $\beta$, i.e. $\phi_\MM^{(j')}(^\varphi \beta_{\ssj(n)}^{(j'-1)})= \beta_{\ssj(n)}^{(j')}A^{(j')}$.
\end{defn}

\begin{prop}  \thlabel{Frob}
Let $\beta$ be an eigenbasis of $\MM \in Y^{[0,h], \tau}(R)$, then $A^{(j')}_{\MM, \beta}, (v+p)^h (A^{(j')}_{\MM, \beta})^{-1} \in \mathcal P_{j}(R)$. $A^{(j')}_{\MM, \beta}$ only depends on $j' \mathrm{\ mod \ } f$ and so will also be denoted as $A^{(j)}_{\MM, \beta}$, for $j = j' \mathrm{\ mod \ } f$. 

The relation between $A^{(j')}_{\MM, \beta}$ and $C^{(j')}_{\MM, \beta}$ is
$$A^{(j')}_{\MM, \beta}=\mathrm{Ad}((\ssj)^{-1}(u')^{-\aaa^{j'}}) C^{(j')}_{\MM, \beta}.$$
\end{prop}

\begin{proof}
    The relation between the $A$ and $C$-matrices is a well-known fact established, for example, in \cite[(5.4)]{LLHLM1}. The dependence-modulo-$f$ statement follows also from the same source. We only need to explain why $A^{(j')}_{\MM, \beta}, (v+p)^h (A^{(j')}_{\MM, \beta})^{-1} \in \mathcal P_j(R)$ in our setting (where the $1$-genericity condition is not assumed).

    Write $a_{ik}$ (resp. $c_{ik}$) for the $(i,k)$-entry of the matrix $A^{(j')}_{\MM, \beta}$ (resp. $C^{(j')}_{\MM, \beta}$). Since $\phi_\MM$ commutes with the descent datum given by $\tau$, for any $g \in \Delta'$, $g(c_{ik})=\chi_k(g) \chi_i^{-1}(g) c_{ik}$. Therefore, $c_{ik}$ lives in $(u')^{[a_k^{j'}-a_i^{j'}]}R\llbracket v \rrbracket$, where $[ \cdot ]$ denotes the minimal non-negative residue modulo $e'$. From $A^{(j')}_{\MM, \beta}=\mathrm{Ad}((\ssj)^{-1}(u')^{-\aaa^{j'}}) C^{(j')}_{\MM, \beta}$, we can see that $$a_{ik}=(u')^{\aaa^{j'}_{\ssj(i)}-\aaa^{j'}_{\ssj(k)}}c_{\ssj(i),\ssj(k)}.$$ Thus $a_{ik}$ lives in $$(u')^{\aaa^{j'}_{\ssj(i)}-\aaa^{j'}_{\ssj(k)}+[\aaa^{j'}_{\ssj(k)}-\aaa^{j'}_{\ssj(i)}]} R\llbracket v \rrbracket.$$

    If $\aaa^{j'}_{\ssj(k)} \ge \aaa^{j'}_{\ssj(i)}$, then $a_{ik}$ lives in $R\llbracket v \rrbracket$. If $\aaa^{j'}_{\ssj(k)} < \aaa^{j'}_{\ssj(i)}$, then $[\aaa^{j'}_{\ssj(k)}-\aaa^{j'}_{\ssj(i)}]=e'-\aaa^{j'}_{\ssj(k)}+\aaa^{j'}_{\ssj(i)}$. Since $(u')^{e'}=v$, $a_{ik}$ lives in $vR\llbracket v \rrbracket$. By our definition of $\mathcal P_{j'}$, this shows that $A^{(j')}_{\MM, \beta}\in \mathcal P_{j'}(R)$. The proof of $(v+p)^h (A^{(j')}_{\MM, \beta})^{-1} \in \mathcal P_{j'}(R)$ is similar.
\end{proof}

\begin{rmk}
    After the transformation from the matrices $C^{(j')}_{\MM, \beta}$ to $A^{(j')}_{\MM, \beta}$, the entries of $A^{(j')}_{\MM, \beta}$ now all lie in $R\llbracket v \rrbracket$.
\end{rmk}

\begin{prop} \thlabel{change basis}
Let $\beta_1, \beta_2$ be two eigenbases of $\MM \in Y^{[0,h], \tau}(R)$ related by $\beta_2^{(j')} D^{(j')}=\beta_1^{(j')}$ with $D^{j'} \in \mathrm{GL}_{n}(R \uuu)$ for $j' \in \mathcal J'$. Set $I^{(j')}=\mathrm{Ad}((\ssj)^{-1}(u')^{-\aaa^{j'}}) D^{(j')}$, then $I^{(j')} \in \mathcal P_{j'}(R)$ and it depends only on $j' \mathrm{\ mod \ } f$, and for all $j' \in \mathcal J'$,
$$A^{(j')}_{\MM, \beta_2}=I^{(j')}A^{(j')}_{\MM, \beta_1} \mathrm{Ad}(\omega^{(j)}(\varphi(I^{(j'-1)})^{-1}) ,$$
where $j = j' \mathrm{\ mod \ } f$.

Furthermore, given $(I^{(j')})_{j' \in \mathcal J'}\in (\mathcal P_{j'}(R))_{j\in \mathcal J'}$ with $I^{(j')}=I^{(j'+f)}$, then $\mathrm{Ad}((u')^{\aaa^{j'}}\ssj)I^{(j')}=D^{(j')} \in  \mathrm{GL}_{n}(R \uuu)$ and for any eigenbasis $\beta$ of $\MM$, $(\beta^{(j')}D^{(j')})_{j' \in \mathcal J'}$ is also an eigenbasis of $\MM$.
\end{prop}

\begin{proof}
    This is established in \cite[Proposition 5.1.8]{LLHLM1}. The way of adjusting to parahoric subgroups in our genericity-free scenario is the same as in the proof of \thref{Frob}.
\end{proof}

\begin{rmk} \thlabel{frob}
The matrices of partial Frobenii $A_{\MM,\beta}^j$ encode all the information of $\MM \in Y^{[0,h], \tau}(R)$. Conversely, given a collection of matrices satisfying the conditions of the first paragraph of \thref{Frob}, there is an $\MM \in Y^{[0,h], \tau}(R)$ whose matrices of partial Frobenii are the given ones.

The two partial Frobenii that are equivalent under the change of basis formula in \thref{change basis} give isomorphic elements in $Y^{[0,h], \tau}(R)$, and vice versa.
\end{rmk}

\subsection{Potentially semistable Breuil-Kisin modules} \label{ps}

Now we will define potentially semistable Breuil-Kisin modules and establish 
categorical equivalences with potentially semistable representations. Following \cite{Gao}, our potentially semistable Breuil-Kisin module \footnote{This is not a standard terminology.} is the usual Breuil-Kisin module with descent data as in \thref{BK1}, along with the data of a $G_K$-action on the associated Breuil-Kisin-Fargues module. 

To define potentially semistable Breuil-Kisin modules, we need to associate a Breuil-Kisin module with a corresponding $\mathbf{A}_{\mathrm{inf}}$-module, also known as a Breuil-Kisin-Fargues module. 

Recall that we have $\underline{-p}={(\pi_0, \pi_1, \pi_2, \cdots) \in \mathcal O_{\mathbb C_p} ^\flat}$, where $\pi_{i+1}^p=\pi_i, \pi_0=-p$. Let $K_{\infty}=\bigcup_{i \ge 0}K(\pi_i)$. Let $G_{\infty}=G_{K_{\infty}}$. 

Similarly, from $\underline{1}={(1, \zeta_p, \zeta_{p^2}, \cdots) \in \mathcal O_{\mathbb C_p} ^\flat}$, where $(\zeta_{p^{i+1}})^p=\zeta_{p^i}$, we have defined $K_{\mathrm{cyc}}=\bigcup_{i \ge 0}K(\zeta_{p^i})$, then we know that $$\mathrm{Gal}(K_{\infty}K_{\mathrm{cyc}}/K_{\mathrm{cyc}})\cong \mathbb Z_p.$$

Let $\underline{\pi'}=(\pi'_0, \pi'_1, \pi'_2, \cdots) \in \mathcal O_{\mathbb C_p} ^\flat$, where $(\pi'_{i+1})^p=\pi'_i, \pi'_0=\pi', (\pi'_i)^{e'}= \pi_i$. Let $L'_\infty=\bigcup_{i \ge 0}L'(\pi_i')$, and let $G'_{\infty}=G_{L'_{\infty}}$. Let $L'_{\mathrm{cyc}}=\bigcup_{i \ge 0}L'(\zeta_{p^i})$. Since we require that $p>2$, $L'_\infty\cap L'_{\mathrm {cyc}}=L'$ according to \cite[Lem 5.1.2]{Liu1}. Similarly, we have $$\mathrm{Gal}(L'_{\infty}L'_{\mathrm{cyc}}/L'_{\mathrm{cyc}})\cong \mathbb Z_p.$$ We denote the topological generator of this group that sends $\pi_i'$ to $\zeta_{p^i} \pi_i'$ to be $\psi$. We also use $\psi$ to denote its preimage in $G_K$ under the map $G_K \twoheadrightarrow \mathrm{Gal}(L'_{\infty}L'_{\mathrm{cyc}}/K)$. 

So we have the following commutative diagram whose rows are exact.
$$\begin{tikzcd} 
	1 & G'_\infty & G_\infty & \Delta & 1\\
	1 & G_{L'} & G_K & \Delta & 1
	\arrow[from=1-1, to=1-2]
    \arrow[from=1-2, to=1-3]
    \arrow[from=1-3, to=1-4]
    \arrow[from=1-4, to=1-5]
    \arrow[from=2-1, to=2-2]
    \arrow[from=2-2, to=2-3]
    \arrow[from=2-3, to=2-4]
    \arrow[from=2-4, to=2-5]
	\arrow[hook, from=1-2, to=2-2]
	\arrow[hook, from=1-3, to=2-3]
    \arrow[from=1-4, to=2-4] 
\end{tikzcd}$$

For any $\mathcal O$-algebra $R$, let $\AAA=\mathbf{A}_{\mathrm{inf}} \hat{\otimes}_{\mathbb Z_p} R :=\lim\limits_{\longleftarrow i}W_i(\mathcal O_{\mathbb C_p}^\flat)\otimes R$. Let $[\underline{-p}] \in \AAA$ be the multiplicative lift of $\underline{-p} \in \mathcal O_{\mathbb C_p} ^\flat$. There is an embedding $\SSS \hookrightarrow \AAA$ given by $v \mapsto [\underline{-p}]$. Using this embedding, we can associate a Breuil-Kisin module $\MM$ over $\SSS$ with an $\AAA$-module $\MMM=\MM \otimes_{\SSS} \AAA$, with $\phi_{\MMM}=\phi_\MM \otimes \varphi_{\AAA}$. Given any family of matrices $(A^{(j)})_{j\in \mathcal J}$ with each $A^{(j)}\in \Ma(R \llbracket v \rrbracket)$, we can regard $(A^{(j)})_{j \in \mathcal J}$ as an element in $\Ma(\SSS)$. Thus under the embedding $\SSS \hookrightarrow \AAA$, it becomes an element of $\Ma(\AAA)$, which we denote $\mathfrak C((A^{(j)})_{j \in \mathcal J})$.

Similarly, there is also an embedding $\mathfrak S_{L',R} \hookrightarrow \AAA$ given by $u' \mapsto [\underline{\pi'}]$ where $[\underline{\pi'}] \in \AAA$ denotes the multiplicative lift of $\underline{\pi'} \in \mathcal O_{\mathbb C_p} ^\flat$. And hence we can also associate a Breuil-Kisin module $\MM$ over $\mathfrak S_{L',R}$ with an $\AAA$-module $\MMM=\MM \otimes_{\mathfrak S_{L',R}} \AAA$, with $\phi_{\MMM}=\phi_{\MM} \otimes \varphi_{\AAA}$. 

We will need some more $p$-adic Hodge theoretic rings. Let $\mathbb C_p^\flat$ be the ring of fractions of the integral domain $\mathcal O_{\mathbb C_p^\flat}$. Write $\mathcal O_{\mathcal E,L'}=\widehat{W(k')\llbracket u' \rrbracket [1/u']} $, where ``$\ \widehat{}\ $" denotes taking the $p$-adic completion. Naturally, we have inclusions $\mathfrak S_{L'} \hookrightarrow \mathcal O_{\mathcal E,L'} \hookrightarrow W(\mathbb C_p^\flat)$, where sending $u'$ to $[\underline{\pi'}]$ again gives the second inclusion. We furthermore define $ \mathcal O_{\mathcal E,L'} ^{\mathrm{ur}}$ to be the unramified closure of $ \mathcal O_{\mathcal E,L'} $ in $W(\mathbb C_p^\flat)$.

Now we introduce a category $\mathrm{wMod}^{\varphi, G_K}_{\mathfrak S, \AAA}$, which is essentially given in \cite[Definition 7.1.3]{Gao}, but with descent data from $L'$ to $K$ and coefficients added.

\begin{defn}\thlabel{BK}
A rank $n$ effective Breuil-Kisin module $(\MM, \phi_\MM)$ with coefficients in $R$ and of descent data $\tau$ consists of a rank $n$ free Breuil-Kisin module $(\MM,\phi_\MM)$ of descent data $\tau$ over $\SSSS$, together with a $G_K$-action on the associated $\AAA$-module $\widehat{\MM}=\MM \otimes_{\SSSS} \AAA  $ such that 
\begin{enumerate}
    \item The $G_K$-action on $\widehat \MM$ is a continuous $\AAA$-semi-linear action that commutes with $\phi_{\widehat\MM}$;
    \item The embedding $\MM \hookrightarrow \widehat\MM$ factors through a $\Delta$-equivariant embedding $\MM\hookrightarrow \widehat\MM^{G_{\infty}'}$, where the $\Delta$-action on $\widehat\MM^{G_{\infty}'}$ is induced by the $G_\infty$-action on $\MMM$ \footnote{Since $\widehat\MM^{G_{\infty}'}$ is fixed by $G_\infty'$, the $G_\infty$-action factors through $\Delta\cong G_\infty/G_\infty'$.}, which is restricted from its $G_K$-action.
\end{enumerate}

\end{defn}

Again, we usually use $\MM$ or $\MMM$ to refer to $(\MM, \phi_\MM)$. We denote the category of rank $n$ effective Breuil-Kisin modules with coefficients in $R$ by $\mathrm{wMod}^{\varphi, G_K}_{ \AAA}$. We have a functor $$T_R:\mathrm{wMod}^{\varphi, G_K}_{\AAA} \ni \MM \mapsto (\MMM \otimes_{\AAA} W(\mathbb C_p^{\flat})_R)^{\phi_{\MMM} \otimes \varphi_{W(\mathbb C_p^{\flat})}=1} \in \mathrm{Rep}_R(G_K),$$
where the $G_K$-action on the target comes from both the $G_K$-action on $\MMM$ and its natural action on $W(\mathbb C_p^{\flat})$.

\begin{lemma} \thlabel{ff}
    Let $R$ be a finite flat $\mathcal O$-algebra. The functor $T_R:\mathrm{wMod}^{\varphi, G_K}_{ \AAA}  \to \mathrm{Rep}_R(G_K)$ is fully faithful and rank-preserving. 
\end{lemma}

\begin{proof}
  We first prove the full faithfulness result in the case where $R=\mathcal O=\mathbb Z_p$. Write $T=T_{\mathbb Z_p}$. Given $\rho=T_{\mathbb Z_p}(\MM),\rho'=T_{\mathbb Z_p}(\MM')$, and $g$ a morphism from $\rho$ to $\rho'$, we first regard $\rho, \rho' \in \R$ as $G_{L'}$-representations. We make use of \cite[Proposition 7.1.4]{Gao}, and treat the field ``$K$" there as our ``$L'$". Then there is a unique map $f:\MM \to \MM'$ such that $T(f)=g$ and $f$ commutes with $\phi$ and the $G_{L'}$-action on $\MMM$. Note that $g$ commutes with the $G_K$-action on $\rho$ and $\rho'$, we claim that $f$ also commutes with the $\Delta=\mathrm{Gal}(L'/K)$-action on $\MM$ and $\MM'$. 
  
  In fact, consider the map $f \otimes W(\mathbb C_p^{\flat}):\MM \otimes_{\mathfrak S_{L'}} W(\mathbb C_p^\flat) \to \MM' \otimes_{\mathfrak S_{L'}} W(\mathbb C_p^\flat)$. Since $T(f)=g$, and $g$ commutes with the $G_\infty$-action, $f \otimes W(\mathbb C_p^{\flat})$ also commutes with the $G_\infty$-action when restricted to $T_R(\MM)$. In other words, for any $\sigma\in G_\infty$, $T(\sigma f \sigma^{-1})=T(f)$. Hence $\sigma f \sigma^{-1}=f$. Therefore $f \otimes W(\mathbb C_p^{\flat})$ commutes with the $G_\infty$-action. 
  
  Since $\mathbb C_p^\flat$ is the field of fractions of $\mathcal O_{\mathbb C_p^\flat}$, there is an injection $\mathbf A_{\mathrm{inf}}\to W(\mathbb C_p^\flat)$ after taking the ring of Witt vectors. Therefore, we have inclusions $\MM \subseteq \MMM \subseteq \MMM \otimes_{\mathbf A_{\mathrm{inf}}}W(\mathbb C_p^\flat)$ and $\MM' \subseteq \MMM' \subseteq \MMM' \otimes_{\mathbf A_{\mathrm{inf}}}W(\mathbb C_p^\flat)$. So $f \otimes W(\mathbb C_p^{\flat})$ can be restricted to a map $(\MMM)^{G_\infty'} \to (\MMM')^{G_\infty'}$. Since we have shown that $f \otimes W(\mathbb C_p^{\flat})$ is $G_\infty$-equivariant, this restriction map is $\Delta$-equivariant. Moreover, the maps $\MM \hookrightarrow (\MMM)^{G_\infty'}, \MM' \hookrightarrow (\MMM')^{G_\infty'}$ are both $\Delta$-equivariant by the second axiom in \thref{BK}, hence $f$ is also $\Delta$-equivariant. Hence we have $f \in \mathrm{Hom}_{\mathrm{wMod}^{\varphi, G_K}_{ \AAA} }(\MM, \MM')$, proving full faithfulness of the functor $T$.

  Now we go back to the general $R$-coefficient setting, where $R$ is a finite flat $\mathcal O$-algebra. Let $\MM, \MM' \in \mathrm{wMod}^{\varphi, G_K}_{ \AAA} $, and let $\rho=T_R(\MM), \rho'=T_R(\MM') \in \R$. Suppose $g$ is a morphism between $\rho$ and $\rho'$, then according to the full faithfulness result in the case of $R=\mathbb Z_p$, there exists a unique $\mathfrak S_{L'}$-homomorphism $f:\MM \to \MM'$ such that $T_{\mathbb Z_p}(f)=g$. (Here, finite flatness is used to ensure both $\MM$ and $\MM'$ are of finite rank over $\mathfrak S_{L'}$.) It suffices to show that $f$ is moreover an $\SSSS$-homomorphism. 
  
  To do this, let $m_r$ be the multiplication map by $r$ for any $r \in R$. Note that $$T_R(fm_r)=T_R(f)m_r=gm_r=m_rg=m_rT_R(f)=T_R(m_rf),$$ and using the full faithfulness result in the $\mathbb Z_p$-coefficient case, we have $fm_r=m_rf$, proving the general case.
\end{proof}

\begin{lemma} \thlabel{GvsLLLM}
 Let $R$ be a finite flat $\mathcal O$-algebra, and let $\MM$ be an effective Breuil-Kisin module of rank $n$ with coefficients in $R$. Let $\rho_\infty=T_R(\MM) |_{G_\infty}:G_\infty \to \GL_n(R)$. Then $\rho_\infty$ is isomorphic to the $G_\infty$-representation obtained from $(\MM \otimes_{\SSS} \SSS^{\text{ur}})^{\phi_\MM \otimes \varphi_{\SSS^{\text{ur}}} =1}$. Similarly, $\rho_\infty'=T_R(\ML)|_{G_{\infty}'}:G_\infty' \to \GL_n(R)$ coincides with the $G_\infty'$-representation obtained from $(\ML \otimes_{\SSSS} \SSSS^{\text{ur}})^{\phi_\MM \otimes \varphi_{\SSSS^{\text{ur}}} =1}$.
\end{lemma}

\begin{proof}
    We will be using the associated \'etale $\varphi$-module $\mathcal M=\MM \otimes_{\SSS}\mathcal O_{\mathcal E,R}$. By \cite[Corollary 2.2.2]{Liu2} (the same proof applies to finite flat coefficients), $(\MM \otimes_{\SSS} \SSS^{\text{ur}})^{\phi=1}=(\mathcal M \otimes_{\mathcal O_{\mathcal E,R}} \mathcal O_{\mathcal E,R}^{\mathrm{ur}})^{\phi=1}$. Therefore, it suffices to show that $T_R(\MM)$ is isomorphic to $(\mathcal M \otimes_{\mathcal O_{\mathcal E,R}} \mathcal O_{\mathcal E,R}^{\mathrm{ur}})^{\phi=1}$. 

    Since we have the $\phi$-equivariant isomorphisms $\MMM \otimes_{\AAA} W(\mathbb C_p^\flat)_R \cong \MM \otimes_{\SSS}W(\mathbb C_p^\flat)_R \cong \mathcal M \otimes_{\mathcal O_{\mathcal E,R}} W(\mathbb C_p^\flat)_R$, $T(\MM)=(\mathcal M \otimes_{\mathcal O_{\mathcal E,R}} W(\mathbb C_p^\flat)_R)^{\phi=1}$.

    Note that we have another $\phi$-equivariant isomorphism $\mathcal M \otimes_{\mathcal O_{\mathcal E,R}}\mathcal O_{\mathcal E,R}^{\mathrm{ur}} \cong (\mathcal M\otimes_{\mathcal O_{\mathcal E,R}}\mathcal O_{\mathcal E,R}^{\mathrm{ur}})^{\phi=1} \otimes_R \mathcal O_{\mathcal E,R}^{\mathrm{ur}}$. Base change to $W(\mathbb C_p^\flat)$ and taking $\phi$-invariants, we get $(\mathcal M \otimes_{\mathcal O_{\mathcal E,R}} W(\mathbb C_p^\flat)_R)^{\phi=1}=(\mathcal M\otimes_{\mathcal O_{\mathcal E,R}}\mathcal O_{\mathcal E,R}^{\mathrm{ur}})^{\phi=1} \otimes_R W(\mathbb C_p^\flat)_R^{\phi=1}=(\mathcal M\otimes_{\mathcal O_{\mathcal E,R}}\mathcal O_{\mathcal E,R}^{\mathrm{ur}})^{\phi=1}$. Hence we are done.
    The second assertion follows similarly.
\end{proof}

\begin{rmk} 
    \thref{GvsLLLM} shows that the two ways of forming $G_\infty$-representations are the same: the one used in \cite{LLHLM1}: $(\MM \otimes_{\SSS} \SSS^{\text{ur}})^{\phi_\MM \otimes \varphi_{\SSS^{\text{ur}}} =1}$ and the one used in \cite{Gao}:  $(\MMM \otimes_{\AAA} W(\mathbb C_p^{\flat})_R)^{\phi_{\MMM} \otimes \varphi_{W(\mathbb C_p^{\flat})}=1}$. We will be implicitly using this result throughout. 
\end{rmk}

\begin{defn} \thlabel{BK1}
   A rank $n$ potentially semistable Breuil-Kisin module $(\MM, \phi_\MM)$ with coefficients in $R$ of height at most $h$ and descent data $\tau$ consists of a rank $n$ effective Breuil-Kisin module $(\MM,\phi_\MM)$ over $\SSSS$ such that
\begin{enumerate}
  \item $\MM / u' \subseteq (\widehat \MM /W(\mathfrak m)\widehat \MM)^{G_{L'}}$ via the embedding $\MM / u' \hookrightarrow \widehat \MM /W(\mathfrak m)\widehat \MM$;
  \item $(\MM,\phi_{\MM})$ is a Breuil-Kisin module over $\SSSS$ of height at most $h$ and descent data $\tau$.
\end{enumerate} 
\end{defn}

For any finite flat $\mathcal O$-algebra $R$, let $X_n^{[0,h], \tau}(R)$ denote the groupoid of rank $n$ potentially semistable Breuil-Kisin modules with coefficients in $R$ of height at most $h$ and descent data $\tau$. These potentially semistable Breuil-Kisin modules are the linear algebraic gadgets we will be using to study the stack $\XXX$. 

Given a finite flat $\mathcal O$-algebra $R$, we still denote the following functor by $T_R$: 
$$T_R:\X(R) \ni \MM \mapsto (\MMM \otimes_{\AAA} W(\mathbb C_p^{\flat})_R)^{\phi_{\MMM} \otimes \varphi_{W(\mathbb C_p^{\flat})}=1} \in \R.$$

\begin{lemma}
Let $R$ be any finite flat $\mathcal O$-algebra. Let $\MM \in \X(R)$, then $T_R(\MM) \in \XXX(R)$.   
\end{lemma}

\begin{proof}
Let $\rho:G_K \to \GL_n(R)$ be the Galois representation given by $T_R(\MM)$. We first claim that $\rho |_ {G_{L'}}$ is semistable. In fact, $\MMM$ has continuous $G_{L'}$-action (restricted from the $G_K$-action on $\MMM$), and the action commutes with $\phi_{\MMM_{L'}}$. We also have $\MM \subseteq \MMM ^{G'_\infty}$. Moreover, $\MM / u' \subseteq (\widehat \MM /W(\mathfrak m)\widehat \MM)^{G_{L'}}$, therefore $\rho | G_{L'}$ is semistable by \cite[Theorem 7.1.7]{Gao}. 

Then, since $\MM \in Y^{[0,h], \tau}(R)$, $\rho$ has Hodge-Tate weight bounded by $[0,h]$. According to \cite[Definitions 4.6.1, 4.6.4]{EG}, $\rho$ has inertial type $\tau$. Hence $\rho \in \XXX(R)$.
\end{proof}

Therefore, in what follows, we will regard $T_R:\X(R) \to \XXX(R)$ as a functor with target $\XXX(R)$ by abuse of notation. We want to show that the functor $T_R$ actually gives a categorical equivalence. 
\begin{thm} \thlabel{equivalence}
Let $R$ be a finite flat $\mathcal O$-algebra. Then $T_R$ gives an equivalence of categories $\X(R) \cong \XXX(R)$. Or in other words, the category of rank $n$ potentially semistable Breuil-Kisin modules of height at most $h$ and inertial type $\tau$ is equivalent to the category of potentially semistable $G_K$-representations of Hodge-Tate weights not exceeding $h$ and inertial type $\tau$.
\end{thm}

\begin{proof}
 First, we show that $T_R$ is fully faithful. It is clear that $\XXX(R)$ is a full subcategory of $\R$, and that $\X(R)$ is a full subcategory of $\mathrm{wMod}^{\varphi, G_K}_{\mathfrak S, \AAA}$. From \thref{ff}, we know that $T_R$ is fully faithful.

Now, we show that $T_R$ is essentially surjective. Let $\rho \in \XXX(R)$. By \cite[Proposition 4.8.12]{EG} and \cite[Proposition 4.8.2]{EG}, we know that $\rho$ comes from a Breuil-Kisin-Fargues $G_K$-module $\MM^{\mathrm{inf}}$ of height at most $h$ and inertial type $\tau$, admitting all descents over $L'$, with a canonical $G_{L'}$-action.

By \cite[F.23 Corollary]{EG}, we can see that the corresponding \'etale $(\varphi, G_K)$-module $M$ (see \cite[Definition 2.7.7]{EG}) satisfies $M= \MM^{\mathrm {inf}} \otimes W(\mathbb C_p^\flat)$. Note that the functor $V$ relating $M$ and the Galois representation $\rho$ also just takes the $\varphi$-invariants, we have $T_R(\MM^{\mathrm {{inf}}})=\rho$. Then we need to verify that $\MM^{\mathrm {{inf}}}$ comes from a potentially semistable Breuil-Kisin module in our sense.

Let $\MMM=\MM^{\mathrm {{inf}}}$, and let $\MM$ be the descent of $\MM^{\mathrm {{inf}}}$ to $\SSSS$. We have $\MMM\cong \MM \otimes_{\SSSS}\AAA$.

Next, we give an action of $\Delta$ on $\MM$. For any $g \in \Delta$, $g\MM$ is also a Breuil-Kisin module. Apply $g$ to both sides of $\AAA \otimes_{\SSSS}\MM=\MMM$, we see that $\AAA\otimes_{\SSSS}g(\MM)=\MMM$. This shows that $g(\MM)$ is also a descent to $\SSSS$ of $\MMM$ as in \cite[Definition 4.2.4]{EG}. According to \cite[Lemma 4.2.8]{EG}, $g(\MM)=\MM$ since such a descent should be unique. In this way, we have obtained an action of $\Delta$ on $\MM$. 

Since $\MM^{\mathrm{inf}}$ has descent data $\tau$, according to \cite[4.6.1]{EG} and its previous paragraph, this is to say that the $\Delta$-action obtained from $\MM/u'$ is $\tau$. This coincides with our \thref{BKK1}. Furthermore, since $\rho$ is of Hodge-Tate weight between $0$ and $h$, $\MM \in \Y(R)$. Also, since the $\Delta$-action is directly constructed from the $G_K$-action on $\MMM$, the inclusion $\MM\hookrightarrow\MMM^{G_\infty'}$ is $\Delta$-invariant, with $\Delta$-action on $\MMM^{G_\infty'}$ induced by the $G_\infty$-action.

For any $g\in G_{L'}$, the Breuil-Kisin module $g\MM \subseteq \MMM$ is a descent of $\MMM$ to $g([\underline{\pi'}])$. Since $\MMM$ admits all descents over $L'$, the $W(k') \otimes_{\mathbb Z_p} R$ module $g(\MM)/g(u')$ inside $\MMM/W(\mathfrak m)\MMM$ doesn't depend on $g \in G_{L'}$ according to \cite[Definition 4.2.4]{EG}, hence $\MM/u' \subseteq (\MMM/W(\mathfrak m)\MMM)^{G_{L'}}$.

Hence, we have verified that $\MM \in \X(R)$ and that it satisfies $T_R(\MM)=\rho$, proving essential surjectivity of $T_R$. 
\end{proof}

\begin{rmk}
    From now on, we will also use $\MM$ or $\MMM$ to denote an object of $\XXX(R)$ when it is convenient.
\end{rmk}

\subsection{A na\"ive model for potentially semistable Breuil-Kisin modules}

Now, we want to construct a na\"ive ``local model" $\ST$ for the moduli of potentially semistable Breuil-Kisin modules $\X$. 

Let $R$ be any $\mathcal O$-algebra. Let $\MM \in \X(R)$. Naturally, $\MM \in \mathrm{wMod}^{\varphi, G_K}_{\mathfrak S, \AAA}$, so there is a $G_K$-action on $\MMM$. Now we study the action of the subgroup $G_{L'}$. Since $\MM \hookrightarrow \MMM$ is fixed under the subgroup $\mathrm{Gal}(\overline K/L'_{\mathrm {cyc}}L'_\infty) \subseteq G_{L'}$ by condition 2 in \thref{BK}, we obtain a map $g:\MM \to \MMM$ for every element $g \in \mathrm{Gal}(L'_{\mathrm {cyc}}L'_\infty/K)$. Note that $\mathrm{Gal}(L'_{\mathrm {cyc}}L'_\infty/L')$ is topologically generated by $\psi \in \mathrm{Gal}(L'_{\mathrm {cyc}}L'_\infty/L'_{\mathrm {cyc}}) $ and $\mathrm{Gal}(L'_{\mathrm {cyc}}L'_\infty/L'_\infty)$ since $L'_{\mathrm {cyc}} \cap L'_\infty=L'$ when $p>2$, the latter subgroup fixes $\MM$ (since every element in this Galois group fixes the field $L'_\infty$ pointwise and $\MM$ is fixed by $G_\infty'$), so the information of the $G_{L'}$-action on $\MMM$ is equivalent to a single map $\psi: \MM \to \MMM$. 

Suppose that $\beta$ is an eigenbasis of $\MM$. Write $\beta_i=(\beta_i^{(j)})_{j \in \mathcal J} \in \MM_i \subseteq \MM$. We also regard $\beta_i$ as elements of $\MMM$. Then $\MMM$ is also a free module of rank $n$ generated by $\beta_i$ for $i=1,\cdots,n$ over $\AAA$. Let $\Psi _{\MM,\beta}\in \Ma(\AAA)$ be the matrix of $\psi \otimes \AAA:\MMM \to \MMM$ under the basis $\beta$.

\begin{defn} \thlabel{Galois}
 Using the above notations, we call the matrix $\Psi=\Psi _{\MM,\beta}$ the matrix of Galois of $\MM \in \X(R)$ under the basis $\beta$.   
\end{defn}

Similar to \thref{change basis}, we also need a base change result of the matrix of Galois in order to construct our local model.

\begin{prop} \thlabel{basis change}
 Let $\MM \in \X(R)$. Let $\beta, \beta'$ be two eigenbases of $\MM$ related by $(\beta')^{(j')} D^{(j')}=\beta^{(j')}$ such that $D^{(j')} \in \mathrm{GL}_{n}(R \uuu)$ and $D^{(j')}=D^{(j'+f)}$ for $j' \in \mathcal J'$. Then the matrices of Galois satisfy the relation
$$\Psi_{\MM, \mathfrak \beta'}=D \Psi_{\MM, \mathfrak \beta} \psi(D)^{-1},$$ where $D=\mathfrak C((D^{(j)})_{j \in \mathcal J})$.  
\end{prop}

\begin{proof}
 We use the fact that $\psi$ is an $\AAA$-semilinear map on $\MMM$.
\end{proof}
Before giving our first version of the model, we need the following notation. Since $L'_{\mathrm {cyc}} \cap L'_\infty=L'$, $\mathrm{Gal}(L'_{\mathrm {cyc}}L'_\infty/L'_{\infty})\cong \mathrm{Gal}(L'_{\mathrm {cyc}}/L') \cong \mathbb Z_p ^\times$.  For any $g\in \mathrm{Gal}(L'_{\mathrm {cyc}}L'_\infty/L'_{\infty})$, let $k_g\in \mathbb Z_p^\times$ satisfy the identity $g(\zeta_{p^i})=(\zeta_{p^i})^{k_g \mathrm{\ mod\ }p^i}$ for all $i=1,2,\cdots$. For any $m \in \mathbb Z_p^\times$, let $g_m \in \mathrm{Gal}(L'_{\mathrm {cyc}}L'_\infty/L'_{\infty})$ be the element that satisfies $g_m(\zeta_{p^i})=(\zeta_{p^i})^{m \mathrm{\ mod\ }p^i}$ for all $i=1,2,\cdots$.

\begin{defn} \thlabel{local model}
 Let $\ST$ be the functor from $\mathcal O$-algebras to groupoids, sending any $\mathcal O$-algebra $R$ to the groupoid of pairs $((A^{(j)})_{j \in \mathcal J},\Psi)$ where

\begin{enumerate} 
    \item $A^{(j)}, (v+p)^h (A^{(j)})^{-1} \in \mathcal P_j(R)$ for all $j \in \mathcal J$;
    \item $\Psi \in \GL_n(\AAA)$, and $C\varphi(\Psi)=\Psi \psi(C)$, where $C=\mathfrak C((A^{(j)} \omega^{(j)})_{j \in \mathcal J})$;
    \item $\Psi^{}- \mathrm{id} \in u'\Ma(\AAA)$;
    \item $\Psi$ is fixed by $\mathrm{Gal}(\overline K/L'_{\mathrm{cyc}}L'_\infty)$. And for any positive integer $m$ which is prime to $p$, $g_m(\Psi)=\Psi^{(m)}$, where $\Psi^{(m)}$ is defined recursively as $\Psi^{(1)}=\Psi$, $\Psi^{(i+1)}=\Psi\psi(\Psi^{(i)})$ for all $i \ge 1$. Here $g_m \in \mathrm{Gal}(L'_{\mathrm{cyc}}L'_\infty/L'_\infty)$ is regarded as any lift of it to $G_K$. \footnote{The choice of which does not matter since $\Psi$ is fixed by $\mathrm{Gal}(\overline K/L'_{\mathrm{cyc}}L'_\infty)$.}
\end{enumerate}
\end{defn}

Using \thref{change basis} and \thref{basis change}, we define an action of $\prod_{j \in \mathcal J} \mathcal  P_j$ on $\ST$ as follows: for any $(I^{(j)})_{j \in \mathcal J} \in \prod_{j \in \mathcal J} \mathcal  P_j(R)$ and $((A^{(j)})_{j \in \mathcal J},\Psi)\in \ST(R)$, let $$(I^{(j)})_j \cdot ((A^{(j)})_j,\Psi))=((I^{(j)}A^{(j)} \mathrm{Ad}(\omega^{(j)}))(\varphi(I^{(j-1)})^{-1})_j,D \Psi \psi(D)^{-1}),$$ where $I=\mathfrak C((I^{(j)})_{j \in \mathcal J})$ and $D=\mathfrak C((D^{(j)})_{j \in \mathcal J})$ with $I^{(j)}=\mathrm{Ad}((\ssj)^{-1}(u')^{-\aaa^{j'}}) D^{(j)}$.

We now want to prove an equivalence between this local model and our $\X$. We will need the following lemma in the proof.

\begin{lemma} \thlabel{act}
    Let $G$ be a group and $H,K$ be two subgroups of $G$. Suppose that $K$ is normal and $HK=G$. Let $X$ be a set where both groups $H$ and $K$ act. Suppose the action of $H$ and $K$ agrees on $H \cap K$. If for any $x \in X, h \in H, k\in K$, we have $h( k( h^{-1} x))=(hkh^{-1})x$ \footnote{Since $K$ is normal, $hkh^{-1} \in K$, hence the right hand side $(hkh^{-1})x$ is well-defined.} , then there is an action of $G$ on $X$ extending both actions of $H$ and $K$.
\end{lemma}

\begin{proof}
    We give an action as follows: for any $g\in G$, write $g=hk$ for some $h \in H, k \in K$, let $gx=h(kx)$ for any $x \in X$. 

    We verify that this really gives an action. The well-definedness is obvious: given $g=hk=h'k'$ with $h,h'\in H,k,k' \in K$, then $h'^{-1}h=k'k^{-1}\in H \cap K$. So $h'^{-1}hy=k'k^{-1}y$ for all $y \in X$ since two actions agree on the intersection. Let $y=kx$, we have $(hk)x=(h'k')x$. The only thing remaining to show is that $(g_1g_2)x=g_1(g_2x)$ for all $g_1,g_2 \in G$ and $x \in X$. 
    
    We first prove this statement for $g_1=k \in K$ and $g_2=h\in H$. Since $K$ is a normal subgroup of $G$, $k':=h^{-1}kh$ is also an element of $K$. By our assumption, $(hk'h^{-1})y=h(k'(h^{-1}y))$ for any $y \in X$. Hence $$k(hx)=(hk'h^{-1})(hx)=h(k'(h^{-1}(hx)))=h(k'(x))=(hk')(x)=(kh)x,$$ for all $x\in X$.
    
    In general, write $g_i=h_ik_i$, with $i=1,2$. Then $g_1(g_2x)=(h_1k_1)((h_2k_2)x)=h_1(k_1((h_2k_2)x))$. Since $K$ is a normal subgroup of $G$, $k_2':=h_2k_2h_2^{-1} \in K$, then $h_1(k_1((h_2k_2)x))=h_1(k_1((k_2'h_2)x))$. We have proved the property $(k_2'h_2)x=k_2'(h_2x)$, hence $h_1(k_1((k_2'h_2)x))=h_1(k_1(k_2'(h_2x))=h_1((k_1k_2')(h_2x))$. Using this property again, $h_1((k_1k_2')(h_2x))=h_1((k_1k_2'h_2x))=h_1(k_1h_2k_2x)$. Repeat this process for $k_1\in K$, using the normality of $K\subseteq G$, we obtain $h_1(k_1h_2k_2x)=(h_1k_1h_2k_2)x=(g_1g_2)x$. Putting things together, we proved that $(g_1g_2)x=g_1(g_2x)$. And this shows that we indeed have extended the action to $G$.
\end{proof}

\begin{thm} \thlabel{big2}
Let $R$ be a $p$-adically complete, topologically of finite type, $\mathcal O$-flat algebra. The groupoid $\ST(R)/ \prod_{j \in \mathcal J} \mathcal P_j(R)$ is isomorphic to $\X(R)$, the groupoid of potentially semistable Breuil-Kisin modules with coefficients in $R$. 
\end{thm}

\begin{proof}
Given any $\MM \in \X(R)$, we form $\MMM$ as usual. Locally, we can find an eigenbasis $\beta=(\beta^{(j')})_{j' \in \mathcal J'}$ and set $(A^{(j)})_{j \in \mathcal J}$ to be the matrix of partial Frobenii. Thus $A^{(j)}, (v+p)^h (A^{(j)})^{-1} \in \mathcal P_j(R)$ for all $j \in \mathcal J$ by \thref{Frob}. Let $\Psi$ be the matrix of Galois under basis $\beta=(\beta_1,\cdots,\beta_n)$. Since the $G_K$-action on $\MMM$ commutes with $\phi_{\MMM}$, we have $C\phi(\Psi)=\Psi \psi(C)$, where $C=\mathfrak C((A^{(j)}\omega^{(j)}))_{j \in \mathcal J})$.

Since $\MM$ is fixed under $G'_\infty$, we know that the information of the $G_{L'}$-action on $\MMM$ is equivalent to the map $\psi: \MM \to \MMM=\MMM$. Since $\MM/u'$ is fixed under $G_{L'}$, we have $\Psi- \mathrm{id} \in u'\Ma(\AAA)$. Now to prove that $((A^{(j)})_{j \in \mathcal J},\Psi)\in \ST(R)$, it suffices to verify condition (4).

In fact, since $$g\psi g^{-1}(\pi_i')=g\psi(\pi_i')=g(\zeta_{p^i}\pi_i')=g(\zeta_{p^i})\pi_i'=\zeta_{p^i}^{k_g}\pi_i',$$ the element $g \psi g^{-1} \in \mathrm{Gal}(L'_{\mathrm {cyc}}L'_\infty/L'_{\mathrm {cyc}})$ is equal to $\psi^{k_g}$. Or in other words, $g_m \psi g_m^{-1}=\psi^m$. Hence for every $i$, we have $(g_m \psi g_m^{-1})(\beta_i)=\psi^m(\beta_i)$. On the other hand, note that $g_m \in \mathrm{Gal}(L'_\infty L'_{\mathrm{cyc}}/L'_\infty)$ fixes $\MM$, hence $(g_m \psi g_m^{-1})(\beta_i)=g_m (\psi(\beta_i))$. Therefore, we have $g_m (\psi(\beta_i))=\psi^m(\beta_i)$, whose matrix form is $g_m(\Psi)=\Psi^{(m)}$. Note that the $g_m$ in the left hand side of this equation refers to any lift of $g_m \in \mathrm{Gal}(L'_\infty L'_{\mathrm{cyc}}/L'_\infty)$ to $G_K$, but the right hand side doesn't depend on such a lift. Therefore, the subgroup $\mathrm{Gal}(\overline K/L'_{\mathrm{cyc}}L'_\infty)$ fixes all entries of the matrix $\Psi$. Hence, we have obtained a pair $((A^{(j)})_{j \in \mathcal J},\Psi)$ in $\ST(R)$ from a given $\MM \in \X(R)$.

Conversely, given a pair $((A^{(j)})_{j \in \mathcal J},\Psi)$, recall from \thref{frob}, we can construct a corresponding Breuil-Kisin module $\MM$ over $\SSSS$ with the matrix of partial Frobenii $(A^{j})_{j \in \mathcal J}$ under a chosen eigenbasis $\beta=(\beta_1,\cdots,\beta_n)$. Let $\Delta$ act on $\beta_i$ by $\chi_i$. This gives $\MM$ the inertial type $\tau$. The conditions $A^{(j)}, (v+p)^h (A^{(j)})^{-1} \in \mathcal P_j(R)$ for all $j \in \mathcal J$ ensure that $\MM$ is of height bounded by $0$ and $h$. In particular, the Frobenius map $\phi$ on $\MM$ determined by the matrices $(A^j)_{j \in \mathcal J}$ commutes with the $\Delta$-action.

We then give a semilinear map $\psi: \MMM \to \MMM$ so that its matrix under basis $\beta$ is $\Psi$. Since $\psi$ is a topological generator of the group $\mathrm{Gal}(L'_{\mathrm {cyc}}L'_\infty/L'_{\mathrm {cyc}}) \cong \mathbb Z_p$, to show that we can extend $\psi$ to an action of $\mathrm{Gal}(L'_{\mathrm {cyc}}L'_\infty/L'_{\mathrm {cyc}})$ on $\MMM$, it suffices to show that $\Psi^{(p^n)} \to \mathrm{id}$ $p$-adically as $n \to \infty$. In fact, from condition (4), we have $\Psi^{(p^n+1)}=g_{p^n+1}(\Psi)$. As elements of the topological group $\mathrm{Gal}(L'_{\mathrm {cyc}}L'_\infty/L'_\infty)$, $g_{p^n+1}\to 1$ as $n \to \infty$. Hence $g_{p^n+1}(\Psi) \to \Psi$ since the action of $G_K$ on $\mathbf{A}_{\mathrm{inf}}$ is continuous. Therefore $\psi^{p^n+1} \to \psi$, i.e. $\psi^{p^n} \to \mathrm{id}$.

Let $\mathrm{Gal}(L'_{\mathrm {cyc}}L'_\infty/L'_{\infty})$ act on $\MM$ trivially, and extend this action semilinearly to $\MMM$. Now we need to use \thref{act} to verify that the actions of the two groups $\mathrm{Gal}(L'_{\mathrm {cyc}}L'_\infty/L'_{\mathrm {cyc}})$ and $\mathrm{Gal}(L'_{\mathrm {cyc}}L'_\infty/L'_\infty)$ can glue to an action of $\mathrm{Gal}(L'_{\mathrm {cyc}}L'_\infty/L')$. That is, for any $g \in \mathrm{Gal}(L'_{\mathrm {cyc}}L'_\infty/L'_{\infty})$, we need to show that $g(\psi(g^{-1}(\beta_i)))=(g\psi g^{-1})(\beta_i)$. 

In fact, we have shown that the element $g \psi g^{-1} \in \mathrm{Gal}(L'_{\mathrm {cyc}}L'_\infty/L'_{\mathrm {cyc}})$ is equal to $\psi^{k_g}$. And since $g \in \mathrm{Gal}(L'_{\mathrm {cyc}}L'_\infty/L'_{\infty})$, $g$ fixes $\beta$. Therefore, $g(\psi(g^{-1}(\beta_i)))=g(\psi(\beta_i))$. Let $g=g_m$ for any positive integer $m$ that is prime to $p$, condition (4) in \thref{local model} gives $\psi^{k_g}(\beta_i)=g(\psi(\beta_i))$. Since the positive integers prime to $p$ are dense in $\mathbb Z_p^\times$ and we have shown that $\psi^{p^n}\to 1$ as $n \to \infty$, we have verified the identity $g(\psi(g^{-1}(\beta_i)))=(g\psi g^{-1})(\beta_i)$.

Let $G_{L'}$ act on $\MMM$ through the quotient $G_{L'}\twoheadrightarrow\mathrm{Gal}(L'_{\mathrm {cyc}}L'_\infty/L')$. 

Next, we need to extend the $G_{L'}$-action on $\MMM$ to a $G_K$-action. We let $G_\infty$ act on $\MM$ through the quotient $G_\infty \twoheadrightarrow \Delta$, with the $\Delta$-action already given by the inertial type $\tau$ on $\beta_i$. Applying \thref{check}, which we put below, and \thref{act}, we can now give a $G_K$-action on $\MMM$ extending the action of $G_{L'}$ and $G_\infty$.

The inclusion $\MM \hookrightarrow \MMM^{G_\infty'}$ is $\Delta$-equivariant since the $G_\infty$-action on the right is defined through the quotient $\Delta$.

Condition (2) in \thref{local model} ensures commutativity of $\phi$ and the $G_{L'}$-action. Since the $G_\infty$-action factors through $\Delta$ and we already know that $\phi$ commutes with the $\Delta$-action, $\phi$ commutes with the $G_\infty$-action. Hence, the Frobenius $\phi$ commutes with the $G_K$-action. 

Since $\mathrm{Gal}(L'_{\infty}L'_{\mathrm{cyc}}/L'_{\mathrm{cyc}})$ is topologically generated by $\psi$, condition (3) ensures that $\MM/u'$ is fixed under $G_{L'}$. Therefore, we have constructed $\MM \in \X(R)$ whose matrices of partial Frobenius and Galois are the given matrices $(A^{(j)})_{j \in \mathcal J}$ and $\Psi$. 

\thref{change basis} and \thref{basis change} show how the change of eigenbases would affect the data $((A^{(j)})_{j \in \mathcal J},\Psi)$, therefore taking the quotient by parahoric subgroups will give the desired isomorphism statement.
\end{proof}

\begin{lemma} \thlabel{check}
     For any $g \in G_\infty$ and $m \in \MM$, we have $g(\psi(g^{-1}m))=(g\psi g^{-1})m$.
\end{lemma}
\begin{proof}
 So far, our Breuil-Kisin module $\MM$ has a $G_\infty$-action. We will construct two other Breuil-Kisin modules $\MM_1$ and $\MM_2$ from here. We first prove this result in the case where $R$ is finite flat, so that we can use the properties of $T_R$.
 
 Let $\MM_1$ be the module $\MM\in\mathrm{wMod}^{\varphi, G_{L'}}_{\mathfrak S, \AAA}$ with $\mathrm{Gal}(L'_\infty L'_{\mathrm{cyc}}/L'_{\mathrm{cyc}})$-action determined by $\mathrm{Gal}(L'_\infty L'_{\mathrm{cyc}}/L'_{\mathrm{cyc}}) \ni \psi:m \mapsto (g\psi g^{-1})m,$ and with the $G_{L'}$-action factoring through $G_\infty'$. The category $\mathrm{wMod}^{\varphi, G_{L'}}_{\mathfrak S, \AAA}$ is given in \cite[Definition 7.1.3]{Gao}. Let $\MM_2$ be the module $\MM\in\mathrm{wMod}^{\varphi, G_{L'}}_{\mathfrak S, \AAA}$ given similarly by $\mathrm{Gal}(L'_\infty L'_{\mathrm{cyc}}/L'_{\mathrm{cyc}})\ni \psi:m \mapsto g(\psi(g^{-1}m))$.
 
 Consider the map $\gamma_g:\MM \to \MM$ given by $m \mapsto g(m)$ for all $m \in \MM$. We can regard $\gamma_g$ as a morphism in the category $\mathrm{wMod}^{\varphi, G_{L'}}_{\mathfrak S, \AAA}$. The $G_{L'}$-action on the target of $\gamma_g$ is given by $m \mapsto (g\psi g^{-1})m$ for all $m \in \MM$. 

Let $\rho:G_{L'} \to \mathrm{GL}_n(R)$ be the representation $T_R(\MM)$. Consider the following identity $$T_R(\gamma_g(\MM))=T_R(\gamma_g)(T_R(\MM)).$$ Both sides can also be viewed as $G_{L'}$-representations by the definition of $T_R$. On the left hand side, the representation is given by $h \mapsto \rho(ghg^{-1})$ for all $h \in G_{L'}$. It is just $T_R(\MM_1)$. On the right hand side, the representation is given by $h \mapsto A\rho(h) A^{-1}$, where $A=T_R(\gamma_g) \in \mathrm{GL}_n(R)$ when written as a matrix. Hence this representation is $T_R(\MM_2)$. We have $T_R(\MM_1)=T_R(\MM_2)$. By the full faithfulness of $T_R$, we have $\MM_1=\MM_2$. In particular, the $\mathrm{Gal}(L'_\infty L'_{\mathrm{cyc}}/L'_{\mathrm{cyc}})$-actions on $\MM_1=\MM_2=\MM$ coincide.

Now let's turn to a general coefficient ring $R$, which is a $p$-adically complete, topologically of finite type, $\mathcal O$-flat algebra. Fix an eigenbasis $\beta$ of $\MM$ and consider the matrix representing the semilinear map $g \circ\psi\circ g^{-1}-g\psi g$ on $\MMM$. Let $x \in \AAA$ be any entry of the matrix. Given any finite flat $\mathcal O$-algebra $R'$, consider the base change of $\MM$ to $R'$. Then we use the previous argument for the finite flat case and conclude that the image of $x$ is $0$ in $\mathbf A_{\text{inf},R'}$. By the following \thref{aux1}, $x=0$. We are done.
\end{proof}

\begin{lemma} \thlabel{aux1} 
 Let $R$ be a $p$-adically complete, topologically of finite type, $\mathcal O$-flat algebra. Let $I \subseteq R$ be an ideal. Assume that for every finite flat $\mathcal O$-algebra $R'$ and every $\mathcal O$-algebra map $R \to R'$, $I_R R'$ is the zero ideal in $R'$. Then $I=0$.
\end{lemma}

\begin{proof}
    Suppose that $I \neq 0$. Take any nonzero element $x \in I$. Let $A=R[1/p]$. Since $R$ is flat over $\mathcal O$, we have an inclusion $R \subseteq A$ and $I \subseteq I[1/p]$. Since $A$ is an affinoid $E$-algebra, it is Jacobson. So we can pick a maximal ideal $\mathfrak m$ of $A$ such that $x\ne 0$ in the localization $A_\mathfrak m$. By Krull's intersection theorem $\bigcap_{i \ge 0} \mathfrak m^iA_{\mathfrak m}=0$. Let $N$ be a positive integer satisfying $x \notin \mathfrak m^N$. Let $B=A/\mathfrak m^N$. Since $A/\mathfrak m$ is a finite dimensional $E$-vector space by Nullstellensatz, so is $B$.
    
    Since $R$ is Noetherian, $I$ is finitely generated. By the construction, $I[1/p]\otimes_A B \ne 0$. Let $R'$ be the sub-$\mathcal O$-algebra of $B$ generated by the image of $R$ under the surjection $A \twoheadrightarrow B$. Since $R$ is topologically of finite type, $R'$ is generated by finitely many power-bounded elements. Hence these generators are integral over $\mathcal O$. Therefore, $R'$ is a finite flat $\mathcal O$-algebra. By the construction, the image of $x$ in $B$ is nonzero. So, we have $IR' \ne 0$, contradicting our assumption!
\end{proof}

\section{The stack of potentially semistable Galois representations} \label{s4}

In this section, we use our definition of potentially semistable Breuil-Kisin modules to build a stack. We will give it some ``algebraic" descriptions. We will in the end show that this stack is isomorphic to the potentially semistable Emerton-Gee stack.
\subsection{Relation with \texorpdfstring{$(\phi, N, \Gamma)$}{Lg}-modules} \label{Relation}

Let $\Gamma$ be either the group $\Delta$ or $\Delta'$. 

Recall that in \thref{{phi,n,gammma}}, we have given the definition of $(\phi,N,\Gamma)$-modules with rational coefficients. We now give a similar definition where the coefficient rings are integral.

\begin{defn} 
Let $R$ be an $\mathcal O$-algebra, we say that $(M,\phi_M, N_M)$ is a $(\phi, N, \Gamma)$-module of rank $n$ with coefficients in $R$ if
 \begin{enumerate}
     \item $M$ is a free $W(k') \otimes_{\mathbb Z_p} R$-module of rank $n$;
     \item $\phi_M:\varphi^*(M) \to M$ is a $W(k') \otimes_{\mathbb Z_p} R$-linear injective homomorphism;
    \item $N_M \in \mathrm{End}_{W(k') \otimes_{\mathbb Z_p} R}(M)$ satisfies $ N\phi_M=p\phi_M N_M$;
    \item $M$ has a $\Gamma$-semilinear action, which commutes with $\phi_M, N_M$.  
 \end{enumerate}
 We denote this category by $L(\phi, N, \Gamma)_R$.
\end{defn}

Following \cite[Definition 2.2]{Liu0}, we have the following definition.

\begin{defn}
 Let $A$ be an $E$-algebra and let $R$ be a finite flat $\mathcal O$-algebra such that $R[1/p]=A$. Let $D$ (resp. $M$) be a $(\phi,N,\Gamma)$-module of rank $n$ with coefficients in $A$ (resp. $R$). We call $M$ a lattice in $D$ with coefficients in $R$ if $M \subseteq D$, and the $\phi, N, \Gamma$-actions on $M$ are all restricted from the corresponding ones on $D$.
\end{defn}

Using the decomposition $K' \otimes_{\mathbb Q_p} A=\bigoplus_{j' \in \mathcal J'}A$, we can write any $(\phi, N, \Gamma)$-module $D$ with coefficients in $A$ as $D=\bigoplus_{j' \in \mathcal J'} D ^{(j')}$, with each $D^{(j')}$ an $A$-module. The map $\phi_D$ and $N_D$ induce $A$-linear maps $\phi_D^{(j')}:D^{(j'-1)} \to D^{(j')}$ and $N_D^{(j')}:D^{(j')} \to D^{(j')}$, satisfying the relation $N_D^{(j')}\phi_D^{(j')}=p\phi_D^{(j')}N_D^{(j'-1)}$. 

Let $R$ be a finite flat $\mathcal O$-algebra, let $\rho:G_K \to \mathrm{GL}_n(R)$ be a potentially semistable representation in $\XXX(R)$. We also use $\rho$ to refer to $\rho[1/p]:G_K \to \mathrm{GL}_n(R[1/p])$. Let $(D,\phi_D,N_D)=D_{\mathrm{pst}}(\rho | _{G_{K'}})$, then $D$ is a $(\phi, N, \Delta')$-module of rank $n$ with coefficients in $R[1/p]$. Since $\Delta'$ fixes $K'$, the action of $\Delta'$ on $D$ is in fact $K' \otimes_{\mathbb Q_p} R[1/p]$-linear. Hence each $D^{(j')}$ is stable under $\Delta'$.

The $(\phi, N, \Delta)$-module $D_{\mathrm{pst}}(\rho)$ has the same underlying $(\phi, N, \Delta')$-module structure as $D_{\mathrm{pst}}(\rho | _{G_{K'}})$, with the $\Delta=\Delta' \rtimes \langle \Fr \rangle$-action given by the $\Delta'$-action on $D$ and $\Fr^{(j')}:D^{(j')}\to D^{(j'+f)}$, an $R[1/p]$-linear isomorphism for each $j' \in \mathcal J'$.

Note that with such a potentially semistable representation $\rho: G_K \to \GL_n(R)$, we can also form a potentially semistable Breuil-Kisin module $\MM$. The relation between $\MM$ and the $(\phi, N, \Delta)$-module $D=D_{\mathrm{pst}}(\rho)$ is as follows.

\begin{prop} \thlabel{DvsM}
    Let $R$ be a finite flat $\mathcal O$-algebra, and let $\rho \in \XXX(R)$. Let $\MM$ be the potentially semistable Breuil-Kisin module such that $T_R(\MM)=\rho$. Let $D=D_{\mathrm{pst}}(\rho)$. Then $D\cong (\MM /u')[\frac{1}{p}] \cong (\varphi^* \MM /u')[\frac{1}{p}]$, with $\Delta$ and $\phi$ acting equivariantly.
\end{prop}

\begin{proof}
   \thref{GvsLLLM} shows that $\MM$ obtained above is the same as the Breuil-Kisin module over $\SSSS$ in \cite{EG}. Then \cite[Corollary F.25]{EG} gives the result.
\end{proof}

Now we introduce some theory of Breuil modules for future constructions. Given any potentially semistable representation $\rho: G_K \to \GL_n(R)$ with coefficients in a finite flat $\mathcal O$-algebra $R$, let $\MM$ be the associated potentially semistable Breuil-Kisin module. We call $\mathcal D := \s \otimes_{\varphi, \SSSS} \MM$ the Breuil module associated to $\rho$ (or $\MM$). Let $D=\mathcal D/I^+_{\mathcal S_R[1/p]} \mathcal D$.
 
It is well-known that there is a unique $\phi,N$-equivariant section $s:D \hookrightarrow \mathcal D$ by \cite[Proposition 6.2.1.1]{Breuil} (the same proof extends to any coefficient ring that is $p$-torsion-free). Hence we have a $\phi,N$-equivariant isomorphism $\mathcal D \cong \s \otimes_{K'} D$.

Now we will recall how \cite{Liu0} constructs a $\phi_D, N_D, \Delta'$-stable $W(k') \otimes_{\mathbb Z_p} R$-lattice $M$ inside the $(\phi, N, \Delta')$-module $D$. We use the following commutative diagram to illustrate the process. Note that \cite[Proposition 2.6]{Liu0} explains the isomorphism between $D$ and $D_{\mathrm{pst}}(\rho)$, giving rise to the monodromy operator $N$ on $M$.
\begin{equation} \label{con}
    \begin{tikzcd}
 \MM & \mathcal S_R \otimes_{\varphi,\SSSS}\MM & \mathcal D=\s \otimes_{\varphi,\SSSS}\MM \\
	M & & D & D_{\mathrm{pst}}(\rho).
	\arrow[twoheadrightarrow,"\mathrm{mod \ } u'",swap, from=1-1, to=2-1]
	\arrow[hook, from=1-1, to=1-2]
    \arrow[hook, from=1-2, to=1-3]
	\arrow[twoheadrightarrow,"\mathrm{mod \ } I^+_{\s}", from=1-3, to=2-3]
	\arrow[hook, from=2-1, to=2-3] 
    \arrow[from=2-3, to=2-4, "\cong"]
\end{tikzcd} 
\end{equation}

In fact, \cite{Liu0} uses the $(\varphi, \hat G)$-module $\MM$. But \cite[Theorem 2.5 (3)]{Liu0} ascertains that it is equivalent to the usual Breuil-Kisin module (for example, those in \cite{LLHLM1}). And hence by \thref{GvsLLLM}, we can indeed use our potentially semistable Breuil-Kisin module to replace it.

\begin{prop} \thlabel{stable}
   Let $R$ be a finite flat $\mathcal O$-algebra. Let $\MM \in \X(R), \rho=T_R(\MM) \in \XXX(R), D=D_{\mathrm{pst}}(\rho)$. Then the $W(k') \otimes_{\mathbb Z_p} R$-module $M$ constructed above is a lattice in the $(\phi,N, \Delta)$-module $D$. Moreover, the association $\XXX(R) \ni \rho \mapsto M \in  L(\phi, N, \Delta)_R$ defines a faithful functor $M_R:\XXX(R) \to  L(\phi, N, \Delta)_R$.
\end{prop}

\begin{proof}
   The argument in \cite[Theorem 2.3]{Liu0} almost gives the proof except that the setting of \cite{Liu0} requires $\rho$ to be semistable over a totally ramified extension (hence we may only apply the results to $\rho|_{G_{K'}}$). We now show how to adapt this to our situation. For the lattice statement, it suffices to further verify that $M$ is $\Delta$-stable. This is clear from the construction $M= \MM /u'$ and the isomorphism in \thref{DvsM}. For the faithfulness of $M_R$, note that the restriction $\rho \mapsto \rho|_{G_{K'}}$ is faithful, hence after composing with the faithful functor $M_{\mathrm{st}}$ in \cite[Theorem 2.3]{Liu0}, we still get a faithful functor. 
   
   We argue in exactly the same way as the proof of \thref{ff} to extend the coefficient ring to any finite flat $\mathcal O$-algebra $R$.
\end{proof}

We record a lemma that will be used later.

\begin{lemma} \thlabel{key}
Let $R$ be a flat $\mathcal O$-algebra, we have a unique $\phi$-equivariant injection $\alpha: \A \otimes_{\varphi,\SSSS}\MM \isoto \AC \otimes_{W(k')} D$ such that for any $m \in M \subseteq D$, $\alpha^{-1}(1 \otimes m) \in \A \otimes_{\varphi,\SSSS}\MM$ maps to $m \in M$ under the reduction-modulo-$u'$ map $ \MM \twoheadrightarrow M $.
\end{lemma}

\begin{proof}
We still make use of the Breuil module $\mathcal D=\s \otimes_{\varphi,\SSSS}\MM $ associated to $\MM$. There is a $\phi$-equivariant injection $\A \otimes_{\varphi,\SSSS}\MM \hookrightarrow \A \otimes_{\mathcal S_R} \mathcal D$. The unique $\phi$-equivariant section $s:D \hookrightarrow \mathcal D$ induces a $\phi$-equivariant isomorphism $\mathcal D \cong \s \otimes_{K'} D$, which also induces a $\phi$-equivariant isomorphism $\A \otimes_{\varphi, \mathcal S_R} \mathcal D \isoto \AC \otimes_{K'} D $. Hence we have this $\phi$-equivariant injection $$\A \otimes_{\varphi,\SSSS}\MM \hookrightarrow \AC \otimes_{K'} D.$$ 

Since $s$ is a section to $\mathcal D \twoheadrightarrow D \cong \mathcal D / I^+_{\s} \mathcal D$, for any $m \in M \subseteq D$, $s(m) \mathrm{\ mod\ } I^+_{\s} =m$. Using the construction of $\alpha$, we can see that $\alpha^{-1}(1 \otimes m) \mathrm{\ mod \ } u' =1 \otimes m \in \AC \otimes_{W(k')} M$ for all $m \in M$.    

The uniqueness of $\alpha$ also follows from the uniqueness of the section $s$.
\end{proof}

Now, we will study how to associate the endomorphism $N$ on the $(\phi, N, \Delta)$-module $D$ (or $M$) from the corresponding $\MM$. 

Recall that given $\rho \in \XXX(R)$, we have associated to it a potentially semistable Breuil-Kisin module $\MM$, and a lattice $M$ inside the $(\phi,N,\Delta)$-module $D$. \thref{key} tells us that there is a $\phi$-equivariant inclusion $\MMM \hookrightarrow \A \otimes_{\varphi,\SSSS}\MM \hookrightarrow \AC \otimes_{W(k')} D$. There is a natural $G_\infty'$-action on $\A \otimes_{W(k')} D$ induced from the $G_\infty'$-action on $\A$. We can see that this inclusion is also $G_\infty'$-equivariant since $\MMM=\AAA \otimes_{\SSSS} \ML $ has trivial $G_\infty'$-action on $\ML$ according to \thref{BK}. 

Now we consider the $G_{L'}$-action on $\MMM$ and extend it semilinearly to $\AC \otimes_{W(k')} M$ using the chain of isomorphisms $$\AC\otimes_{\mathbf A_{\mathrm{inf}},\varphi}\MMM \cong \A \otimes_{\SSSS, \varphi }\MM \cong\AC \otimes_{W(k')} M.$$ It is clear that the $G_{L'}$-action on $\AC \otimes_{W(k')} M$ is given by the map $\psi:M \to \AC\otimes_{W(k')} M$ since $M$ is fixed by $G_\infty'$ and the group $\mathrm{Gal}(L'_{\infty}L'_{\mathrm{cyc}}/L'_{\mathrm{cyc}})$ is topologically generated by $\psi$. The next proposition relates this Galois action with the operator $N$ on $M$.

\begin{thm} \thlabel{exp} \thlabel{vip}
   Using the notations from the previous passage, consider the inclusion $\MMM \hookrightarrow \AC \otimes_{W(k')} D$. For any $m \in \MM$, we write $1 \otimes m$ for its image in  $\AC \otimes_{W(k')} D$. Then $$\psi(m)=\mathrm{exp}(t \otimes N)(m),$$ where $\mathrm{exp}$ is defined by its usual Taylor series, i.e. $\mathrm{exp}(x)=\sum_{i=0}^\infty x^i/i!$.
\end{thm}

\begin{proof}
  We can give a $G_{L'}$-action on $\A \otimes_{\mathcal S_R} \mathcal D$ through $N_{\mathcal D}=N_{\s}\otimes N_D$ using the formula \cite[(5.1.1)]{Liu1}:
   $$g(a \otimes x)= \sum_{i=0}^{\infty} g(a) \frac{(\mathrm{log \ } (\underline{\epsilon}(g)))^i}{i!} \otimes N^i(x),$$
  where $g \in G_{L'}, a \in \A, x \in \mathcal D$.
  (Note that here we regard everything as a $G_{L'}$-representation, hence the field ``$K$" in \cite[Section 5]{Liu1} is our $L'$.) 

  By \cite[Lemma 5.2.1]{Liu1}, we have a $G_{L'}$-equivariant homomorphism $$\mathbb V_{\mathrm{st}}(\mathcal D) \cong ((\A \otimes_{\mathcal S_R} \mathcal D) \otimes_{\A} \B)^{\phi=1, \mathrm{Fil}},$$
  where $\mathbb V_{\mathrm{st}}(\mathcal D)$ is defined to be $\mathrm{Fil}^0(\mathcal D \otimes_{\mathcal S_R}\mathbf A_{\mathrm {st},R})^{\phi=1,N=0}$.

  Furthermore, since $\mathbb V_{\mathrm{st}}(\mathcal D)\cong \mathbb V_{\mathrm{st}}(D):=\mathrm{Fil}^0( D \otimes_{\mathcal S_R}\mathbf A_{\mathrm {st},R})^{\phi=1,N=0}$, we have the following $G_{L'}$-equivariant embedding
  $$\mathbb V_{\mathrm{st}}(\mathcal D) \hookrightarrow \B \otimes_{K'} D.$$ 
  
  \cite[Section 2.3]{Liu1} gives another $G_{L'}$-equivariant isomorphism $$\B \otimes_{K'} D \cong T_R(\MMM) \otimes_R \B.$$

  Semilinearity gives the following $G_{L'}$-equivariant embeddings:
  $$T_R(\MMM) \hookrightarrow \MMM \otimes_{\AAA} W(\mathbb C_p^\flat),$$
  $$((\A \otimes_{\mathcal S_R} \mathcal D) \otimes_{\A} \B)^{\phi=1, \mathrm{Fil}} \hookrightarrow (\A \otimes_{\mathcal S_R} \mathcal D) \otimes_{\A} \B.$$
  
 Putting these together, we conclude that the inclusion $\MMM \hookrightarrow \A \otimes_{\mathcal S_R} \mathcal D$ is $G_{L'}$-equivariant. From the construction of the lattice $M \subseteq D$, we obtain a $G_{L'}$-equivariant inclusion $\MMM \hookrightarrow \AC \otimes_{W(k')} D$.

From the paragraph preceding this proposition, $\A \otimes_{\varphi,\SSSS}\ML$ has a unique natural $G_{L'}$-action extending that on $\MMM$. This natural action is given by the element $\psi$. By uniqueness, this $G_{L'}$-action must coincide with the action defined in \cite[(5.1.1)]{Liu1}. From \cite[(5.1.4)]{Liu1}, we know that $\psi(m)=\mathrm{exp}(t \otimes N(m))$ holds.
\end{proof}

\subsection{The moduli stack of potentially semistable Breuil-Kisin modules}

Recall that in \autoref{ps}, we have defined the groupoids $\X(R)$ and $\ST(R)$ for finite flat $\mathcal O$-algebras. Our next subject is to upgrade these to a stack $\X$ and to relate it to $\XXX$.  

We can easily see that $\ST(R)$ is not defined in a good algebraic way to build a stack. Therefore, we need to build an auxiliary $\STT(R)$ using results from \autoref{Relation} to model $\X(R)$. 

\begin{defn} \thlabel{algstack}
Let $\STT$ be the functor from $\mathcal O$-algebras to groupoids, sending any $\mathcal O$-algebra $R$ to the groupoid of triples $((A^{(j)})_{j \in \mathcal J},N,\Xi)$ where
\begin{enumerate} 
    \item $A^{(j)}, (v+p)^h (A^{(j)})^{-1} \in \mathcal P_j(R)$ for all $j \in \mathcal J$ and $N \in \Ma(W(k') \otimes_{\mathbb Z_p} R)$;
    \item $\Xi \in \GL_n(\mathcal S_R[1/p])$ and $\Xi-1 \in u'\Ma(\mathcal S_R[1/p])$, where $\mathcal S_R$ is defined in \thref{Breuil}, satisfying $$\Xi^{-1} \e \left(t\otimes N \right) \psi(\Xi)\in \GL_n(\AAA),$$ and  $$\Xi \varphi(C)\varphi(\Xi)^{{-1}}= \overline C\in \Ma(W(k') \otimes_{\mathbb Z_p} R),$$ where $C:=\mathfrak C((A^{(j)} \omega^{(j)})_{j \in \mathcal J})$ and $\overline C:=C \mathrm{ \ mod \ } u'$; 
    \item $N \overline C =p \overline C \varphi(N)$;
    \item $\Psi^{}- \mathrm{id} \in u'\Ma(\AAA)$, where $\Psi:=\varphi^{-1}(\Xi^{-1} \e \left(t\otimes N \right) \psi(\Xi)) \in \GL_n(\AAA)$ \footnote{Note that $\varphi$ is bijective on $\AI$.}.
\end{enumerate}
\end{defn}

\begin{rmk}
    Since $N$ is naturally nilpotent, the expression $\Xi^{-1} \e \left(t\otimes N \right) \psi(\Xi)$ is in fact a finite sum.
\end{rmk}

\begin{prop} \thlabel{Nphi}
    Let $((A^{(j)})_{j \in \mathcal J},N,\Xi) \in \STT(R)$. Let $C=\mathfrak C((A^{(j)} \omega^{(j)})_{j \in \mathcal J})$ and $\Psi:=\varphi^{-1}(\Xi^{-1} \e \left(t\otimes N \right) \psi(\Xi))$. We have the identity $$\Psi \psi(C)=C\varphi(\Psi).$$
\end{prop}

\begin{proof}
Since $N \overline C =p \overline C \varphi(N)$, we obtain $N^i \overline C =p^i \overline C\varphi(N)^i$ by induction. Hence $(t\otimes N)^i \psi(\overline C)=p^i \overline C\varphi(N)^i$. Using the identity $\varphi(t)=pt$, we get $(t\otimes N)^i \psi(\overline C)= \overline C(\varphi(t)\otimes \varphi(N))^i$. 

Using $\Xi \varphi(\Psi)\psi(\Xi)^{-1}=\e(t \otimes N)$ and summing up the Taylor series for the exponential, one gets $\Xi \varphi(\Psi) \psi(\Xi)^{-1} \psi(\overline C)=\overline C \varphi(\Xi \varphi(\Psi) \psi(\Xi)^{-1})$. Substituting $\overline C$ by $\Xi\varphi(C)\varphi(\Xi)^{-1}$, we get our desired identity.
\end{proof}

\begin{thm} \thlabel{stackyy}
   Let $R$ be a $p$-adically complete, topologically of finite type, $\mathcal O$-flat algebra. We have an equivalence of groupoids $$\STT(R) / \prod_{j \in \mathcal J} \mathcal  P_j(R) \cong \X(R).$$ 
\end{thm}

\begin{proof}
    \thref{big2} has already shown that $\ST(R) / \prod_{j \in \mathcal J} \mathcal  P_j(R)$ is isomorphic to $\X(R)$. It remains to show that $$\ST(R) / \prod_{j \in \mathcal J} \mathcal  P_j(R) \cong \STT(R) / \prod_{j \in \mathcal J} \mathcal  P_j(R).$$
    
 Consider the following map from $\STT(R)$ to $\ST(R)$: $$\delta(R): ((A^{(j)}), N, \Xi)\mapsto ((A^{(j)}), \varphi^{-1}(\Xi^{-1} \e \left(t\otimes N \right) \psi(\Xi)) ) .$$ The requirements for $\Psi$ in conditions (2) and (4) in \thref{algstack} have ensured that $\delta(R)$ is well-defined. Condition (2) of \thref{local model} is guaranteed by \thref{Nphi}. It remains to verify condition (4). The claim that $\Psi$ is fixed by $\mathrm{Gal}(\overline K/L'_{\mathrm{cyc}}L'_\infty)$ is obvious since $t$ is fixed by this group. It suffices to prove that $g_m(\Psi)=\Psi^{(m)}$. In fact, let $\varphi(\Psi)= \Xi^{-1} \e\left(t\otimes N\right) \psi(\Xi)$, a direct calculation shows that $\varphi(\Psi)^{(m)}=\Xi^{-1}\e\left(mt\otimes N\right) \psi^m(\Xi)$ since $\psi(t)=t$. In the proof of \thref{big2}, we have shown that $g\psi g^{-1}=\psi^{k_g}$, which implies $g_m \psi g_m^{-1}=\psi^m$ for all $m$ prime to $p$. Hence $\psi^m(\Xi)=g_m \psi g_m^{-1} (\Xi)=g_m\psi(\Xi)$ since $\Xi$ has entries in $\mathcal S_R$, where $\mathrm{Gal}(L'_\infty L'_\mathrm{cyc}/L'_\infty)$ acts trivially. Therefore $$g_m(\varphi(\Psi))=g_m(\Xi)^{-1}g_m(\e\left(t\otimes N\right)) g_m(\psi(\Xi))=\Xi^{-1} \e\left(mt\otimes N\right) \psi^m(\Xi) =\varphi(\Psi)^{(m)}.$$
 
 To construct an inverse of $\delta(R)$, we start with $((A^{(j)}), \Psi) \in \ST(R)$, and associate with it a potentially semistable Breuil-Kisin module $\MM \in \X(R)$ with an eigenbasis $\beta$ such that the map $\psi: \MMM \to \MMM$ is given by the matrix $\Psi$ under the basis $\beta$. We form the $W(k') \otimes_{\mathbb Z_p} R$-module $M$ according to diagram (\ref{con})  and the operator $N:M \to M$. 

 Consider the $\phi$-equivariant inclusion $\varphi^*\MM \hookrightarrow \s \otimes_{W(k')\otimes R} M$. Let $\Xi$ be the matrix of the inclusion under basis $\beta$. In other words, the basis $\varphi^*\beta=1 \otimes \beta$ of $\varphi^*\MM$ is mapped to $\bar  \beta \Xi$ in $\s \otimes M$, where $\bar \beta:= \beta\mathrm{\ mod \ }u'$. The matrix of $\phi:M \to M$ under the basis $\bar \beta$ is $\bar C:= C\mathrm{\ mod \ }u'$. Since the inclusion is $\phi$-equivariant, we have $\Xi^{-1}\bar C\varphi(\Xi)= \varphi(C)$. The relation $\Xi-1 \in u'\Ma(\mathcal S_R[1/p])$ comes from the fact that $M=\MM/u'$. The uniqueness result in \thref{key} shows that this condition determines the inclusion.
 
 \thref{vip} shows that $\psi(m)=\mathrm{exp}\left(t\otimes N(m)\right)$ under this inclusion. Therefore, we obtain the matrix $N$. The uniqueness of the preimage of $((A^{(j)}), \Psi)$ is again guaranteed by the uniqueness of the embedding $\MMM \hookrightarrow \AC \otimes_{W(k')} M$.

 Since the equivalence relations on $\ST$ and $\STT$ are compatible, $\delta$ induces an isomorphism. 
\end{proof}

We will be using $\STT$ to define our stack $\X$. 

 Let $\mathcal W$ be the functor from $\mathcal O$-algebras to groupoids, sending any $\mathcal O$-algebra $R$ to the groupoid of triples $((A^{(j)})_{j \in \mathcal J},N,\Xi)$, where $A^{(j)}, (v+p)^h (A^{(j)})^{-1} \in \mathcal P_j(R)$ for all $j \in \mathcal J$, and $N \in \Ma(W(k') \otimes_{\mathbb Z_p} R)$ is nilpotent, and $\Xi \in 1+u'\Ma(\s)$. It is clear that $\mathcal W$ is a stack. Also, given any $\mathcal O$-algebra $R$, we have $\STT(R) \subseteq \mathcal W(R)$.

 Now we want to show that the conditions (2), (3) and (4) in \thref{algstack} cut out a substack of $\mathcal W$. For any $\mathcal O$-algebra $R$ and triple $W=((A^{(j)})_{j \in \mathcal J},N,\Xi) \in \mathcal W(R)$, we will define an ideal $I_W$ of $R$ such that $W$ is in $\STT$ if and only if $I_W=0$.

 We first look at the conditions $\varphi(\Psi)=\Xi^{-1} \e(t\otimes N) \psi(\Xi) \in \GL_n(\AAA)$ and $\Psi-\text{id} \in u'\Ma(\AAA)$. A priori, any entry of $\varphi(\Psi)$ lies in the ring $\A[1/p]$, so does any entry of $\Psi-\text{id}$. Regard both $\AAA[1/p]$ and $\A[1/p]$ as $\mathbb Q_p$-vector spaces. The inclusion $\AAA[1/p] \hookrightarrow \A[1/p]$ gives another $\mathbb Q_p$-vector space. Similarly, we can form the $\mathbb Q_p$-vector space $\A[1/p]/u'\A[1/p]$. Let $$V:=\frac{\AC[1/p]}{\AI[1/p]} \bigoplus \frac{\AC[1/p]}{\AI[1/p]} \bigoplus\frac{\AC[1/p]}{u'^p\AC[1/p]}.$$ We denote its continuous dual $\text{Hom}^{\text{cont}}_{\mathbb Q_p}(V,\mathbb Q_p)$ by $V^\vee$. We first construct an ideal $J_1 \subseteq R[1/p]$ such that $J_1=0$ if and only if $\varphi(\Psi)$ lies in $\GL_n(\AAA[1/p])$ and $\Psi-\text{id} \in u'\Ma(\A[1/p])$. Let $I_1 = J_1 \cap R$ be an ideal of $R$.

 In fact, we can define $$J_1=\text{the ideal generated by } \{(1\otimes\lambda)(\bar F_{\mu,\nu}, (\bar F^{-1})_{\mu,\nu}, (\bar F-I)_{\mu,\nu})|\lambda \in V^\vee, 1 \le \mu, \nu \le n\}\subseteq R[1/p],$$ where $\bar F$ denotes the reduction of $\varphi(\Psi)\in \A[1/p]$, and $A_{\mu,\nu}$ denotes the $(\mu,\nu)$-entry of the matrix $A$. We regard $1\otimes \lambda$ as a functional $R[1/p]\widehat\otimes_{\mathbb Q_p} V \to R[1/p]$. 
 
 It is clear that $J_1=0$ is equivalent to $\varphi(\Psi)\in \GL_n(\AAA[1/p])$ and $\Psi-\text{id} \in u'\Ma(\A[1/p])$. (Here, we have used $\varphi(u')=u'^p$.) 

 Next, we deal with the condition $\Xi C\varphi(\Xi)^{-1}= \overline C$, where $\overline C:=C \mathrm{ \ mod \ } u' \in \Ma(W(k')\otimes_{\mathbb Z_p}R)$. Let $W_{\mu,\nu} \in \s$ be the $(\mu, \nu)$ entry of the matrix $\Xi \varphi(C)\varphi(\Xi)^{-1}-\overline C$. Let $J_{\mu,\nu}$ be the ideal of $R[1/p]$ generated by the coefficients of the PD envelope expansion in $u'$ of $W_{\mu,\nu}$ (see \thref{S}). Let $J_2$ be the ideal of $R[1/p]$ generated by $J_{\mu,\nu}$ for every entry $(\mu, \nu)$. Let $I_2=J_2 \cap R$. Then $I_2=0$ if and only if $\Xi C\varphi(\Xi)^{-1}= C \mathrm{ \ mod \ } u'$.

 At last, let's consider the commuting relation in condition (3) $N \overline C = p \overline C \varphi(N)$. Let $W_{(\mu,\nu)} \in W(k') \otimes_{\mathbb Z_p} R$ be the $(\mu, \nu)$ entry of the matrix $N \overline C - p \overline C \varphi(N)$. Note that $W(k') \otimes_{\mathbb Z_p} R\cong \bigoplus_{j' \in \mathcal J'} R$. Let $I_{(\mu, \nu)}$ be the ideal of $R$ generated by the projection of $W_{(\mu,\nu)} $ to each component $R$ in its decomposition $\bigoplus_{j' \in \mathcal J'} R$. Let $I_3$ be the ideal generated by $I_{(\mu,\nu)}$ for every entry $(\mu, \nu)$.  Then $I_3=0$ if and only if $N \overline C = p \overline C \varphi(N)$.

 For any triple $W \in \mathcal W(R)$, we have defined ideals $I_1, I_2, I_3$ of $R$. Let $I_W'=I_1 + I_2 + I_3$. And let $I_W=(I_W':p^\infty)$ be the $p$-saturation of $I_W'$, i.e. the ideal of $R$ given by $$\{x \in R| p^mx \in I_W' \text{ for some non-negative integer }m \}.$$

In order to build a stack in the end, we need the following result about the base change of ``defining ideals". This is a potentially semistable version of \cite[Corollary 7.1.5]{LLHLM1} and \cite[Proposition 7.1.6]{LLHLM1}.

 \begin{prop} \thlabel{bc1}
     Let $R$ be a $p$-adically complete, topologically of finite type $\mathcal O$-algebra. Let $W \in \mathcal W(R)$. Then, the ideal $I_W$ of $R$ we constructed satisfies
\begin{enumerate}
    \item $R/I_W$ is $\mathcal O$-flat; 
    \item For any flat map $R \to S$ such that $S$ is also a $p$-adically complete, topologically of finite type $\mathcal O$-algebra giving rise to $W_S \in \mathcal W(S)$, one has $I_WS = I_{W_S}$.
\end{enumerate}
     
 \end{prop}

 \begin{proof}
Property (1) holds from our construction of $I_W$ out of $I_W'$.
     
     To verify property (2), let $f:R \to S$ be such a flat map. Since evaluation by a $\mathbb Q_p$-linear functional commutes with the extension of scalars, we can see that $J_{1,S}=J_{1,R}S[1/p]$. Thus $I_{1,S}=I_{1,R}S$ by the $p$-saturatedness of both ideals. The cases for $I_2$ and $I_3$ are straightforward. From the definition of $I_W$, we obtain $I_WS = I_{W_S}$. 
 \end{proof}

 \begin{prop} \thlabel{stacky}
     Let $R$ be any $p$-adically complete, topologically of finite type, flat $\mathcal O$-algebra. A triple $W \in \mathcal W(R)$ is in $\STT(R)$ if and only if the ideal $I_W=0$. 
 \end{prop}

\begin{proof}
Let $W\in \mathcal W(R)$. One can see from the definition that $W\in \STT(R)$ implies $I_W=0$.

On the other hand, we know that $I_W=0$ if and only if $I_1=I_2=I_3=0$. It is clear that $I_2=0$ implies $\Xi \varphi(C)\varphi(\Xi)^{-1}=\overline C$ and that $I_3=0$ implies $N\overline C=p\overline C \varphi(N)$. According to \thref{algstack}, it suffices to verify that $I_1=0$ would furthermore imply $\varphi(\Psi)\in \GL_n(\AAA)$ and $\Psi-\text{id} \in u'\Ma(\AAA)$.

In fact, from $I_1=0$, we first know that $J_1=0$. According to the construction of $J_1$, we have $\varphi(\Psi)\in \GL_n(\AAA[1/p])$ and $\Psi-\text{id} \in u'\Ma(\A[1/p])$. Suppose that $\Psi \notin \Ma(\AAA)$, then there exists a smallest positive integer $m$ such that $p^m \Psi \in \Ma(\AAA)$. Let $Y=p^m \Psi \text{ mod }p \in \Ma(\AAA/p)$. By the minimality of $m$, $Y \ne 0$. Since $p^m \Psi-p^m \in u'\Ma(\A[1/p])$ and $m \ge 1$, $Y \in u'\Ma(\AAA/p)$ by \thref{intersection}.

Note that we already have $I_3=0$, implying $N\overline C=p\overline C \varphi(N)$. From \thref{Nphi}, we know that $p^m\Psi \psi(C)=C\varphi(p^m\Psi)$. Reducing mod $p$, we get $Y \psi(C)=C\varphi(Y)$. Since $Y \in u'\Ma(\AAA/p)$ and $\varphi(u')=u'^p$, we have $Y \in u'^p\Ma(\AAA/p)$. Applying this repeatedly, we can see that $Y$ is divisible by any high power of $u'$, and thus is $0$. Contradiction! Therefore, $\Psi\in \Ma(\AAA)$. Similarly, we can prove $\Psi^{-1} \in \Ma(\AAA)$, thus $\varphi(\Psi)\in \GL_n(\AAA)$.

Since we already have $\Psi-\text{id} \in u'\Ma(\A[1/p])$, $\Psi-\text{id}  \in \Ma(\AAA \cap u'\A[1/p])$. From \thref{intersection}, we obtain $\Psi-\text{id} \in u'\Ma(\AAA)$. Therefore, $I_W=0$ implies $W \in \STT(R)$.

\end{proof}

\begin{lemma} \thlabel{intersection}
   Let $R$ be an $\mathcal O$-flat algebra, we have $\AAA \cap u'\A[1/p]=u'\AAA$.
\end{lemma}

\begin{proof}
   $\A$ is the PD envelope of $\AI$ with respect to the ideal $E(u')$. Since $E(u') \equiv p \mathrm{\ mod \ }u'$, $\A/u'$ is the PD envelope of $\AAA/u'$ with respect to $p$. Since in the latter ring $\AAA/u'$, $p$ already has its divided powers $p^i/i! \in \mathbb Z_p$. Therefore, we have an injection $\AAA/u' \hookrightarrow \A/u'$, which implies $u'\A \cap \AAA=u'\AAA$. Since $\A/u'\A$ is $p$-torsion-free, we arrive at $\AAA \cap u'\A[1/p]=u'\AAA$.
\end{proof}

We define an action of  $\prod_{j \in \mathcal J} \mathcal  P_j$ on $\STT$ as follows.

Given $W=((A^{(j)})_{j \in \mathcal J}, N, \Xi) \in \STT(R)$, and $(I^{(j)})_{j \in \mathcal J} \in \prod_{j \in \mathcal J} \mathcal  P_j(R)$, we define $(I^{(j)})\cdot W$ as $$W'=((I^{(j)}A^{(j)} \mathrm{Ad}(\omega^{(j)})(\varphi(I^{(j-1)})^{-1})_j), \bar I N \bar I^{-1}, I \Xi \varphi(I)^{-1}),$$  where $I=\mathfrak C((I^{(j)})_{j \in \mathcal J})$ and $\bar I=I \mathrm{\ mod \ }u'$. It is easy to see that the defining ideals are invariant under the action of  $\prod_{j \in \mathcal J} \mathcal  P_j$, i.e. $I_W=I_{W'}$.

 Using \thref{bc1}, we can define our stack $\X$ from the ambient stack $[\mathcal W / \prod_{j \in \mathcal J}\mathcal P_j]$.

 \begin{defn} \thlabel{new}
     Let $X_n^{[0,h],\tau} \hookrightarrow [\mathcal W / \prod_{j \in \mathcal J}\mathcal P_j]$ be the unique $\mathcal O$-flat closed substack characterized by the following property: for any $p$-adically complete, topologically of finite type flat $\mathcal O$-algebra $R$ and a map $f:\mathrm{Spf\ }R \to [\mathcal W / \prod_{j \in \mathcal J}\mathcal P_j]$ corresponding to the triple $W$, $f$ factors through $\X$ if and only if the ideal $I_W\subseteq R$ is $0$.
 \end{defn}

 \begin{rmk}
     The construction here is parallel to the potentially crystalline case in \cite{LLHLM1}. The existence of such a substack follows from the general construction of \cite[Section 9]{Eme}.
 \end{rmk}

 \begin{rmk} \thlabel{coincide}
     Because of \thref{stackyy} and the defining property of the stack \thref{stacky}, we can see that this definition of $\X$ coincides with our previous definition of $\X(R)$ when $R$ is a $p$-adically complete, topologically of finite type, $\mathcal O$-flat algebra. Hence, we call it the stack of potentially semistable Breuil-Kisin modules. 

 \end{rmk}

Now we want to establish an isomorphism of stacks between $\XXX$ and $\X$. We first produce a map $T:\X \to \mathcal X_n$ and then show that $T$ factors through $\XXX$.

Let $U:=\mathrm{Spf\ } R \to \X$ be any smooth cover. We can assume that $R$ is $\mathcal O$-flat and $p$-adically complete, topologically of finite type since $\X$ is $\mathcal O$-flat by \thref{new}. Then $\X(R)$ corresponds to a potentially semistable Breuil-Kisin module $\MM$ with coefficients in $R$ according to \thref{coincide}. Let $\MMM$ be the associated Breuil-Kisin-Fargues module, and let $\mathcal M=\MMM \otimes_{\AAA} W(\mathbb C_p^{\flat})_R$ be the associated $(\varphi, G_K)$-module. According to \cite[Section 3.6.4]{EG}, $\mathcal M$ determines a map $T_U:U \to \XXX$. Then we need to check that $T_U$ descends to a morphism of stacks $T:\X \to \mathcal X_n$.

Let $U_1=U \times_{\X}U$. Let $p_1, p_2$ be the two projections of $U_1$ to $U$. Take any smooth cover $g:V=\mathrm{Spf\ }A \to U_1$, where $A$ is also $\mathcal O$-flat and $p$-adically complete, topologically of finite type. Suppose that $p_i\circ g$ is represented by $f_i: R \to A$ for $i=1,2$. The map $g:V \to U_1$ implies there is an isomorphism $\alpha:f_1^*\MM \to f_2^* \MM$, where $f_i^*\MM=\MM \hat\otimes_{R,f_i} A$. Note that $p_i^* \mathcal M=(f_i^*\MMM)\otimes_{\mathbf{A}_{\mathrm{inf},A}}W(\mathbb C_p^{\flat})_A$. And obviously, base change to $W(\mathbb C_p^{\flat})$ and extension of the $G_K$-actions respect isomorphisms. Therefore, we have a natural isomorphism $\beta:p_1^* \mathcal M \to p_2^* \mathcal M$. Similarly, one can see that the isomorphism $\beta$ satisfies the cocycle condition. Therefore, $T_U$ descends to $T:\X \to \mathcal X_n$. 

We call it $T$ because for any finite flat $\mathcal O$-algebra $R$, the functor $T(R)$ is just our functor $T_R$ in \thref{equivalence}. We first show that the target of $T$ is the potentially semistable Emerton-Gee stack.

\begin{prop}
    The map $T:\X \to\mathcal X_n$ factors through $\XXX$.
\end{prop}
\begin{proof}
 Take a smooth chart $U=\text{Spf }R \to \X$. Composing with $T$, we have $U \to \mathcal X_n$. Consider the base change $U \times_{\mathcal X_n} \XXX$. Since $\XXX \hookrightarrow \mathcal X_n$ is closed, $U \times_{\mathcal X_n} \XXX \hookrightarrow U$ is defined by an ideal $I \subseteq R$. Let $R\to R'$ be any $\mathcal O$-algebra homomorphism with $R'$ finite flat. Consider the composition $\text{Spf }R' \to U \to \X \to \mathcal X_n$. Since $R'$ is finite flat, this map $\text{Spf }R' \to \mathcal X_n$ corresponds to a point on $\XXX$. Thus $IR'=0$. By \thref{aux1}, $I=0$. Therefore, $T$ factors through $\XXX$.
\end{proof}

So now, by abuse of notation, we will regard the map $T$ as $T:\X \to \XXX$. We need the following lemma to reduce to the finite flat case to prove that $T$ is an isomorphism.

\begin{lemma} \thlabel{aux2}
    Let $R$ be any $p$-adically complete, topologically of finite type, and $\mathcal O$-flat algebra. Given any $\MM_1,\MM_2 \in \X(R)$, let $\mathcal M_i=T(\MM_i)$ for $i=1,2$. Let $f:\mathcal M_1 \to \mathcal M_2$ be a map between two \'etale $(\varphi, G_K)$-modules. For any finite flat $\mathcal O$-algebra $R'$ and any $\mathcal O$-algebra homomorphism $R \to R'$, we write $\mathcal M_{i,R'}$, $\MM_{i,R'}$ and $f_{R'}$ for the base change to $R'$ of the corresponding objects or maps. Then there is an ideal $I_R$ such that $f_{R'}$ is of the form $T(g)$ with $g$ a map between potentially semistable Breuil-Kisin modules if and only if $I_R R'$ is the zero ideal of $R'$.
\end{lemma}

\begin{proof}
  Let $\F$ be the stack of \'etale $\varphi$-modules. We first recall that in the proof of \cite[Proposition 3.7.2]{EG}, a map $F:\XXX \to \F$ that is an analogue of restricting $G_K$-representations to $G_\infty$-representations is constructed. Let $\mathcal M_{\infty,i}=F(\mathcal M_i)\in \F(R)$. When regarding $\mathcal M_i$ as \'etale $(\varphi, G_K)$-modules, we have $\mathcal M_{\infty,i}\subseteq \mathcal M_i$ since $\mathcal O_{\mathcal E,L'} \subseteq W(\mathbb C_p^\flat)$. Letting $h=F(f)$, we can see that $f$ can be restricted to a map $h:\mathcal M_{\infty,1} \to \mathcal M_{\infty,2}$. Note that $\MM_i[1/u']=\mathcal M_{\infty,i}$. Hence $h$ can be further restricted to a map $g: \MM_1 \to \MM_2$ if and only if all the coefficients of negative powers of $u'$ in all the entries of the matrix representation of $h$ vanish. 

After fixing eigenbases of $\MM_1$ and $\MM_2$ locally, let the matrix of $h$ be $\sum_{i>0}(u')^{-i}A_i+B$ with $B$ having entries in $\SSSS$. Set $I_R$ to be the ideal of $R$ generated by all entries of all the $A_i$. (Here, of course, we decompose $W(k')\otimes R$ into copies of $R$ as before.) We note that this definition is independent of the bases chosen since the change-of-basis matrices both have coefficients in $\SSSS$, without any negative power of $u'$.

Now we can see that $f$ comes from $g$ if and only if $I_R=0$. We can do the same for any finite $\mathcal O$-flat base change $R'$, and get $I_{R'}$. By the construction, $I_{R'}=I_R  R' \subseteq R'$. 
\end{proof}

\begin{thm} \thlabel{big1}
There is an isomorphism between the potentially semistable Emerton-Gee stack $\XXX$ and the stack of potentially semistable Breuil-Kisin modules $\X$.
\end{thm}

\begin{proof}
  We first show that $T:\X \to \XXX$ is a monomorphism of stacks. From \thref{equivalence}, we know that $T(R)$ is an equivalence of groupoids for any finite flat $\mathcal O$-algebra $R$. Now let $R$ be a $p$-adically complete flat $\mathcal O$-algebra which is topologically of finite type. We want to first show that $T(R)$ is fully faithful. Let $\MM_i\in \X(R)$ ($i=1,2$), with $\mathcal M_i=T(\mathfrak M_i) \in \XXX(R)$. Consider a map $\alpha: \mathcal M_1 \to \mathcal M_2$. To show faithfulness, we regard $\mathcal M_i$ as \'etale $\varphi$-modules over $L'$, with coefficients in $R$. Since $\MM_i \subseteq \MM_i[1/u']=\mathcal M_i$, $\alpha$ determines at most one map between $\MM_1$ and $\MM_2$. Faithfulness follows. 
    
    Now, we prove fullness. Given any $\MM_1,\MM_2 \in \X(R)$, let $\mathcal M_i=T(\MM_i)$ for $i=1,2$. Let $f:\mathcal M_1 \to \mathcal M_2$ be a map between two \'etale $(\varphi, G_K)$-modules. We know from \thref{equivalence} that if we base change to any finite flat $\mathcal O$-algebra $R'$, the corresponding $f_{R'}$ would come from a map of potentially semistable Breuil-Kisin modules. By \thref{aux2}, this means $I_R  R'$ is the zero ideal of $R'$ for each finite flat $\mathcal O$-algebra $R'$ and map $R \to R'$. Using \thref{aux1}, we can see that $I_R=0$. Fullness follows by applying \thref{aux2} again.

    To prove that $T$ is an isomorphism from the fact that it is a monomorphism, we want to apply \cite[Lemma 7.2.6 (1)]{LLHLM1} since we already have essential surjectivity on finite flat points due to \thref{equivalence}. It remains to verify that $\X$ is of finite type. We already know that the stack of Breuil-Kisin modules, denoted by $\Y$ as in \cite{LLHLM1}, is of finite type. The same also holds when including the monodromy operator $N \in \Ma(W(k') \otimes_{\mathbb Z_p}R)$. It suffices to show that with given $C,N$, the matrix $\Xi \in 1+u'\Ma(\s)$ is uniquely determined.

    In fact, suppose that $\Xi_i \in 1+u'\Ma(\s)$ ($i=1,2$) both satisfy $\Xi_i \varphi(C) \varphi(\Xi_i)^{-1}=\overline C$. Let $\Xi=\Xi_2^{-1}\Xi_1$. Then $$\Xi=\varphi(C\Xi C^{-1}).$$ From $\Xi_1\equiv \Xi_2 \equiv 1 \text{ mod } u'$, we know that $\Xi\equiv 1 \text{ mod } u'$. Using the above equation and the fact that $\varphi(u')=u'^p$, we further get $\Xi\equiv 1 \text{ mod } u'^p$. Continuing this process and using the $u'$-adic completeness, we know that $\Xi=1$, implying $\Xi_1=\Xi_2$ as desired.
\end{proof}

\begin{rmk}
    From now on, we will use the two stacks $\XXX$ and $\X$ interchangeably. This equivalence will be used throughout the next chapter.

    One aim of this theorem is that for a general $p$-adically complete, topologically of finite type, $\mathcal O$-flat algebra $R$, we can view an $R$-point of $\XXX$ (not considering endomorphisms) as an object of $\STT(R)$ (using the defining property \thref{stacky} of $\X$). It is easy to describe the latter groupoid.
\end{rmk}

\section{The stack of Weil-Deligne representations} \label{s5}

\subsection{Weil-Deligne representations}

In \cite{BG}, a stack parametrizing Weil-Deligne representations is constructed. We first recall the definition and then relate it to the potentially semistable locus of the Emerton-Gee stack.
\begin{defn}
Let $A$ be an $E$-algebra. We say that $(\mathfrak D,\varrho,N)$ is a Weil-Deligne representation of rank $n$ with coefficients in $A$, and of descent $L'/K$ if 
\begin{enumerate}
    \item $\mathfrak D$ is a free $A$-module of  rank $n$;
    \item $\varrho:W_K \to \mathrm{GL}(\mathfrak D) \cong \mathrm{GL}_n(A)$ is a continuous representation of $W_K$ such that $\varrho(I_K)$ factors through $I(L'/K)=\Delta'$;
    \item $N \in \mathrm{End}_A(\mathfrak D)$ satisfies $ N\varrho(w)=p^{f \nu(w)} \varrho(w) N$ for every $w \in W_K$.
\end{enumerate}
\end{defn}

By abuse of notation, we often use $\mathfrak D$ to refer to the triple $(\mathfrak D,\varrho,N)$ if there is no confusion.

\begin{defn}
Let $\WD$ be the stack over $E$, where \'etale locally, the fiber over any $E$-algebra $A$ is the groupoid of rank $n$ Weil-Deligne representations with coefficients in $A$ of descent $L'/K$. 
\end{defn}

\begin{rmk}
Given $\mathfrak D \in \WD$, the data of the homomorphism $\varrho:W_K \to \mathrm{GL}(\mathfrak D)$ is equivalent to the data of $\varrho(I_K)$ and $\varrho(\Fr)$. Hence if we fix the representation of $\Delta'$ to be $\tau $ introduced in \autoref{Breuil-Kisin modules with descent data}, then the data of $\mathfrak D$ is equivalent to the information of the two $A$-linear maps $\varrho(\Fr), N$ on $\mathfrak D$, satisfying some commuting relations. We often write $\phi_\mathfrak D=\varrho(\Fr)$.
\end{rmk}

We will now recall a more concrete way to describe $\WD$ following \cite{BG}. 

\begin{defn} \thlabel{weil-deligne}
 Let $\W$ be the stack over $E$ whose fiber over any $E$-algebra $A$ is the groupoid of triples $(\phi, N, T)$ where

\begin{enumerate}
    \item $\phi \in \mathrm{GL}_n(A), N \in \Ma(A)$, $T:\Delta' \to \mathrm{GL}_n(A)$ is a representation of $\Delta'$;
    \item $ N\phi=p^{f} \phi N$;
    \item $\phi T(g)=T(\Fr g \Fr^{-1})\phi$, and $N=T(g) N T(g)^{-1}$ for all $g \in \Delta'$.
\end{enumerate}
\end{defn}

We further define an action of $\GL_n(A)$ on $\W(A)$ by $$a\cdot(\phi,N,T)=(a \phi a^{-1}, aNa^{-1}, aTa^{-1})$$ for all $a \in \GL_n(A), (\phi,N,T) \in \W(A)$. Let $\phi$ be the matrix of $\phi_\mathfrak D$. \cite[Lemma 2.1.3]{BG} gives the following result:

\begin{thm} \thlabel{wd}
    The quotient stack $[\W / \mathrm{GL}_n]$ is equivalent to the stack $\WD$.
\end{thm}

For later use, we define the $\tau$-part of the stack $\WD$ and give its main property.
\begin{defn}
    The substack $\WD^{\tau}$ is defined as the substack of $\WD$ whose $A$-points consist of triples $(\phi, N, T)$ with $T$ isomorphic to $\tau \otimes A$.
\end{defn}

\begin{prop} \thlabel{wdsmooth}
    The stack $\WD^\tau$ has a dense open substack which is formally smooth.
\end{prop}

\begin{proof} 
  By \cite[Corollary 2.4.5]{BG}, the stack $\WD$ has a dense open formally smooth locus. It suffices to show that $\WD^{\tau}$ is an open and closed substack of $\WD$.

  Let $V$ be the representation space of $T$, which is isomorphic to $A^n$. We can decompose the tame inertial type as $\tau=\bigoplus_i \chi_i^{m_i}$, where $\chi_i$ are mutually distinct characters. Let $$e_i=\frac{1}{e'}\sum_{g\in \Delta'} \chi_i(g^{-1})g \in E[\Delta']$$ be the idempotent in the group algebra $E[\Delta']$ corresponding to $\chi_i$. Then the $\chi_i$-eigenspace is just $e_iV$. So $T$ is isomorphic to $\tau \otimes A$ if and only if $\mathrm{rank}_A(e_i V)=m_i$ for each $i$.

    Since each $e_i V$ is a finite projective direct summand of $V$, and the rank of a finite projective module is locally constant, the rank conditions $\mathrm{rank}_A(e_i V)=m_i$ cut out an open and closed substack of $\WD$, which is our $\WD^\tau$.
\end{proof}

In order to construct the maps in the next section, we need to use the analytification of the Weil-Deligne stack to match up with the rigid analytification of the generic fiber of $\XXX$. By definition, given any $E$-algebra $A$, we have $$\WD^{\mathrm {an}}(\mathrm{Sp\ }A)=\WD(\mathrm{Spec\ }A).$$ 

\subsection{From potentially semistable representations to Weil-Deligne representations}
We now recall the recipe of associating a Weil-Deligne representation to a potentially semistable representation in \cite{Fon}. 

We start with a potentially semistable representation $\rho:G_K \to \mathrm{GL}_n(E)$. After fixing an $\mathcal O$-lattice, let $\MM$ be the corresponding potentially semistable Breuil-Kisin module, with inertial type $\tau$. Let $(D,\phi_D,N_D)=D_{\mathrm{pst}}(\rho )$. We can decompose $D$ as $D=\bigoplus_{j' \in \mathcal J'} D^{(j')}$, such that $D^{(j')}=D \otimes_{W(k'),\sigma_{j'}}E$ for any $j' \in \mathcal J'$. For each $j' \in \mathcal J'$, we also have a map $\phi_D^{(j')}:D^{(j'-1)}\to D^{(j')}$.

 The action of $W_K$ on $D$ is given by the formula $$\varrho(w)=w \circ (\phi_D)^{f \nu(w)},$$ for every $w \in W_K$, where $w \in W_K$ on the right hand side acts on $D$ through the inertial subgroup $I_K$ by $\tau$.

Let $\D=D^{(0)}, \varrho=\varrho^{(0)}$. We define the monodromy operator $N$, an $E$-endomorphism on $\D$ through $N=N_D|_{\D}: \D \to \D$. We call $(\D,\varrho,N)$ the Weil-Deligne representation associated to the potentially semistable representation $\rho$. \cite[Lemme 2.2.1.2]{BM} shows that the process of forming the Weil-Deligne representation is independent of the embedding $\sigma_{j'}$ chosen. 

In fact, the recipe of Fontaine above gives a functor $$\iota:\mathrm{Rep}^{[0,h],\tau}_E(G_K) \to \WD^{\mathrm {an}}(E).$$

Now we generalize this process to give a map $\Xe(A) \to \WD^{\mathrm {an}}(A)$ for any finite field extension $A/E$. We will still call this functor $\iota$ the recipe of Fontaine. We make use of the following notation. Given any $x\in W(k') \otimes_{\mathbb Z_p} A$, we define $x^\circ\in A$ to be the projection onto its $\sigma_0$-factor. The same notation $( \cdot )^\circ$ also applies to matrices in $\Ma( W(k') \otimes_{\mathbb Z_p} A)$. 

Let $\mathcal M \in \Xe(A)$. Since $A$ is a $p$-adic field, as we have already noticed, for example in the proof of \thref{k}, $\mathrm{Spf\ }R$, where $R$ is the ring of integers of $A$, is the only Raynaud formal model of $\mathrm{Sp\ }A$. So $\mathcal M$ corresponds to a potentially semistable Breuil-Kisin module $\MM \in \X(R)$. Suppose that it corresponds to the triple $((A^{(j)})_{j \in \mathcal J},N,\Xi)$. As usual, let $$\overline C= \mathfrak C((A^{(j)} \omega^{(j)})_{j \in \mathcal J}) \mathrm{\ mod \ } u', \Psi=\varphi^{-1}(\Xi^{-1} \e(t\otimes N) \psi(\Xi)).$$ Let $\phi= (\overline C^{f})^\circ \in \Ma(R)$. We also have $N^\circ \in \Ma(R)$. 

By iterating the identity $N \overline C=p\overline CN$, we have $N^\circ \phi=p^f \phi N^\circ$. We set $T: \Delta' \to \GL_n(E)$ to be the inertial type $\tau$. Then $(\phi, N^\circ,T) \in \WD(A)$ according to \thref{weil-deligne}. We call it the Weil-Deligne representation $WD(\MM)$ associated to $\MM$. If we use the usual language of Weil-Deligne representations, then the underlying module of $WD(\mathcal M)$ is $(\MM/u') \otimes _{W(k'), \sigma_0} A$. The maps $\phi,N$ are given by $$\phi=((\phi_{\MM})^f \mathrm{\ mod\ }u')^\circ,$$ $$ N=(N_{\MM} \mathrm{\ mod\ }u')^\circ.$$

The following theorem further generalizes this process to a morphism $WD:\Xe \to \WD^{\mathrm {an}}$ between the two stacks. 

\begin{thm}
    There is a morphism $WD: \Xe \to \WD^{\mathrm {an}}$ between the two stacks such that for any given finite type point $\mathcal M \in \Xe(A)$, the image $WD(\mathcal M)\in \WD^{\mathrm {an}}(A)$ coincides with Fontaine's recipe of associating Weil-Deligne representations to potentially semistable representations. 
\end{thm}

\begin{proof}
    We first give a construction of the morphism $WD$. Given any affinoid $E$-algebra $A$ and a point $x \in \Xe(\mathrm{Sp \ }A)$, we can assume that $x$ comes from $y \in \XXX(Y)$ according to \thref{rig}, where $Y$ is a Raynaud formal model of $\mathrm{Sp \ }A$. Let $U$ be any affinoid open of $Y$ given by $\mathrm{Spf \ } R$, where $R$ is an admissible formal $\mathcal O$-algebra. In particular, $R$ is of finite type and flat by admissibility. For each such $R$, we have a corresponding module $\MM_U \in \X(R)$. And we can associate a Weil-Deligne representation $\D_U=(\MM_U/u')  \otimes _{W(k'), \sigma_0} R[1/p]$, with the operator $\phi, N$ defined as in the paragraph above.  

    Now let $U$ range through all the affine formal opens of $Y$, and form all the corresponding $\D_U$. The gluing data on $Y$ are preserved since the association of $\MM$ to $\MM/u'$ is functorial. Therefore, various $\D_U$ glue to some $\D(Y) \in \WD^{\mathrm{an}}(Y_\eta)$ \footnote{We use $Y_\eta$ to mean the generic fiber of $Y$, which, in this case, is just $\mathrm{Sp \ }A$.}. We define the image of $x \in \Xe(\mathrm{Sp \ }A)$, which is represented by the model $Y$, in $\WD^{\mathrm {an}}(\mathrm{Sp\ }A)=\WD(\mathrm{Spec\ }A)$ to be $$WD_Y(x)=\D(Y).$$ 
    
    Now let $Y'$ be another Raynaud formal model of $\mathrm{Sp \ }A$, with map $f: Y' \to Y$. Let $y'=f^*y \in \XXX(Y')$ be the pullback of $y \in \XXX(Y)$. We can form $WD_{Y'}(x)=\D(Y')$ likewise as before. Since the map $y' \to y$ can be restricted to each affine formal open on both $Y'$ and $Y$, and the association of $\MM$ to $\MM/u'$ is functorial, we will have an isomorphism $\theta_f: WD_{Y'}(x) \to f^*WD_Y(x)$.  It is obvious that the isomorphism $\theta_f$ is compatible with pullbacks in the source and composition. Therefore we obtain a morphism between the two rigid stacks.

    Clearly, the construction process coincides with the case where the $E$-algebra $A$ is a finite field extension of $E$, hence it coincides with Fontaine’s recipe of associating Weil-Deligne representations to potentially semistable representations.
\end{proof}

\begin{rmk}
    In the case of $A=E$, we can reformulate this result as follows: there is a unique way that the functor $\iota:\mathrm{Rep}^{[0,h],\tau}_E(G_K) \to \WD^{\mathrm {an}}(E)$ can factor through the functor $WD$, which is $$\iota=WD \circ \kappa,$$ where $\kappa$ is given in \thref{k}.
\end{rmk}

\subsection{Smoothness: construction of the square-zero lift} \label{5.3}
In the final subsections of this paper, we will show the smoothness of the map $WD$.

Let the field $F$ be any extension of $E$. Let $\mathcal M_0 \in \Xe(F)$. Let $\D_0=WD(\MM_0)$ be the image of $\MM_0$, viewed as in $\WD(F)$. In fact, $\D_0 \in \WD^{\tau}(F)$. We will show that the map $WD$ is smooth at $\mathcal M_0$ using the definition of smoothness.

Let $\widetilde A \twoheadrightarrow A$ be a surjection of finite local Artinian $F$-algebras with kernel $I$ satisfying $I \mathfrak m=0$, where $\mathfrak m$ denotes the maximal ideal of $\widetilde A$. Then $\widetilde A/\mathfrak m$ is a finite extension of $F$, which we denote $\widetilde F$. Since $\mathfrak m$ is the unique prime ideal of the Artinian ring $\widetilde A$, any Raynaud formal model of either $\mathrm{Sp\ }\widetilde A$ or $\mathrm{Sp\ } A$ is an affinoid. Let $\mathrm{Spf\ } R$ be any model of $\mathrm{Sp\ } A$. By the definition of Raynaud models, $R$ is topologically of finite type and flat over $\mathcal O$. We fix an isomorphism $R[1/p] \cong A$ and view $R$ as a subring of $A$.

Let $\mathcal M \in \Xe(A)$ and $\D=WD(\mathcal M) \in \WD^{\mathrm{an}}(A)$. We view $\D$ as in $\WD(A)$. According to \thref{rig}, suppose that $\mathcal M$ comes from $\MM \in \XXX(R)$. Suppose further that $\D \in \WD( A)$ maps to $\D_0\in \WD(F)$ through $A \twoheadrightarrow \widetilde A/\mathfrak m$. We further assume that there exists $\widetilde \D \in \WD^\tau(\widetilde A)$ lifting $\D$.

According to \cite[Proposition 4.1]{BS}, there is a unique $(\phi,N,\Delta)$-module $D \in L(\phi, N, \Delta)_A$ corresponding to $\D$. Let $M=M_R(\MM)$, then $D=M[\frac{1}{p}]$. Similarly, we obtain $\widetilde D$. By functoriality, we have $\widetilde D \otimes_{\widetilde A} A\cong D$. 

Let $N_D$ (resp. $N_{\widetilde D}$) denote the map $N$ on $D$ (resp. $\widetilde D$). Let $\phi_D$ (resp. $\phi_{\widetilde D}$) denote the map $\phi$ on $D$ (resp. $\widetilde D$).

We will construct $\widetilde{\mathcal M} \in \Xe(\widetilde A)$, so that it lifts $\mathcal M \in \Xe(A)$, and satisfies $WD(\widetilde {\mathcal M})=\widetilde \D$. 

Let $\widetilde R$ be the preimage of $R \subseteq A$ under the surjection $\widetilde A \twoheadrightarrow A$. The kernel of this map is $I \cap \widetilde R$, which we denote by $\mathfrak I$. So we have the following commutative diagram with column maps being reduction maps:  
$$\begin{tikzcd}
 \widetilde R & \widetilde A  \\
	R &  A.
	\arrow[,,swap, from=1-1, to=2-1]
	\arrow[hook, from=1-1, to=1-2]  
	\arrow[,, from=1-2, to=2-2]
	\arrow[hook, from=2-1, to=2-2] 
\end{tikzcd}$$ 

Here we define $\mathcal S_{\widetilde R}=\mathcal S_R \oplus (\mathfrak I \otimes_A \s)$, with multiplication given by $(s,x)(s',x')=(ss',sx'+s'x)$, where $s,s' \in \mathcal S_R$ and $x,x' \in \mathfrak I \otimes_A \s$. Then we have $\mathcal S_{\widetilde R}/\mathfrak I \mathcal S_{\widetilde R} \cong \mathcal S_R$. Also, we define $\mathfrak S_{L', \widetilde R}= \SSSS \oplus (\mathfrak I \otimes_A \SSSS[1/p])$ and give it similar multiplication as in $\mathcal S_{\widetilde R}$. The other rings $\mathbf A_{\mathrm{inf}, \widetilde R}$ and $\mathbf A_{\mathrm{cris}, \widetilde R}$ are defined in the same way as $\mathfrak S_{L', \widetilde R}$.

Note that $\mathrm{Spf\ }\widetilde R$ is not a formal model of $\mathrm{Sp\ }\widetilde A$. Our strategy is to construct a potentially semistable Breuil-Kisin module with coefficients in $\widetilde R$ lifting $\MM \in \XXX(R)$ in the first place. Then we explain how to define $\wM$ on Raynaud formal models of $\mathrm{Sp\ } \widetilde A$, and thus obtain $\widetilde{\mathcal M} \in \Xe(\widetilde A)$.

The surjection $\widetilde D \twoheadrightarrow D$ induces a surjection $\mathfrak F: \AC \otimes_{W(k')} \widetilde D \twoheadrightarrow \mathbf A_{\mathrm{cris}}  \otimes_{W(k')} D$. Let $\widetilde M$ be the preimage of $M \subseteq D$ under this surjection $\widetilde D \twoheadrightarrow D$. Hence $\mathfrak F$ restricts to $ \AC \otimes_{W(k')} \widetilde M \twoheadrightarrow \mathbf A_{\mathrm{cris}}  \otimes_{W(k')} M$, which we also call $\mathfrak F$ by abuse of notation. We endow $\mathbf A_{\mathrm{cris}} \otimes_{W(k')} D$ with the usual $\phi$-map given by $\phi=\varphi_{\mathbf A_{\mathrm{cris}} } \otimes \phi_{D}$, and do the same for $\AC \otimes_{W(k')} \widetilde D$.

 \begin{lemma} \thlabel{fix}
     The $W(k') \otimes \widetilde R$-module $\widetilde M$ is stable under $\phi_{\widetilde D}$. It is free of rank $n$.
 \end{lemma}
\begin{proof}
    For any $m \in \widetilde M$, $\mathfrak F(\phi_{\widetilde D}(m))=\phi_D(\mathfrak F(m)))\in \phi_D(M)\subseteq M$. Hence by the definition of $\widetilde M$, $\phi_{\widetilde D}(m) \in \widetilde M$. Since $\widetilde R$ is constructed as the preimage of $R\subseteq A$ under the surjection $\widetilde A \twoheadrightarrow A$, and $D$ is a free module of rank $n$ over $W(k')\otimes_{\mathbb Z_p} A$, $\widetilde M$ is free of rank $n$ by its construction.
\end{proof}
 
We will also denote $\phi_{\widetilde D}|_{\widetilde M}$ by $\phi_{\widetilde M}$, or simply $\phi$. Now $\widetilde M$ can be seen as a submodule of $\AC \otimes_{W(k')}\widetilde D$.

Let $\beta$ be an eigenbasis of $\MM$. Then $\bar \beta:=\beta \mathrm{\ mod\ }u'$ is a $W(k')\otimes_{\mathbb Z_p} R$-basis of $M$. Choose any lift of $\bar \beta$ to a $W(k')\otimes_{\mathbb Z_p}\widetilde R$-basis of $\widetilde M$, denoted by $\tilde{\bar \beta}$. 

Suppose that the inclusion $\varphi^* \MM \hookrightarrow \mathcal S_R[1/p]\otimes M$ in \autoref{Relation} maps $\bar \beta$ to $\bar \beta \Xi$, where $\Xi \in \GL_n(\mathcal S_R[1/p])$ (also see \thref{algstack}). 

Now we first consider any lift of $\Xi \in \GL_n(\mathcal S_R[1/p])$ to $\tilde \Xi \in \GL_n(\mathcal S_{\widetilde R}[1/p])$ satisfying $\tilde \Xi-1 \in u' \Ma(\mathcal S_{\widetilde R}[1/p])$. We will give properties for such a lift before fixing a particular lift $\tilde \Xi$ that we are interested in. Let $\wM$ be the sub-$\mathfrak S_{L',\widetilde R}$-module of $\mathbf A_{\mathrm{cris}} \otimes_{W(k')} \widetilde D$ generated by $\tilde\beta:=\tilde{\bar\beta} \tilde \Xi$. It is, for the time being, only an $\mathfrak S_{L', \widetilde R}$-module. We have the following properties.

\begin{lemma} \thlabel{0}
    The $\mathfrak S_{L',\widetilde R}$-module $\wM$ is free of rank $n$. It satisfies $\wM/u'\wM\cong\widetilde M$ and $\wM/\mathfrak I \wM \cong \MM$.
\end{lemma}
\begin{proof}
    We know that $\wM$ is generated by $\tilde \beta$ as an $\mathfrak S_{L',\widetilde R}$-module. To show that $\wM$ is free of rank $n$, it suffices to show independence of the vectors in $\tilde \beta$. Since $\tilde\beta=\tilde{\bar\beta} \tilde \Xi$ and $\tilde  \Xi \in \GL_n(\mathcal S_{\widetilde R}[1/p])$, it suffices to show that $\tilde{\bar\beta}$ is independent over the ring $\mathfrak S_{L',\widetilde R}$. Let $\tilde{\bar\beta}=(m_1, \cdots,  m_n)$, where each $m_i$ is in $\widetilde M$. Suppose that $a_1 m_1+\cdots+a_n m_n=0$ for some $a_i \in \mathfrak S_{L',\widetilde R}$, then after reducing modulo $u'$, we have $\widehat{a_1}m_1+\cdots +\widehat{a_n} m_n=0$, where $\widehat{a_i}=a_i \mathrm{\ mod \ }u'$. Since $\tilde{\bar\beta}$ is a basis of $\widetilde M$ over $W(k')\otimes_{\mathbb Z_p}\widetilde R$, $\widehat{a_i}=0$ for all $i$. Hence $u'$ divides $a_i$ for all $i$. Let $b_i=a_i/u'$, then we still have $b_1 m_1 +\cdots +b_n  m_n=0$. Similarly, we see that $u'$ divides every $b_i$. Repeating this process, we can see that $a_i$ is divisible by any power of $u'$. So $a_i=0$ for all $i$.

    To show that $\wM/u'\wM=\widetilde M$, we first construct a map $\wM \to \widetilde M$ as follows. We can write $\tilde \Xi=1+u'X$ with $X \in \Ma(\mathcal S_{\widetilde R}[1/p])$. Write $X_i\in (\mathcal S_{\widetilde R}[1/p])^n$ for the $i$-th column of the matrix $X$. Then the basis $\tilde \beta$ of $\wM$ has vectors $m_1+u'\tilde{\bar \beta}X_1, m_2+u'\tilde{\bar \beta}X_2, \cdots, m_n+u'\tilde{\bar \beta}X_n$. We send $\sum_{i=1}^n s_i(m_i+u'\tilde{\bar \beta}X_i)$ to $\sum_{i=1}^n \bar{s_i} m_i \in \widetilde M$, where $s_i \in \mathfrak S_{L',\widetilde R}$ and $\bar{s_i}=s_i { \mod\ }u' \in W(k')\otimes_{\mathbb Z_p}\widetilde R$. Since $(m_1,\cdots, m_n)$ is a $W(k')\otimes_{\mathbb Z_p}\widetilde R$-basis of $\widetilde M$, $\sum_{i=1}^n \bar{s_i} m_i=0$ is equivalent to $s_i \in u'\mathfrak S_{L',\widetilde R}$ for all $i$. Since $\wM$ is freely generated by $\tilde \beta$, this shows that the kernel of the map $\wM \to \widetilde M$ is $u'\wM$. Thus the second claim is proved.

    Finally, we show that $\wM/\mathfrak I\wM=\MM$. Let $\tilde \beta=(\tilde m_1,\cdots,\tilde m_n)$. Let $\beta=(m_1,\cdots,m_n)$. Then $\beta=\tilde \beta \mathrm{\ mod\ }\mathfrak I$. Consider the map $\wM \to \MM$ defined by $\sum_{i=1}^n a_i \tilde m_i \mapsto \sum_{i=1}^n \bar a_i m_i$, where $a_i \in \mathfrak S_{L', \widetilde R}$ for all $i$ and $\bar a_i= a_i \mathrm{\ mod \ } \mathfrak I \in \SSSS$.  Since $\beta$ is an $\SSSS$-basis of $\MM$, the kernel of the map contains elements of the form $\sum_{i=1}^n a_i \tilde m_i$ with $ a_i \in \mathfrak I \mathfrak S_{L', \widetilde R}$. Since $\wM$ is an $\mathfrak S_{L', \widetilde R}$-module freely generated by $\tilde \beta$, the kernel is exactly $\mathfrak I\wM$. Hence we obtain the desired isomorphism.
\end{proof}

We now have the following commutative diagram: 
$$\begin{tikzcd} 
    \label{1}
	\wM & \mathcal S[1/p] \otimes_{W(k')}\widetilde M & \AC \otimes_{W(k')} \widetilde D \\
	\MM & \mathcal S[1/p] \otimes_{W(k')}M & \mathbf A_{\mathrm{cris}} \otimes_{W(k')} D.
	\arrow[twoheadrightarrow,"\mathfrak F",swap, from=1-1, to=2-1]
	\arrow[hook, from=1-1, to=1-2]
	\arrow[hook, from=2-1, to=2-2]
    \arrow[twoheadrightarrow,"\mathfrak F", from=1-2, to=2-2] 
    \arrow[hook, from=1-2, to=1-3]
    \arrow[hook, from=2-2, to=2-3]
    \arrow[twoheadrightarrow, "\mathfrak F", from=1-3, to=2-3]
\end{tikzcd}$$

By abuse of notation, we call all the column reduction maps $\mathfrak F$.

Let $\tilde c$ be the matrix of $\phi$ under basis $\tilde{\bar \beta}$ of $\widetilde M$. Then $\tilde c \in \Ma(W(k')\otimes_{\mathbb Z_p}\widetilde R)$ and $\tilde c \in \GL_n((W(k')\otimes_{\mathbb Z_p}\widetilde R)[1/p])$ since $\phi_{\wD}$ is invertible. Let $c$ be the matrix of $\phi$ under basis $\bar \beta$ of $M$. Since $\MM$ is stable under $\phi$, we have $ \Xi^{-1}  c \varphi(  \Xi )\in \Ma(\varphi(\SSSS))$.

Now we prove that the basis $\varphi^*\tilde \beta$ can be adjusted in a way such that the corresponding $\wM$ is stable under $\phi$. We also prove that such a module $\varphi^*\wM$ is unique. This gives the lift of $\varphi^*\MM$ we really need. We will be verifying its other properties in the final subsection.

\begin{thm} \thlabel{good}
    There is a lift $\tilde  \Xi \in \GL_n(\mathcal S_{\widetilde R}[1/p])$ of $\Xi \in \GL_n(\mathcal S_R[1/p])$ satisfying $\tilde  \Xi^{-1} \tilde c \varphi(\tilde  \Xi) \in \Ma(\varphi(\mathfrak S_{L',\widetilde R}))$. Furthermore, such a matrix $\tilde  \Xi$ is unique up to right multiplication by $1+\Ma(\mathfrak I \otimes \varphi(\mathfrak S))$. Or in other words, the corresponding $\varphi^*\wM$ is stable under $\phi$ when regarded as a submodule of $\mathcal S[1/p] \otimes_{W(k')} \widetilde M$. And such a $\phi$-stable lift of the module $\varphi^*\MM$ inside $\mathcal S[1/p] \otimes_{W(k')} \widetilde M$ is unique.
\end{thm} 

\begin{proof}

    We first pick any lift $\tilde  \Xi_0 \in \GL_n(\mathcal S_{\widetilde R}[1/p])$ of $\Xi$ satisfying $\tilde \Xi_0 -1 \in u' \Ma(\mathcal S_{\widetilde R}[1/p])$. Let $\tilde  \Xi=\tilde  \Xi_0(1+X)$ for some $X \in u'\Ma(\mathfrak I \otimes \mathcal S[1/p])$. We will pick a suitable $X$ such that $\tilde  \Xi^{-1} \tilde c \varphi(\tilde  \Xi)\in \Ma(\varphi(\mathfrak S_{L',\widetilde R}))$.

    Let $m_0=\tilde  \Xi_0^{-1} \tilde c \varphi(\tilde  \Xi_0) \in \Ma(\mathcal S_{\widetilde R}[1/p])$. Since $\tilde c \in \GL_n((W(k')\otimes_{\mathbb Z_p}\widetilde R)[1/p])$, $m_0$ is invertible and lies in $\GL_n(\mathcal S_{\widetilde R}[1/p])$. Let $m$ be any lift of $\Xi^{-1}  c \varphi(\Xi)\in \Ma(\varphi(\SSSS))$ to $\Ma(\varphi(\mathfrak S_{L',\widetilde R}))$ satisfying $m \equiv m_0 \mathrm{\ mod\ }u'$. This is always possible since one can first arbitrarily lift $\Xi^{-1}  c \varphi(\Xi)\in \Ma(\varphi(\SSSS))$ and then adjust its constant term ($u'=0$ specialization) to be the same as that of $m_0$. Since $m_0$ and $m$ reduce to the same matrix $\Xi^{-1}  c \varphi(\Xi)\in \Ma(\varphi(\SSSS))$, $m_0-m \in u'\Ma(\mathfrak I \otimes \mathcal S[1/p])$. Let $t$ be a positive even integer such that $p^{t/2} m_0, p^{t/2} m_0^{-1} \in \Ma(\mathcal S_{\widetilde R})$. Let $Y=p^t (m-m_0) \in u'\Ma(\mathfrak I \otimes \mathcal S)$.

    Using the fact that $X^2=0$, we calculate that $$\tilde \Xi^{-1} \tilde c \varphi(\tilde \Xi)=(1-X)\tilde \Xi_0^{-1}\tilde c \varphi(\tilde \Xi_0)(1+\varphi(X))=m_0-Xm_0+m_0\varphi(X).$$

    We expect that $-Xm_0+m_0\varphi(X)= m-m_0 $. If we can find an $X$ satisfying this equation, then $\tilde \Xi^{-1} \tilde c \varphi(\tilde \Xi) \in \Ma(\varphi(\mathfrak S_{L',\widetilde R}))$ since we already have $m \in \Ma(\varphi(\mathfrak S_{L',\widetilde R}))$. Hence we only need to show that the map $$\mathcal F:u'\Ma(\mathfrak I \otimes  \mathcal S[1/p]) \ni X \mapsto   Xm_0-m_0 \varphi(X) \in u'\Ma(\mathfrak I \otimes \mathcal S[1/p])$$ contains $m-m_0$. By linearity, this is equivalent to showing that it contains the element $Y\in u'\Ma(\mathfrak I \otimes \mathcal S)$.

   To do this, we first use \thref{S} to write $Y$ as $\sum_{i=1}^\infty y_i \frac{u'^i}{i!}$ with $y_i \in \Ma(W(k') \otimes_{\mathbb Z_p} \widetilde R)$ and $y_i \to 0$ $p$-adically as $i \to \infty$, where $\nu$ denotes the least $p$-adic order of all entries of a matrix. 

  Let $\mathcal N:u'\Ma(\mathfrak I \otimes \mathcal S[1/p]) \ni X \mapsto m_0\varphi(X) \varphi(m_0)^{-1} \in u'\Ma(\mathfrak I \otimes \mathcal S[1/p])$. Let $W=\sum_{i=0}^\infty \mathcal N^i(Y)$. Suppose we know that this series converges, then $\mathcal F(Wm_0^{-1})=W-\mathcal N(W)=Y$, proving that $Y$ is in the image of $\mathcal F$. 
  
  Now we will show the convergence of $W$. Since $\varphi(u')=u'^p$, we obtain by iterating that $$\mathcal \varphi^i(Y)=\sum_{j=1}^\infty y_{ij} \frac{u'^{p^ij}}{j!}$$ for some $y_{ij} \in \Ma(W(k') \otimes_{\mathbb Z_p} \widetilde R)$ satisfying $y_{ij} \to 0$ $p$-adically as $j \to \infty$. Recall that $p^tm_0^{-1} \in \Ma(\mathcal S_{\widetilde R})$. Each entry of the matrix $\mathcal N^i(Y)$ has the form $$\sum_{j=1}^\infty a_{ij} \frac{u'^{p^ij}}{p^{ti}j!},$$ where $a_{ij} \in \mathcal S_{\widetilde R}$ and $a_{ij} \to 0$ as $j \to \infty$. We can apply \thref{S} and write $a_{ij}=\sum_{k=0}^\infty b_{ij\kappa }\frac{u'^\kappa }{\kappa !}$ where $b_{ij\kappa } \in W(k')\otimes _{\mathbb Z_p} \widetilde R$ satisfies $b_{ij\kappa } \to 0$ $p$-adically as $\kappa  \to \infty$ for each pair $(i,j)$. Therefore, $$W=\sum_{i=0}^\infty \sum_{j=1}^\infty \sum_{\kappa =0} ^\infty \frac{u'^{p^ij}}{p^{\kappa i}j!} b_{ij\kappa }\frac{u'^\kappa }{\kappa !}=\sum_{l=1}^\infty \big(\sum_{(i,j):p^ij \le l} \frac{b_{i,j,l-p^ij}l!}{p^{ti}j!(l-p^ij)!} \big) \frac{u'^l}{l!}.$$ Our previous assumption on $Y$ having no constant term forces the indices $j$ and $l$ to take positive integer values, avoiding potential divergence issues of the inner summation for $l=0$. We separate the estimation of $p$-adic order $\nu$ of the coefficients of the series $W$ into \thref{v-adic}, from which we can see the convergence of $\sum_{i=0}^\infty \mathcal N^i(Y)$. Thus the existence is proved.

    To prove the uniqueness statement, suppose that $\tilde \Xi_0 \in \GL_n(\mathcal S_{\widetilde R}[1/p])$ is a lift of $\Xi \in \GL_n(\mathcal S_R[1/p])$ satisfying $m_0:=\tilde \Xi_0^{-1} \tilde c \varphi(\tilde \Xi_0) \in \Ma(\varphi(\mathfrak S_{L',\widetilde R}))$. We have another lift $\tilde \Xi=\tilde \Xi_0(1+X)$ with $X \in u'\Ma(\mathfrak I \otimes \mathcal S[1/p])$ also satisfying $\tilde \Xi^{-1} \tilde c \varphi(\tilde \Xi) \in \Ma(\varphi(\mathfrak S_{L',\widetilde R}))$.
     From $$\tilde \Xi^{-1} \tilde c \varphi(\tilde \Xi)=(1-X)\tilde \Xi_0^{-1}\tilde c \varphi(\tilde \Xi_0)(1+\varphi(X))=m_0-Xm_0+m_0\varphi(X),$$ we know that $$Xm_0-m_0\varphi(X)\in \Ma(\mathfrak I \otimes \varphi(\mathfrak S)).$$ We want to show that $X \in u'\Ma(\mathfrak I \otimes \varphi(\mathfrak S))$. This is equivalent to showing that the map $$\mathcal F:\frac{u'\Ma(\mathfrak I \otimes \mathcal S[1/p])}{u'\Ma(\mathfrak I \otimes \varphi(\mathfrak S))}\ni X \mapsto   Xm_0-m_0 \varphi(X) \in \frac{u'\Ma(\mathfrak I \otimes \mathcal S[1/p])}{u'\Ma(\mathfrak I \otimes \varphi(\mathfrak S))}$$ is injective, which is proved in \thref{FFF}. Therefore, the uniqueness is also proved.
\end{proof}

\begin{lemma} \thlabel{v-adic}
In the setting of \thref{good}, the infinite series $$W=\sum_{l=1}^\infty \big(\sum_{(i,j):p^ij \le l} \frac{b_{i,j,l-p^ij}l!}{p^{ti}j!(l-p^ij)!} \big) \frac{u'^l}{l!}$$ converges $p$-adically.
\end{lemma}
\begin{proof}
    Since the convergence of $W=\sum_{i=0}^\infty \mathcal N^i(Y)$ doesn't depend on the terms labeled by $i=0,1,2$, we assume from now on that $i \ge 3$. We have the inequalities $$\nu\big(\frac{b_{i,j,l-p^ij}l!}{p^{ti}j!(l-p^ij)!} \big) = \nu(b_{i,j,l-p^ij})+ \nu\big(\frac{l!}{j!(l-p^ij)!}\big) - ti \ge \nu(b_{i,j,l-p^ij})+ p^{i-2}j-ti.$$ Here, we have used $\nu\big(\frac{l!}{j!(l-p^ij)!}\big)\ge \nu(\frac{(p^i j)!}{j!}) \ge \nu(((p^{i}-1)j)!)\ge \left\lfloor (p^i-1)j/p \right\rfloor \ge (\frac{p^i-1}{p} -1 )j \ge p^{i-2}j$ for $i \ge 3$, where $\left\lfloor \cdot \right \rfloor$ denotes the floor function. Let $S$ be any positive integer. We will find $l_0$ such that for any $l>l_0$, and any pair of positive integers $(i,j)$ with $p^i j \le l$ and $i>2$, we have $\nu\big(\frac{b_{i,j,l-p^ij}l!}{p^{ti}j!(l-p^ij)!} \big)  \ge S$. First, we pick a positive integer $i_0$ such that $p^{i-2}-ti$, when regarded as a sequence labeled by $i$, is increasing for $i\ge i_0$ with $p^{i_0-2}-ti_0 \ge S$. Then, we pick $j_0$ such that $p^{-2}j_0-ti_0 \ge S$. Finally, since $b_{ij\kappa } \to 0$ as $\kappa  \to \infty$ for any fixed pair $(i,j)$, we can take $\kappa _0$ to be a positive integer satisfying $\nu(b_{ij\kappa })>S-T$ for all $\kappa \ge \kappa _0$ and $i \le i_0, j \le j_0$, where the constant $T$ denotes the minimum of the function $f(x)=p^{x-2}-tx$ in the domain $[0, \infty)$. Note that $T<0$ since $f(1)<1-t \le 0.$
  
  Now, let $l_0=p^{i_0}j_0+\kappa _0$. Let $l$ be any integer larger than $l_0$.  For all pairs $(i,j)$ with $i \ge i_0$, we have $$\nu\big(\frac{b_{i,j,l-p^ij}l!}{p^{ti}j!(l-p^ij)!} \big) \ge p^{i-2}j-ti \ge p^{i_0 - 2}-ti_0 \ge S.$$ For all pairs $(i,j)$ with $i<i_0$ and $j \ge j_0$, we have $$\nu\big(\frac{b_{i,j,l-p^ij}l!}{p^{ti}j!(l-p^ij)!} \big) \ge p^{i-2}j-ti \ge p^{ - 2}j_0-ti_0 \ge S$$  For all pairs $(i,j)$ with $i<i_0$ and $j<j_0$, we have $l-p^ij\ge l_0-p^{i_0}j_0=\kappa _0$. Hence $\nu(b_{i,j,l-p^ij}) \ge S-T$. Therefore, $$\nu\big(\frac{b_{i,j,l-p^ij}l!}{p^{ti}j!(l-p^ij)!} \big) \ge \nu(b_{i,j,l-p^ij})+T>S-T+T=S.$$ This shows that the coefficient of $u'^l /l !$ in $W$ tends to $0$ $p$-adically as $l \to \infty$.
\end{proof}

\begin{lemma} \thlabel{FFF}
    In the setting of \thref{good}, the map $$\mathcal F:\frac{u'\Ma(\mathfrak I \otimes \mathcal S[1/p])}{u'\Ma(\mathfrak I \otimes \varphi(\mathfrak S))}\ni X \mapsto   Xm_0-m_0 \varphi(X) \in \frac{u'\Ma(\mathfrak I \otimes \mathcal S[1/p])}{u'\Ma(\mathfrak I \otimes \varphi(\mathfrak S))}$$ is injective.
\end{lemma}

\begin{proof}
    Let $$\mathcal N:\frac{u'\Ma(\mathfrak I \otimes \mathcal S[1/p])}{u'\Ma(\mathfrak I \otimes \varphi(\mathfrak S))}\ni X \mapsto m_0\varphi(X) m_0^{-1} \in\frac{u'\Ma(\mathfrak I \otimes \mathcal S[1/p])}{u'\Ma(\mathfrak I \otimes \varphi(\mathfrak S))}. $$ We first show that this map is well-defined. Since $m_0 \in \Ma(\varphi(\mathfrak S_{L', \widetilde R}))$ is the Frobenius matrix of a Breuil-Kisin module, $m_0^{-1} \in \Ma(\varphi(\mathfrak S_{L', \widetilde R}[1/\varphi(E)]))$. Since $$\frac{1}{\varphi(E)}=\sum_{i \ge 0} (-1)^ip^{-1-i}(u')^{pe'i},$$ and $p$ is invertible in $\mathfrak I$ by the construction of $I$ and $\mathfrak I$, conjugation by $m_0$ preserves the lattice $u'\Ma(\mathfrak I \otimes \varphi(\mathfrak S))$.
    
    Since $\varphi(u')=u'^p$, the map $\mathcal N$ is topologically nilpotent by the $u'$-adic completeness. We have $\mathcal F =R_{m_0}(1-\mathcal N)$, where $R_{m}$ denotes the right multiplication by $m$. We know that the map $R_{m_0}(1-\mathcal N)$ has a left inverse $(\sum_{i=0}^\infty \mathcal N^i) R_{m_0^{-1}}$, which is well-defined since $\mathcal N$ is topologically nilpotent. Thus injectivity of $\mathcal F$ is proved as desired. 
\end{proof}

\thref{good} gives the basis $\varphi^*\tilde \beta:=\tilde{\bar \beta} \tilde \Xi$ of $\varphi^* \wM$. The map between bases $\tilde \beta \mapsto \varphi^* \tilde \beta$ gives a $\phi$-equivariant $\mathfrak S_{L',\widetilde R}$-semilinear inclusion $\wM \hookrightarrow  \varphi^* \wM$. From now on, we assume that $\wM$ is determined as such. In particular, we have a $\phi$-equivariant $\mathfrak S_{L',\widetilde R}$-semilinear inclusion $\wM \hookrightarrow  \mathcal S[1/p] \otimes_{W(k')} \widetilde M$. Let $ \widetilde \MMM=\widetilde \MM \otimes_{\mathfrak S_{L',\widetilde R}} \mathbf A_{\mathrm{inf},\widetilde R}$ as usual. Then we have a $\phi$-equivariant inclusion $\MMM \to \AC \otimes_{W(k')}\wD$.

\subsection{Smoothness: verification and applications} \label{5.4}

In this subsection, we verify that the module $\wM$ constructed above is really a potentially semistable Breuil-Kisin module of the desired weight and type. We conclude by applications to the formal smoothness properties of the stack and the corresponding deformation rings.

We define the map $\psi:\AC \otimes_{W(k')}\wD \to \AC \otimes_{W(k')}\wD$ as $$\psi(d)= \mathrm{exp \ }(t \otimes N_{\widetilde D})(d),$$ for any $d\in \widetilde D$, and extend semilinearly to $\AC \otimes_{W(k')}\wD$.

\begin{prop} \thlabel{psi}
    The $\mathbf A_{\mathrm{inf},\widetilde R}$-module $\widetilde \MMM$ inside $\mathbf A_{\mathrm{cris}} \otimes_{W(k')} \widetilde D$ is stable under $\psi$. 
\end{prop} 
\begin{proof}
  Consider the $\mathbf A_{\mathrm{inf},\widetilde R}$-module $\psi(\varphi^*\widetilde \MMM)$ inside $\mathbf A_{\mathrm{cris}} \otimes_{W(k')} \widetilde D$. After reducing mod $\mathfrak I$, it becomes $\psi(\varphi^*\MMM)=\varphi^*\MMM \subseteq \mathbf A_{\mathrm{cris}} \otimes_{W(k')} D$. From \thref{Nphi}, we see that $\phi$ and $\psi$ commute. So $\psi(\varphi^*\widetilde \MMM)$ is $\phi$-stable. From the uniqueness part of \thref{good}, we get $\psi(\varphi^*\widetilde \MMM)=\varphi^*\widetilde \MMM$, which gives the result.
\end{proof}

Note that $\AC \otimes_{W(k')}\wD$ and $\AC \otimes_{W(k')}D$ both have a $\Delta$-action which commutes with $\phi$ and the reduction map $\mathfrak F$. We know that $\MM \subseteq \AC \otimes_{W(k')}D$ is stable under the $\Delta$-action. Since $\wM$ is the preimage of $\MM$ under $\mathfrak F$, $\wM \subseteq \AC \otimes_{W(k')}\wD$ is also stable under the $\Delta$-action. This gives $\wM$ its $\Delta$-action.

As usual, we denote $\psi|_{\widetilde \MMM}$ by $\psi_{\widetilde \MMM}$ and its matrix under basis $\tilde \beta$ as $\Psi_{\widetilde \MM}$. To verify that $\widetilde \MMM \in \X(\widetilde A)$, we need the next lemma.

\begin{lemma} \thlabel{3}
  The matrix $\Psi_{\widetilde \MM}$ satisfies $\Psi_{\widetilde \MM} \in \Ma(\mathbf A_{\mathrm{inf},\widetilde R})$ and $\Psi_{\widetilde \MM} -\mathrm{id}  \in u' \Ma(\mathbf A_{\mathrm{inf},\widetilde R})$. 
\end{lemma}

\begin{proof}
    The first statement is equivalent to \thref{psi}, hence holds true.

   We use $\Psi_\MM$ (resp. $\Psi_{\widetilde \MM}$) to denote the matrix of Galois of $\MM$ (resp. $\widetilde \MM$) under the basis $\beta$ (resp. $\widetilde \beta$). From \thref{vip}, the relation between the Galois action and the monodromy operator on $\AC \otimes_{W(k')} D$ (resp. $\AC \otimes_{W(k')} \widetilde D$) is $\varphi(\Psi_{\MM})=\Xi^{-1} \mathrm{exp}(t\otimes N)\psi(\Xi)$ (resp. $\varphi(\Psi_{\widetilde \MM}) = \widetilde\Xi^{-1} \mathrm{exp}(t \otimes \widetilde N) \psi(\widetilde \Xi)$) where $N$ (resp. $\widetilde N$) is the matrix of $N_D$ (resp. $N_{\widetilde D}$) under basis $\beta$ (resp. $\tilde \beta$). Since $\tilde \beta$ is a lift of $\beta$, $\widetilde \Xi \mathrm{\ mod\ }I=\Xi$.
   
   In each term of the power series expansion of the exponential function, $\widetilde N \in \Ma(W(k') \otimes_{\mathbb Z_p} \widetilde R)$ reduces to $ N \in \Ma(W(k') \otimes_{\mathbb Z_p} R)$, hence $\Psi_{\widetilde \MM}  \in \GL_n(\mathbf A_{\mathrm{inf},\widetilde R})$ reduces to $\Psi_{\MM} \in \GL_n(\mathbf A_{\mathrm{inf}, R})$. Since $\Psi_{\MM} \equiv 1 \mathrm{\ mod\ }u'$, we can write  $\Psi_\MM=1+u'X$, with $X \in \Ma(\AAA)$. Write $\Psi_{\wM}=1+u'\tilde X$, then $\tilde X$ reduces to $X \in \Ma(\AAA)$. In particular, $\tilde X \in \Ma(\Aw)$. Since $\Psi_{\MMM} \in \GL_n(\mathbf A_{\mathrm{inf}, \widetilde R})$, $\tilde X \in \Ma(\frac{1}{u'}\mathbf A_{\mathrm{inf},\widetilde R})$. Hence $\tilde X \in \Ma(\frac{1}{u'}\mathbf A_{\mathrm{inf},\widetilde R} \cap \Aw)$. From \thref{intersection}, we can see that $\tilde X \in \Ma(\mathbf A_{\mathrm{inf},\widetilde R})$.
\end{proof}

Let $r_1, \cdots, r_\nu$ be the finitely many topological generators of $R$ over $\mathcal O$. Let $\tilde r_i$ be a lift of $r_i$ to $\widetilde R$ for each $i$. Let $\mathfrak R$ be the sub-$\mathcal O$-algebra of $\widetilde R$ generated by $\tilde r_1, \cdots, \tilde r_\nu$. Note that the matrices of $\phi$ and $N_{\widetilde D}$ are both polynomials with $\widetilde R$-coefficients. These finitely many coefficients will generate another sub-$\mathcal O$-algebra $\widetilde R_+$ of $\widetilde R$. From the formulation of a potentially semistable Breuil-Kisin module, we can see that $\wM$ can be defined on the subalgebra $\widetilde R_+$ of $\widetilde A$. 

Let $\mathrm{Spf\ } \widetilde R_0$ be any formal model of $\mathrm{Sp\ } \widetilde A$ such that $\widetilde R_0$ contains $\widetilde R_+$ and its image under the reduction $\widetilde R \twoheadrightarrow R$ contains $R$. Such a model $\mathrm{Spf\ } \widetilde R_0$ exists. For example, if we take $\widetilde R_0$ to be the $p$-adic completion of the subalgebra of $\widetilde R$ generated by $\widetilde R_+$ and $\mathfrak R$, then we can see that it gives a formal model. For such an $\widetilde R_0$,  we will prove that $\wM \in \XXX(\widetilde R_0)$. We need the following lemma. 

\begin{lemma} \thlabel{crys}
  Let $\MM \in \mathcal X_1^{[0,h],\tau}(R)$ corresponding to a Galois character $G_K \to R^\times$. Then the $\phi$-map of $\MM$ is multiplication by $\eta E(u')^r$, where $r \in [0,h], \eta \in \SSSS^\times$.
\end{lemma}

\begin{proof}
As usual, suppose that $\phi$ is scalar multiplication by $C(u') \in \SSSS$. Since $\MM \in \mathcal X_1^{[0,h],\tau}(R)$, $\lambda(u') C(u')=E(u')^h$ for some $\lambda(u') \in \SSSS$. Assume, without loss of generality, that $\MM$ has exact Hodge-Tate weight $h$. Then the corresponding filtered $(\varphi,N)$-module $(\MM/u')[1/p]$ has its $\phi$-map being multiplication by $\kappa p^h$ for some $\kappa \in R^\times$. Setting $u'=0$ in $\lambda(u') C(u')=E(u')^h$, we have $\lambda(0) \kappa p^h=p^h$. Hence $\lambda(0)$ is a unit. So $\lambda(u')$ is also a unit in $\SSSS$. Hence $C(u')=\lambda(u')^{-1}E(u')^h$ as desired.
\end{proof}

Now we prove our main theorem.

\begin{thm} \thlabel{2}
    The Breuil-Kisin module $\wM$ satisfies $\wM \in \XXX(\widetilde R_0)$.
\end{thm}
\begin{proof}
    We have proved \thref{0}, \thref{good}, \thref{psi}, \thref{3}. The verification of the correct inertial type is straightforward. Since $\widetilde \Lambda \in \WD^\tau(\widetilde A)$, and the corresponding $\widetilde D\in L(\phi,N,\Delta)_{\widetilde A}$ is given by $(\wM/u')[1/p]$, $\wM$ is of inertial type $\tau$ according to \thref{BKK1}. We can assume that $\tilde \beta^*$ is an eigenbasis of $\wM$ (see \thref{eigenbasis} and \thref{matrix}), and the matrices of partial Frobenius $(\tilde A^{(j)})_{j \in \mathcal J}$ satisfy $\tilde A^{(j)} \in \Ma(\widetilde R\llbracket v \rrbracket)$ for all $j \in \mathcal J$.

   Now let's verify the remaining conditions from \thref{algstack}. The relation between $\varphi$ and $N$ holds obviously since $\wM/u'\cong \widetilde M$. So it remains to show that $(v+p)^h (\tilde A^{(j)})^{-1} \in \Ma(\widetilde R_0\llbracket v \rrbracket)$ for all $j \in \mathcal J$. To verify this, we turn to the associated matrices $(C^{(j)})_{j \in \mathcal J}$ of $(A^{(j)})$ (see \thref{Frob}) (resp. $(\widetilde C^{(j)})$ of $(\widetilde A^{(j)})$) and set $C=\mathfrak C((C^{(j)})_{j \in \mathcal J})$ (resp. $\widetilde C=\mathfrak C((\widetilde C^{(j)})_{j \in \mathcal J})$). 

     Let $\bigwedge^n \MM$ be the top wedge product of $\MM$, the matrix of Frobenius on $\bigwedge ^n \MM$ is the scalar $\mathrm{det}(C)$, i.e. the determinant of the matrix of Frobenius on $\MM$. We know from \thref{crys} that $\mathrm{det}(C)=\lambda E(u')^r$ for some non-negative integer $r \le nh$ and $\lambda \in \SSSS^\times$. Since $\tilde C$ maps to $C$ under the reduction-modulo-$\mathfrak I$ map, $\mathrm{det}(\tilde C)=\tilde \lambda E(u')^r+x$ for some $\tilde \lambda \in \mathfrak S_{L',\widetilde R}^\times$ and $x \in \mathfrak I \mathfrak S_{L',\widetilde R}$. Here we have used the fact that the unit $\lambda$ always lifts to a unit $\tilde \lambda$ if $\mathfrak I$ is square zero. Hence $\mathrm{det}(\tilde C)(\tilde \lambda^{-1} E(u')^r-\tilde \lambda^{-2}x)=E(u')^{2r}$. Since $\tilde \lambda^{-1} E(u')^r-\tilde \lambda^{-2}x\in  \mathfrak S_{L',\widetilde R}$, $E(u')^{2r} \tilde C^{-1} \in \Ma( \mathfrak S_{L',\widetilde R})$. We have $\wM \in \mathcal X_{n}^{[0,2nh],\tau}(\widetilde R_0)$. 
     
     Here, we use our technique of reducing to the finite flat case. We first do the case where $\widetilde R_0$ is a finite flat $\mathcal O$-algebra. In this case, $\wM$ comes from a lattice of a Galois representation $\tilde \rho: G_K \to \GL_n(\widetilde R_0[1/p])$ by \thref{equivalence}. Suppose that $\MM$ corresponds to the Galois representation $\rho:G_K \to \GL_n(R[1/p])$. We have $(\mathbf B_{\mathrm{dR}} \otimes_{\mathbb Q_p} \rho)^{G_K}\cong \big(\mathbf B_{\mathrm{dR}}\otimes_{\mathbb Q_p} \tilde \rho \otimes_{\widetilde R_0[\frac{1}{p}]} R[\frac{1}{p}]\big)^{G_K} \cong (\mathbf B_{\mathrm{dR}}\otimes_{\mathbb Q_p} \tilde \rho)^{G_K}\otimes_{\mathbb Q_p} R[\frac{1}{p}]\cong \big(K\otimes_{\mathbb Q_p} R[\frac{1}{p}]\big)\otimes_{K\otimes_{\mathbb Q_p} \widetilde R_0[\frac{1}{p}]} (\mathbf B_{\mathrm{dR}}\otimes \tilde \rho)^{G_K}$. The same isomorphisms apply when we replace $\mathbf B_{\mathrm{dR}}$ with its filtrations. This shows that the Hodge-Tate weights of $\tilde \rho$ and $\rho$, regarded as jumps between filtrations, are the same. Therefore, $\wM \in \XXX(\widetilde R_0)$.

    In general, let $\widetilde R$ be any finite flat $\mathcal O$-algebra, with map $\widetilde R_0 \to \widetilde R$. Let $R=\widetilde R/I\widetilde R$. We can show that  $\wM_{\widetilde R} \in \XXX(\widetilde R)$ from our reasoning in the finite flat case. Since we already have $\wM \in \mathcal X_{n}^{[0,2nh],\tau}(\widetilde R_0)$, the partial Frobenii satisfy $(v+p)^{2nh}(\tilde A^{(j)})^{-1} \in \Ma(\widetilde R_0 \llbracket v \rrbracket)$ for all $j \in \mathcal J$. Consider each entry of $W(v)=(v+p)^{2nh}(\tilde A^{(j)})^{-1} \in \Ma(\widetilde R_0 \llbracket v \rrbracket)$. The desired condition that $(v+p)^{h}(\tilde A^{(j)})^{-1} \in \Ma(\widetilde R_0 \llbracket v \rrbracket)$ is equivalent to $$\big(\frac{d}{dv}\big)^i|_{v=-p}W(v)=0$$ for all $i=0,1, \cdots, (2n-1)h-1$. All entries of the matrices $\big(\frac{d}{dv}\big)^i|_{v=-p}W(v)$ for all $j \in \mathcal J$  will generate an ideal $I \subseteq \widetilde R_0$. And our finite flat case shows that any finite flat specialization $\widetilde R$ of $\widetilde R_0$ will give $I \widetilde R=0$. Using \thref{aux1}, we know that $I=0$, proving that in the general case, we still have $(v+p)^{h}(\tilde A^{(j)})^{-1} \in \Ma(\widetilde R_0 \llbracket v \rrbracket)$. Hence $\wM \in \XXX(\widetilde R_0)$. 
\end{proof}

We have proved the following theorem:

\begin{thm} \thlabel{smooth}
The map $WD:\Xe \to (\WD^\tau)^{\mathrm {an}}$ is smooth.
\end{thm} 

\begin{proof}
    According to \thref{0} and \thref{2} we have constructed $\wM\in \XXX(\widetilde R_0)$ that maps to $\widetilde \Lambda$ under $WD$ and lifts $\MM \in\XXX(R)$. This shows the formal smoothness of the map $WD$.

 Since $\XXX$ is topologically of finite type over $\mathcal O$, $\Xe$ is also of finite type over $E$. The same holds for $\WD^\tau$ and its analytification. (One can easily see this through its presentation by \thref{wd} and \thref{weil-deligne}.) So $WD$ is smooth.
\end{proof}

\begin{thm} \thlabel{rigid}
 The rigid generic fiber $\Xe$ of the stack $\XXX$ has a dense formally smooth locus.
\end{thm}

\begin{proof}
  By \thref{smooth}, the preimage of any point that is in the formally smooth locus of $(\WD^\tau)^{\mathrm {an}}$ under the map $WD$ is also formally smooth in $\Xe$. By \thref{wdsmooth}, $\WD^\tau$ has a dense formally smooth locus. So does its analytification $(\WD^\tau)^{\mathrm {an}}$. Considering that $WD$ is open, the image of any open substack of $\Xe$ under $WD$ will intersect the dense formally smooth locus of $\WD^\tau$. Therefore, $\Xe$ also has a dense formally smooth locus.
\end{proof}

Like \cite[Theorem 3.3.4]{Kis}, we derive information regarding the potentially semistable deformation rings.

\begin{thm} \thlabel{pst def ring}
    Let $\bar \rho: G_K \to \GL_n(\mathbb F)$ be a mod $p$ representation. Then $\mathrm{Spec\ }R_{\bar \rho}^{[0,h],\tau}$ also has a dense open subscheme which is regular.
\end{thm}

\begin{proof}

We only need to prove the statement on the generic fiber $\mathrm{Spec\ }R_{\bar \rho}^{[0,h],\tau}[1/p]$.

If the ring $R_{\bar \rho}^{[0,h],\tau}$ is $0$, then we are done. Otherwise, pick any potentially semistable lift $\rho:G_K \to \GL_n(\mathcal O)$ of $\bar \rho$ of height bounded by $h$ and inertial type $\tau$. (We may enlarge the coefficient field $E$ if needed.) Suppose that the versal ring of $\XXX$ at the $\mathcal O$-point $\rho$ is $R$. Then $R=\widehat R_{\bar \rho}^{[0,h],\tau}$, where $\widehat \cdot$ denotes the completion at the kernel of the map $R_{\bar \rho}^{[0,h],\tau} \to \mathcal O$ corresponding to the map $\mathrm{Spf \ } \mathcal O \to \mathrm{Spf \ } R_{\bar \rho}^{[0,h],\tau}  \to \XXX$ given by $\rho$. We also know that the versal ring of the rigid generic fiber $\Xe$ at the point $\rho$ is the rigidification $R^{\mathrm {rig}}$, which is just $R[1/p]$ in our case when $R$ is topologically of finite type over $\mathcal O$. 

On the other hand, let $\pi:U \to \Xe$ be a smooth surjection. Let $u: \mathrm{Sp\ }E \to U$ be a point lying over $\rho\in \XXX(\mathcal O)$. By \thref{rigid}, $\Xe$ has a dense subset of formally smooth points, so does $\mathcal O_{U,u}$. The same also holds for its $u$-adic completion $(\mathcal O_{U,u})^{\hat u}$. Note that up to a formally smooth base change, $(\mathcal O_{U,u})^{\hat u}$ is isomorphic to the versal ring $R[1/p]$ of $\Xe$ at $\rho$. Hence $\mathrm{Spec\ }R[1/p]$ also has a dense subset of formally smooth points. Since $R=\widehat R_{\bar \rho}^{[0,h],\tau}$ and completion preserves regularity, we obtain our claim.
\end{proof}

\subsection{The potentially crystalline case}
We now explain how to arrive at the crystalline counterpart of the results above. First, let $\WD^0$ denote the substack of $\WD$ with $N=0$; similarly for the notation $\WD^{\tau,0}$. It is clear that the functor $WD$ will restrict to $WD^0:\Xec \to \WD^0$. We use $L(\phi,0,\Delta)_A$ to denote the full subcategory of $L(\phi,N,\Delta)_A$ of objects $D$ with $N=0$. The setting of the lifting argument is almost the same as in \autoref{5.3}, except we now have the stronger assumption that $\mathcal M \in \Xec(A)$ and that $\Lambda \in \WD^0(A)$. And this gives $D\in L(\phi,0,\Delta)_A$. We have the following observation.

\begin{prop} \thlabel{wd0}
    The stack $\WD^{\tau,0}$ is formally smooth.
\end{prop}
\begin{proof}
    From the proof of the first part of \cite[Proposition 3.1.2]{Kis}, we can see that $L(\phi,0,\Delta)$ is formally smooth. Using the equivalence of \cite[Proposition 4.1]{BS}, we know that the same holds for the stack $\WD^{\tau,0}$.
\end{proof}

This shows that we can form $\widetilde D \in L(\phi,0,\Delta)_{\widetilde A}$. The same construction gives $\wM \in \XXX(\widetilde R_0)$ for some formal model $\mathrm{Spf\ }\widetilde R_0$ of $\mathrm{Sp\ } \widetilde A$. Note that the construction starts with $\widetilde D$ with $N=0$, hence the resulting $\wM$ lies in the potentially crystalline substack $\XXXc$. In this way, we have shown that the functor $WD^0:\Xec \to \WD^0$ is formally smooth. As in \thref{rigid}, but now using \thref{wd0}, we obtain

\begin{thm} \thlabel{potcrys}
    The stack $\Xec$ is formally smooth.
\end{thm}

As in \cite[Theorem 3.3.8]{Kis}, we give the application to deformation rings.

\begin{thm}
        Let $\bar \rho: G_K \to \GL_n(\mathbb F)$ be a mod $p$ representation. Let $R_{\bar \rho}^{[0,h],\tau, \mathrm{cr}}$ be the corresponding potentially crystalline deformation ring. Then $R_{\bar \rho}^{[0,h],\tau, \mathrm{cr}}[1/p]$ is regular.
\end{thm}

\begin{proof}
The proof is completely analogous to \thref{pst def ring} after replacing the application of \thref{rigid} by the above \thref{potcrys}.
\end{proof}

\newpage

\bibliographystyle{amsalpha}
\bibliography{ref.bib}

\end{document}